\documentclass[12pt,reqno,a4paper]{amsart}
\usepackage{extsizes}
\usepackage[margin=0.5in]{geometry}

\calclayout
\usepackage{graphicx}
\usepackage[utf8]{inputenc}
\usepackage{nicefrac}
\usepackage{mathabx}
\usepackage{subcaption}
\usepackage{amsrefs}
\usepackage{esint}
\usepackage{mathtools}

\usepackage{wrapfig}
\usepackage{xcolor}
\usepackage{cleveref}
\usepackage{enumitem}
	
\usepackage{ulem}
\usepackage{autonum}

\newtheorem{theorem}{Theorem}[section]
\newtheorem{lemma}[theorem]{Lemma}
\newtheorem{proposition}[theorem]{Proposition} 
\newtheorem{corollary}[theorem]{Corollary} 

\theoremstyle{definition}
\newtheorem{definition}[theorem]{Definition}

\allowdisplaybreaks
\theoremstyle{remark}
\newtheorem{remark}[theorem]{Remark}

\numberwithin{equation}{section}

\DeclareMathOperator{\cn}{cn}

\DeclareMathOperator{\sn}{sn}
\DeclareMathOperator{\am}{am}

\allowdisplaybreaks

\newcommand{\supp}{\mathrm{supp}}

\usepackage{amsmath}
\usepackage{braket} 

\makeatletter
\@namedef{subjclassname@2020}{%
  \textup{2020} Mathematics Subject Classification}
\makeatother

\begin{document}

\title[Nongraphical obstacle problem for elastic curves]{A nongraphical obstacle problem for elastic curves}


\author{Marius Müller}
\address[M.~M\"uller]{Universität Augsburg, Institut für Mathematik, Universitätsstraße 14, 86159 Augsburg}
\email{marius1.mueller@uni-a.de}

\author{Kensuke Yoshizawa}
\address[K.~Yoshizawa]{Nagasaki University, Faculty of Education, 
1-14 Bunkyo-machi, Nagasaki, 852-8521}
\email{k-yoshizaw@nagasaki-u.ac.jp}


\date{\today}



\begin{abstract}
We study an obstacle problem for the length-penalized elastic bending energy for open planar curves pinned at the boundary. 
We first consider the case without length penalization and investigate the role of global minimizers among graph curves in our minimization problem for planar curves.
In addition, for large values of the length-penalization parameter $\lambda>0$, we expose an explicit threshold parameter above which minimizers touch the obstacle, regardless of its shape. On contrary, for small values of $\lambda>0$ we show that the minimizers do not touch the obstacle, and they are given by an explicit elastica. 
\end{abstract}

\maketitle

\section{Introduction}

In this article we study an obstacle problem for the (length-penalized) \textit{elastic bending energy} 
\begin{equation}
    \mathcal{E}_\lambda[\gamma] := \int_\gamma |\boldsymbol{\kappa}|^2 \; \mathrm{d}s + \lambda \int_\gamma 1 \; \mathrm{d}s = B[\gamma] + \lambda L[\gamma],
\end{equation}
where $\gamma \in W^{2,2}(0,1,\mathbf{R}^2)$ is an immersed curve with curvature vector $\boldsymbol{\kappa}= \partial_s^2 \gamma$ and $\lambda \geq 0$ is a penalization parameter.  
The critical points of the bending energy $B$ have attracted great interest for more than 250 years, and their 
history can go back to L.\ Euler \cite{Euler}. They are called \textit{(free) elasticae} and satisfy the Euler--Lagrange equation $2 \partial_s^2 \boldsymbol{\kappa} + |\boldsymbol{\kappa}|^2 \boldsymbol{\kappa}  = 0$. 
Recently, 
so-called \textit{obstacle problems} for the bending energy  have
been studied: e.g., minimization problems for $B$ (i) among graphs with unilateral constraint \cite{DD18, Mueller19, Miura21, Mueller21, Ysima, GO21, GrunauOkabe25}; (ii) among graphs with adhesive terms \cite{Miura16}; (iii) among closed curves confined in a given set \cite{DMN18}; and (iv) an obstacle problem for a generalized bending energy \cite{DMOY24}. 

In this paper, we consider the obstacle problem for the bending energy for \textit{open planar} curves. 
More precisely, for given \textit{obstacle function} $\psi : \mathbf{R} \rightarrow \mathbf{R}$ we consider the minimization problem  with the admissible set 
\begin{equation}
    A := \{ \gamma \in W^{2,2}_{\rm imm}(0,1; \mathbf{R}^2) \mid \gamma(0) = (0,0), \ \gamma(1) =(1,0), \ \gamma^2(x) \geq \psi(\gamma^1(x)) \  \textrm{for all $x \in [0,1]$} \}.  
\end{equation}
Here, $W^{2,2}_{\rm imm}(0,1; \mathbf{R}^2)$ denotes the set of immersed $W^{2,2}$-curves: 
\[
 W^{2,2}_{\rm imm}(0,1; \mathbf{R}^2):=\Set{ \gamma \in W^{2,2}(0,1; \mathbf{R}^2) |\,  |\gamma'(x)| \neq 0 \ \text{ for all} \ x\in[0,1] }.
\]
In general we impose the following assumptions on the obstacle function $\psi$:
\begin{itemize}
    \item[(A1)] $\psi: \mathbf{R} \rightarrow \mathbf{R}$ is uniformly continuous. 
    \item[(A2)] $\psi\vert_{(-\infty,0]}< 0$ and $\psi\vert_{[1,\infty)} < 0$.
    \item[(A3)] $\max_{x \in [0,1]} \psi(x)> 0$. 
\end{itemize}
An important special case, which we will often consider, is the case of \textit{symmetric cone obstacles}, that is, $\psi\vert_{(-\infty,\frac{1}{2}]}$ is affine linear and $\psi= \psi(1-\cdot)$. 
Notice that without Assumption (A3), each minimizer would be given by a trivial line segment connecting $(0,0)$ and $(1,0)$. 

It turns out that in the case of $\lambda = 0$ 
global minimizers in $A$ can not exist.
Due to this fact, we study only local minimizers in this case. On contrary, for $\lambda > 0$ one readily obtains existence of minimizers with the standard direct method. This motivates again the consideration of the length-penalized elastic energy.
A goal of this article is to understand how 
the behavior of minimizers of $\mathcal{E}_\lambda$ depends on $\lambda$.

Many articles involving the obstacle problem for the bending energy have focused on the smaller admissible set of \textit{graphs} lying above the obstacle, i.e. the authors minimize in 
\begin{align}\label{eq:M_graph}
    A_{\rm graph}:=
\Set{\gamma:[0,1]\to\mathbf{R}^2 |  \gamma(x)=(x,u(x))\ \text{ for some }\ u\in W^{2,2}(0,1)\cap W^{1,2}_0(0,1) \text{ with } u\geq\psi} \subseteq A. 
\end{align}
Our larger admissible set $A$ explicitly allows for open nongraphical curves, which we believe is a less studied subject in the current literature.

Allowing the curves to be non-graphical yields new difficulties. 
One of difficulties comes from \textit{loss of compactness}.
For example, let $c_\lambda \in A$ be a circular arc with radius $\lambda^{-\frac{1}{2}}$ for a sufficiently small $\lambda>0$. Then we see that $\mathcal{E}_\lambda[c_\lambda]=O(\sqrt{\lambda})$ as $\lambda\to0$.
The smaller $\lambda$, the greater the distance between  $c_\lambda$ and $\psi$, and this also implies that $c_\lambda$ does not touch $\psi$.
On the other hand, as in Remark~\ref{rem:escaping-circular_arcs} below, there also exists an admissible curve $C_\lambda\in A$ such that $C_\lambda$ always touches $\psi$ but also satisfies $\mathcal{E}_\lambda[C_\lambda]=O(\sqrt{\lambda})$ as $\lambda\to0$. 
The points where the curves $C_\lambda$ touch the obstacle escape to infinity as $\lambda \rightarrow 0$.
As a result, one can find both a `touching' planar curve and a `non-touching' planar curve with the vanishing energy as $\lambda\to0$, and hence it is not easy to give some answers to  the question
\begin{itemize}
    \item[(Q1)] How large is the \textit{coincidence set} $I_\gamma := \{ x \in [0,1] \mid \gamma^2(x) = \psi( \gamma^{1}(x)) \}$ of a minimizer $\gamma \in A$?
\end{itemize}
It will turn out that the answer to this question will in general depend on $\lambda$ and the precise shape of $\psi$. 
This is completely different from  results for
graphical curves, as it is shown that minimizers of $\mathcal{E}_0=B$ among the graphical curves must always touch the obstacle. 
In addition, our problem has the same difficulties as those in the existing literature. 
For example, several useful tools such as a maximum principle are not available since $\mathcal{E}_\lambda$ is an energy of higher order. 
Another difficulty is that (local) minimizers for the obstacle problem are characterized 
by a
so-called \textit{variational inequality}, not equality. 
This causes  a loss of regularity: in fact, it is well known that the optimal regularity for minimizers among open  graphical curves (in the sense of Sobolev class) is $W^{3,\infty}$. 
Due to the fact that (local) minimizers $\gamma \in A$ only satisfy a variational inequality, it is not a standard matter to answer questions as follows
\begin{itemize}
    \item[(Q2)] What is the \textit{optimal regularity} of $\gamma$?
    \item[(Q3)] What are \textit{geometric properties} of $\gamma$? (E.g.\ symmetry, graph representation etc.) 
\end{itemize}
Because of such difficulties questions (Q1)--(Q3) 
are delicate even if we restrict the class $A$ to the symmetric class 
$A_{\rm sym}$, where 
\begin{equation}
    A_{\rm sym} : = \{ \gamma \in A \mid \gamma(1-x)= (1-\gamma_1(x), \gamma_2(x))  \ \text{ for all } \ x\in [0,1]\}.
\end{equation}

Our first main result concerns the case $\lambda=0$ as a natural extension of existing studies. Here let us recall the known results.
According to \cite{Miura21},
if $\psi$ is a symmetric cone obstacle and satisfies $\psi(\frac{1}{2}) < h_*$, where
\begin{align}\label{def:h_*}
    h_*:=2\mathsf{B}(\tfrac{3}{4}, \tfrac{1}{2})^{-1} \simeq 0.83463 
\end{align}
and $\mathsf{B}$ denotes the Beta function, 
then minimizers of $\mathcal{E}_0=B$ in $A_{\rm graph}\cap A_{\rm sym}$ uniquely exist and can be computed explicitly. 
More precisely, away from the obstacle, the minimizer is characterized as a specific free elastica, that is a particular solution of $2\partial_s^2\boldsymbol{\kappa}+|\boldsymbol{\kappa}|^2\boldsymbol{\kappa}=0$. 
In addition, the minimizer touches $\psi$ only at its tip $(\frac{1}{2},\psi(\frac{1}{2}))$ with a horizontal tangent vector, 
and its curvature vanishes at the endpoints. 
In summary, if $\psi(\frac{1}{2})<h_*$, then the minimizer of $\mathcal{E}_0$ in $A_{\rm graph}\cap A_{\rm sym}$ consists of two isometric copies of a free elastica $\gamma_{\rm rect}$ chosen in such a way that $\kappa(0) = 0$ (see Figure~\ref{fig:SCF} (left)), which we call a \textit{symmetric cut-and-glued free-elasticae} in this paper.
On the other hand, if a symmetric cone obstacle $\psi$ satisfies $\psi(\frac{1}{2}) \geq h_*$, then minimizers do not exist, and hence symmetric cut-and-glued free-elasticae cannot be considered in $A_{\rm graph}$.

In this paper, even in the case of $\psi(\frac{1}{2}) \geq h_*$, we consider symmetric cut-and-glued free-elasticae by extending its notion to non-graphical curves (see Figure~\ref{fig:SCF} and Definition~\ref{def:SCF} for more details).  
As symmetric cut-and-glued free-elasticae yield minimizers in $A_{\rm graph}$ for suitably small obstacles, it is reasonable to expect that they also yield local minimizers of $\mathcal{E}_0$ in $A$. However, this conclusion turns out to be false above the aforementioned threshold height $h_*$. 
This destabilization phenomenon is characterized in the first main result and can be understood as a contribution to question (Q3) in the case of $\lambda= 0$. 

\begin{theorem}[Destabilization of symmetric cut-and-glued free-elasticae]\label{thm:destabilization-SCF}
Assume that $\psi$ is a symmetric cone obstacle and satisfies Assumptions (A1)--(A3). 
Let $\gamma\in A$ be a unique symmetric cut-and-glued free-elastica. 
Then the following assertions hold: 
\begin{itemize}
    \item[(i)] if $\psi(\frac{1}{2})<h_*$, then $\gamma$ is a local minimizer of $B$ in $A_{\rm sym}$ (but not a global minimizer);
    \item[(ii)] if $\psi(\frac{1}{2}) \geq h_*$, then $\gamma$ is not a local minimizer of $B$ in $A_{\rm sym}$ (as well as $A$), 
\end{itemize}
where $h_*>0$ is the universal constant defined by \eqref{def:h_*}. 
\end{theorem}

\begin{center}
    \begin{figure}[htbp]
      \includegraphics[scale=0.22]{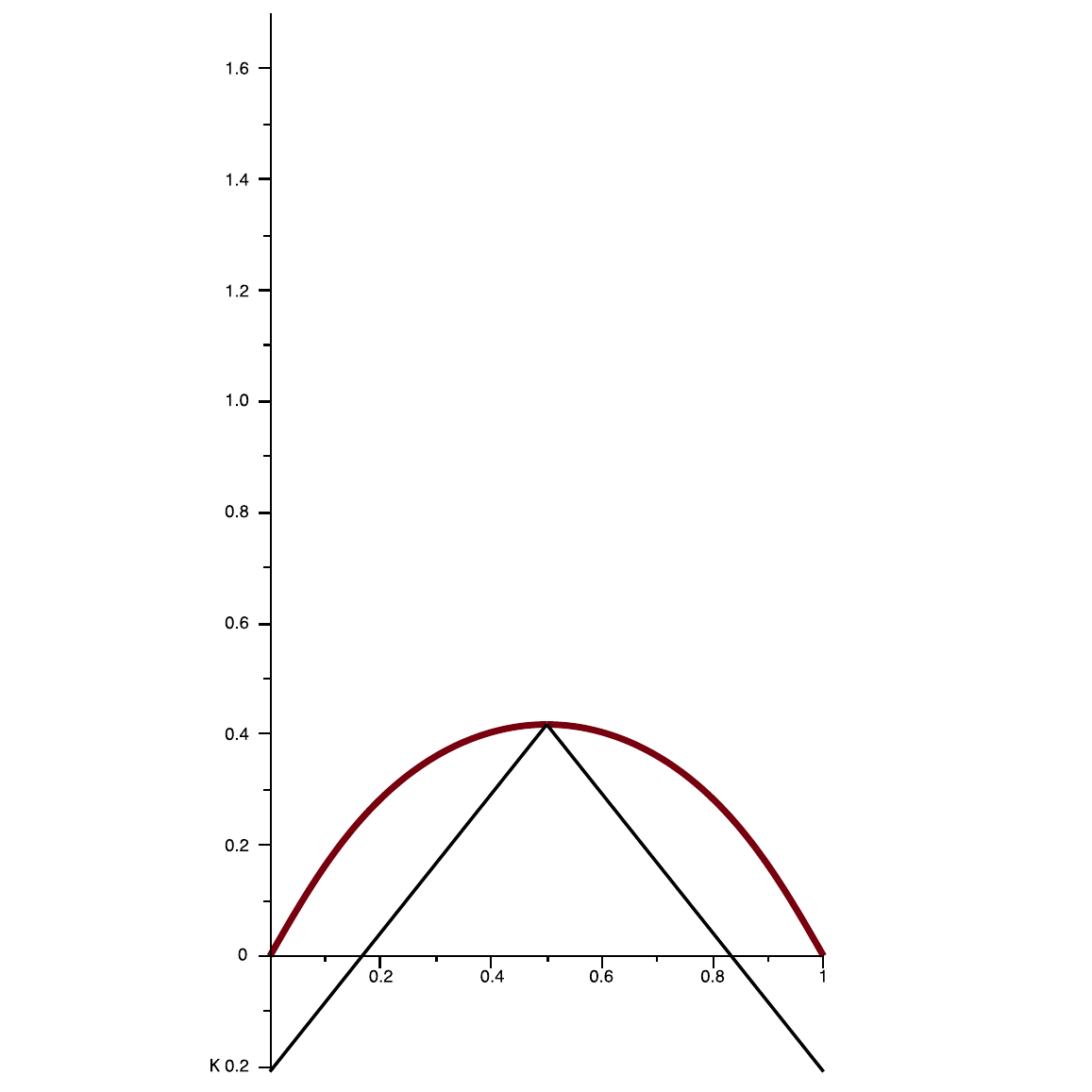}
      \hspace{-15pt}
      \includegraphics[scale=0.22]{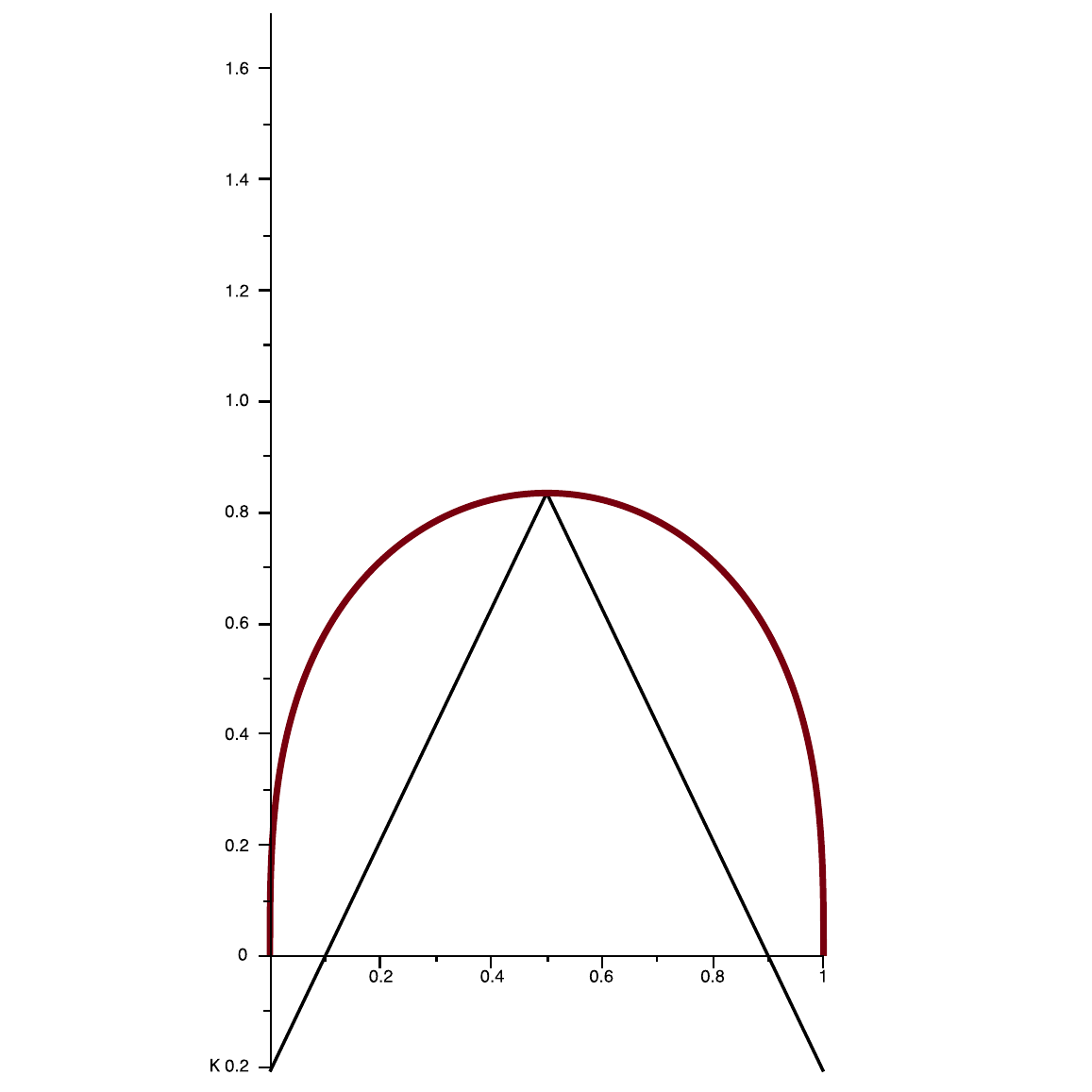}
      \hspace{0pt}
      \includegraphics[scale=0.22]{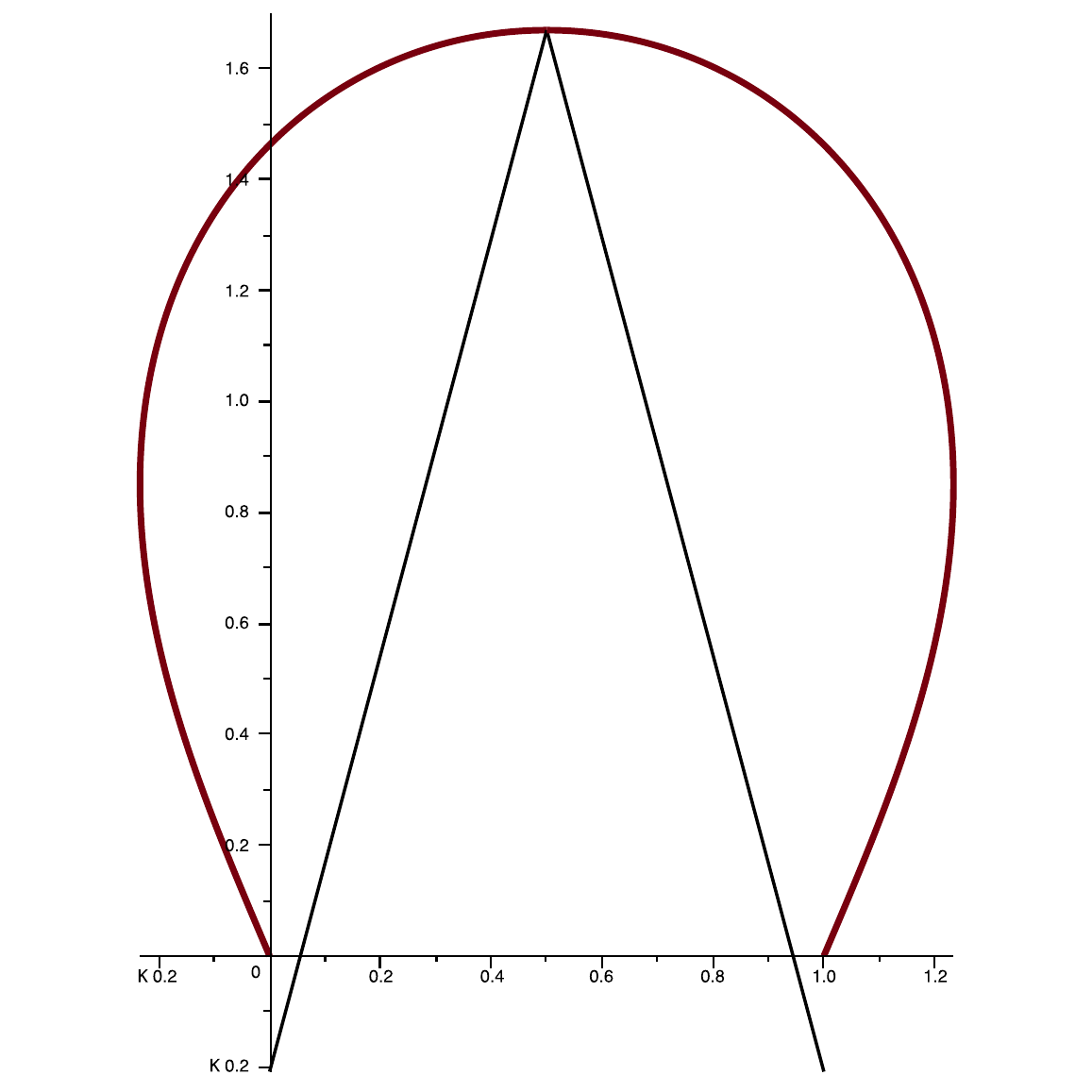} \\
      \includegraphics[scale=0.18]{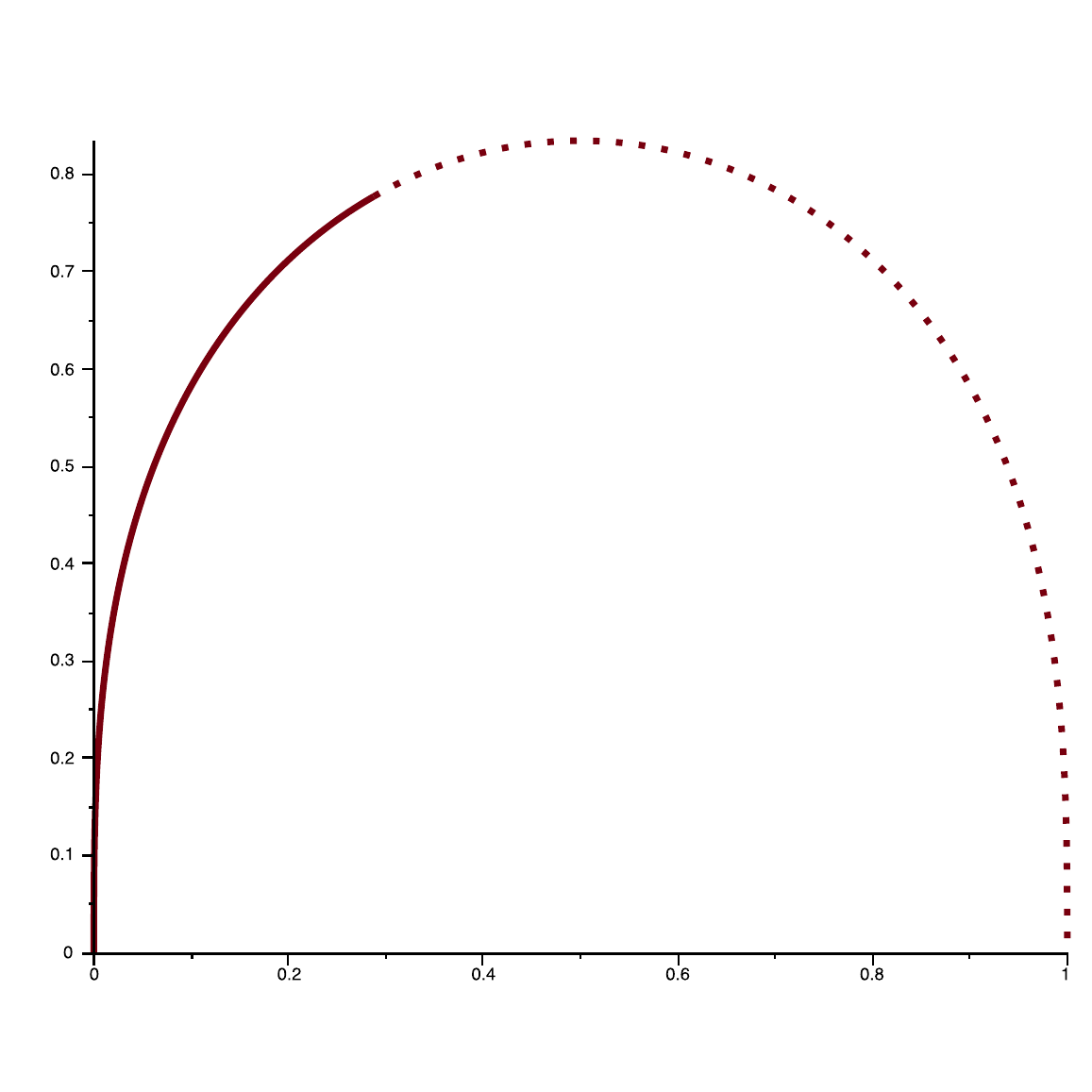}
      \hspace{10pt}
      \includegraphics[scale=0.18]{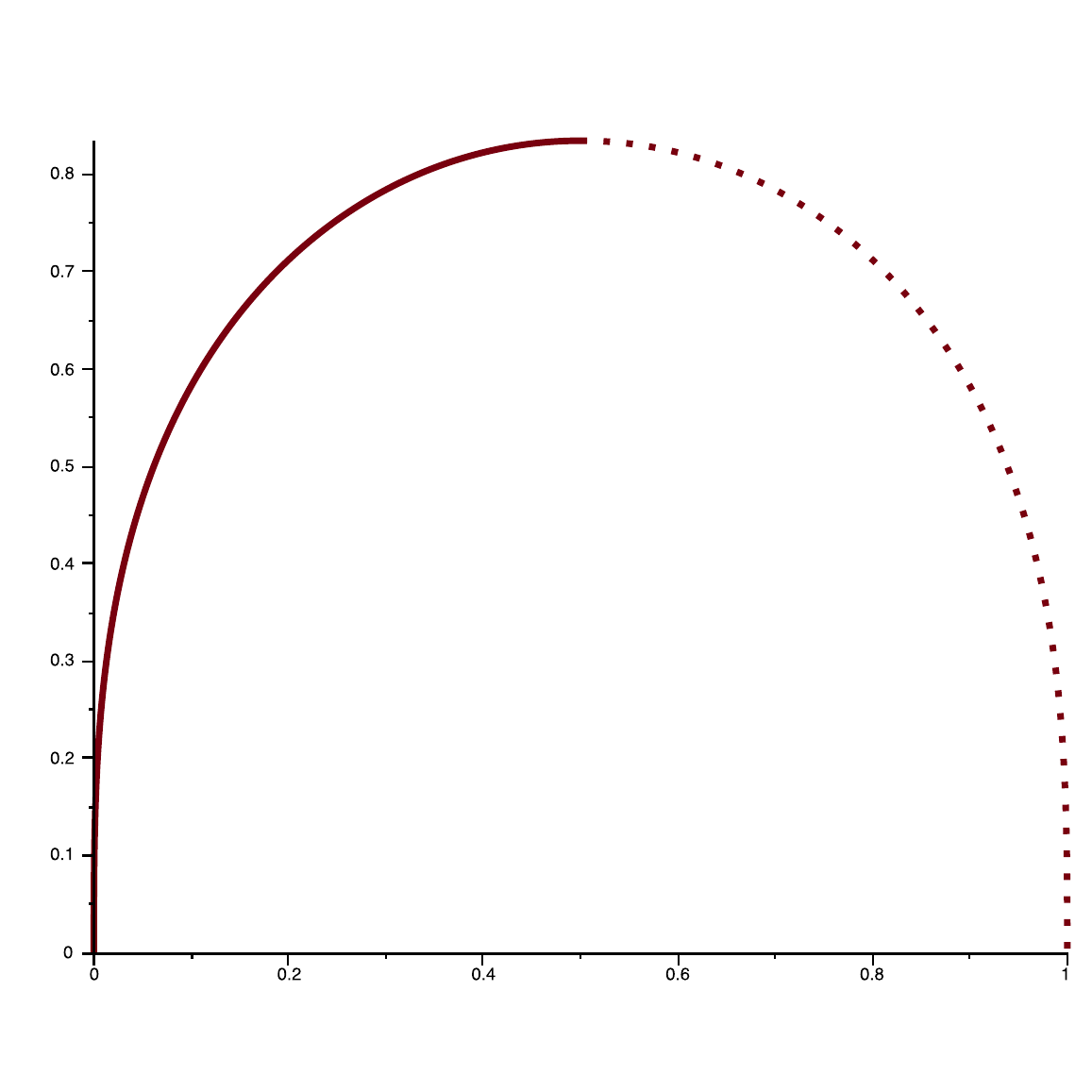}
      \hspace{10pt}
      \includegraphics[scale=0.18]{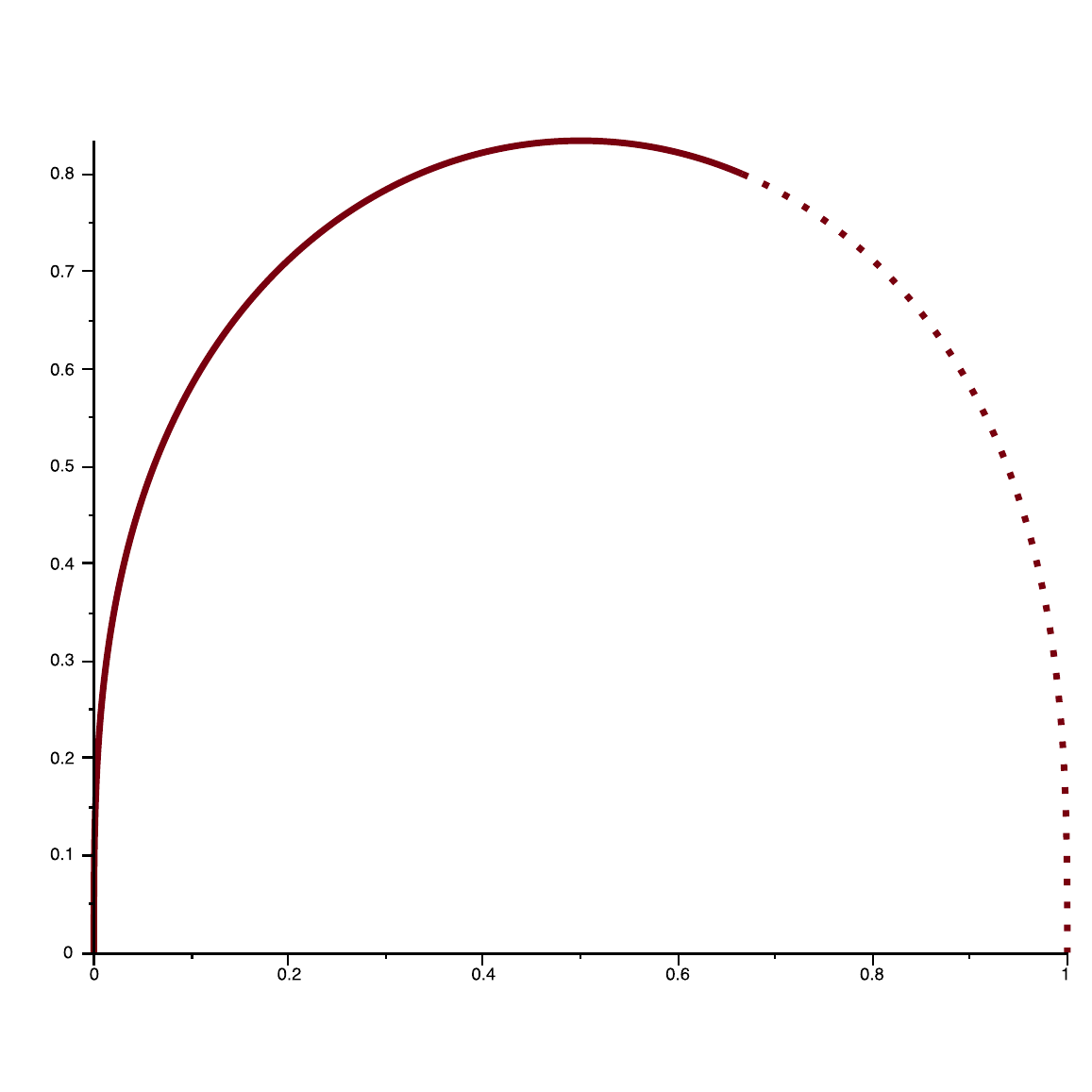}
  \caption{Symmetric cut-and-glued free-elasticae with $\psi(\frac{1}{2})=h_*/2$ (top left), $\psi(\frac{1}{2})=h_*$ (top center), $\psi(\frac{1}{2})=2h_*$ (top right), where $h_*\simeq 0.83463$ (cf.\ \eqref{def:h_*}). The bottom three figures represent rectangular elasticae $\gamma_{\rm rect}$ of each corresponding top figure.}
  \label{fig:SCF}
  \end{figure}
\end{center}

\vspace{-\baselineskip}

The key ingredient for the study of this destabilization phenomenon is the discussion of the variational inequality, see Section~\ref{eq:sec4} for details.

The variational inequality we derive can also be used to give an answer to (Q2) --- again optimal $W^{3,\infty}$-regularity can be obtained under the mild assumption that the obstacle is Lipschitz continuous (see Corollary~\ref{cor:optimal_regularity} below). 

Our two remaining main theorems deal with question (Q1), i.e., they study the coincidence set of minimizers of $\mathcal{E}_\lambda$ in $A_{\rm sym}$. Note that for large values of $\lambda$ the length of $L[\gamma]$ plays an important role in the energy $\mathcal{E}_\lambda$. Due to the scaling behavior of the length ($L[r\gamma] = r L[\gamma]$ for any $r > 0$) one would expect minimizers for large $\lambda$ to be as \textit{short as possible} and in particular to touch the obstacle. On contrary, for small $\lambda> 0$ the bending term $B$ dominates the behavior of $\mathcal{E}_\lambda$. Its scaling behavior is different from the length; more precisely $B[r\gamma] = r^{-1} B[\gamma]$ for all $r > 0$. In particular, for small $\lambda > 0$ minimizers are expected to be larger and in particular nontouching. These expectations are confirmed by the following two theorems.

\begin{theorem}[Nontouching symmetric minimizers for small $\lambda > 0$]\label{thm:hitting-escaping} 


There exists some $\bar{\lambda}>0$ (depending on $\psi$) such that for all $\lambda \in (0, \bar{\lambda})$ each minimizer of $\mathcal{E}_\lambda$ in $A_{\rm sym}$ does not touch the obstacle, i.e., $I_\gamma = \emptyset$. 
Moreover, in this case the minimizer is unique and given by the critical point $\gamma_{\rm larc}^{\lambda,1,1}$.
\end{theorem}

Here, $\gamma_{\rm larc}^{\lambda,1,1}$ is called \textit{$(\lambda,1,1)$-longer arc}, and it will be discussed in more detail in Section~\ref{sec:penalizedpinnedelasticae} (see Figure~\ref{fig:larc_3patterns}). 
The constant $\bar{\lambda}$ has to be chosen at least so small that 
\begin{align}\label{eq:height-nontouching}
\psi(\tfrac{1}{2}) < h_\lambda,
\end{align} 
where $h_\lambda>0$ is given by \eqref{eq:h_lambda} (which tends to $\infty$ as $\lambda \rightarrow 0$). This ensures admissibilty of $\gamma_{\rm larc}^{\lambda,1,1}.$

\begin{center}
    \begin{figure}[htbp]
      \includegraphics[scale=0.18]{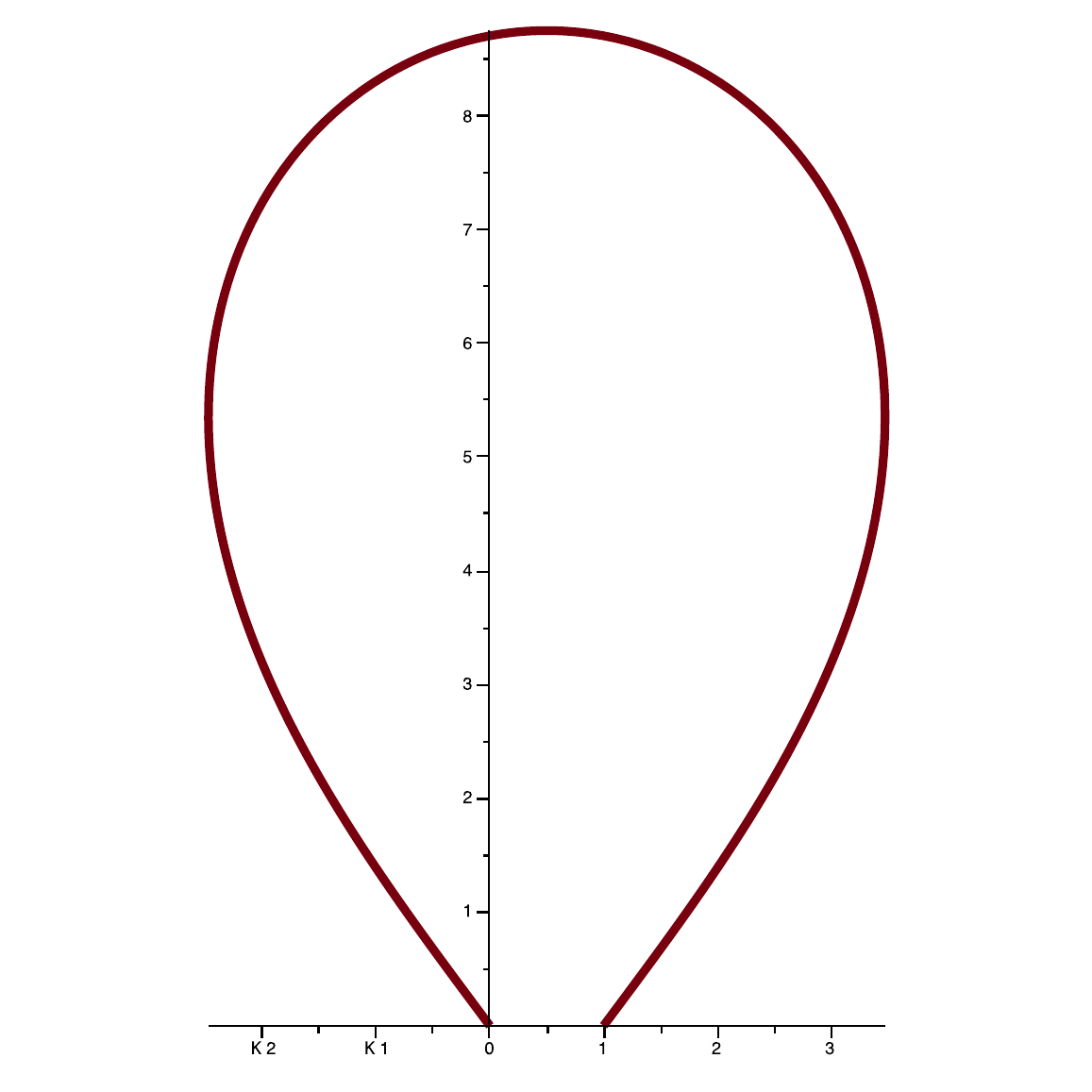}
      \includegraphics[scale=0.18]{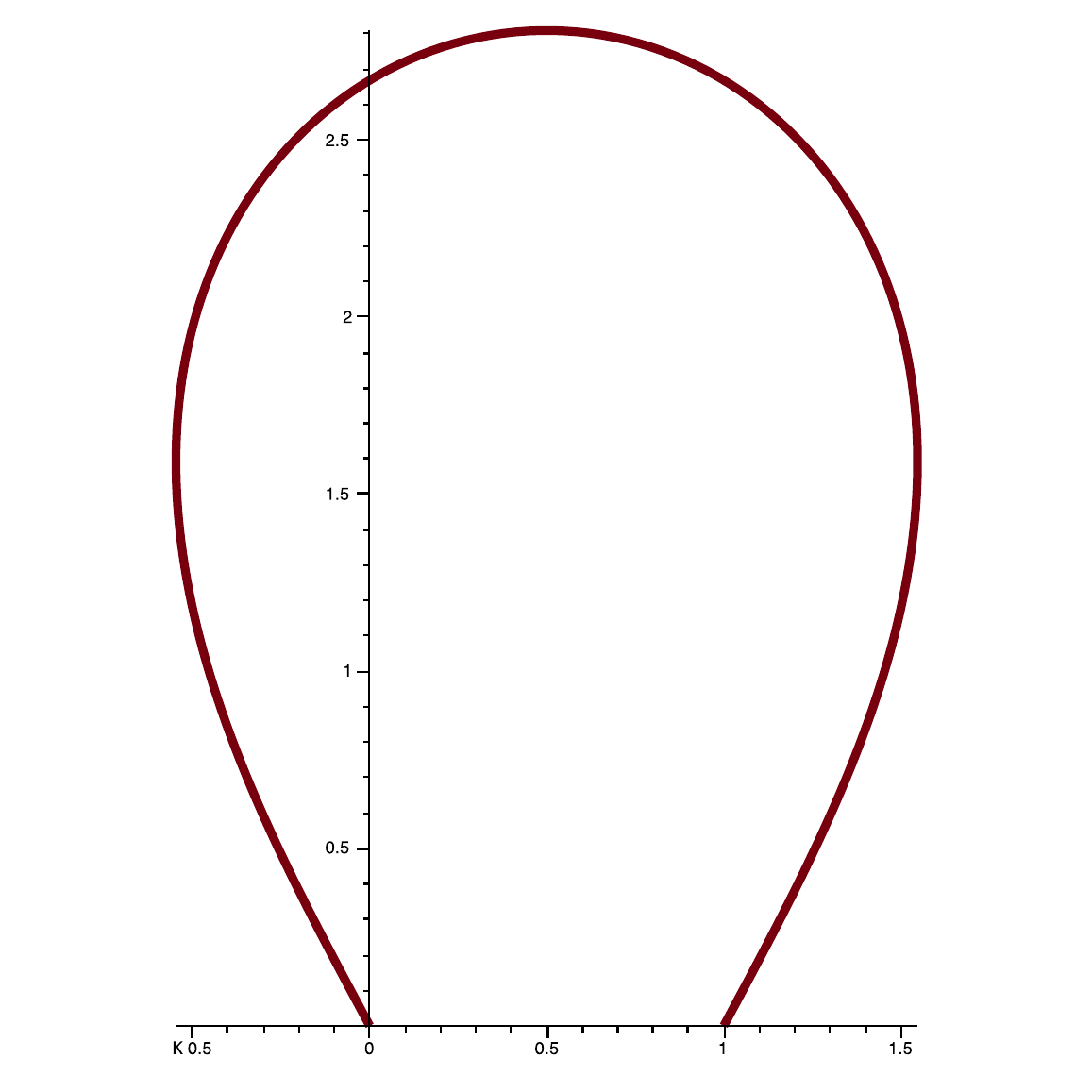}
      \includegraphics[scale=0.18]{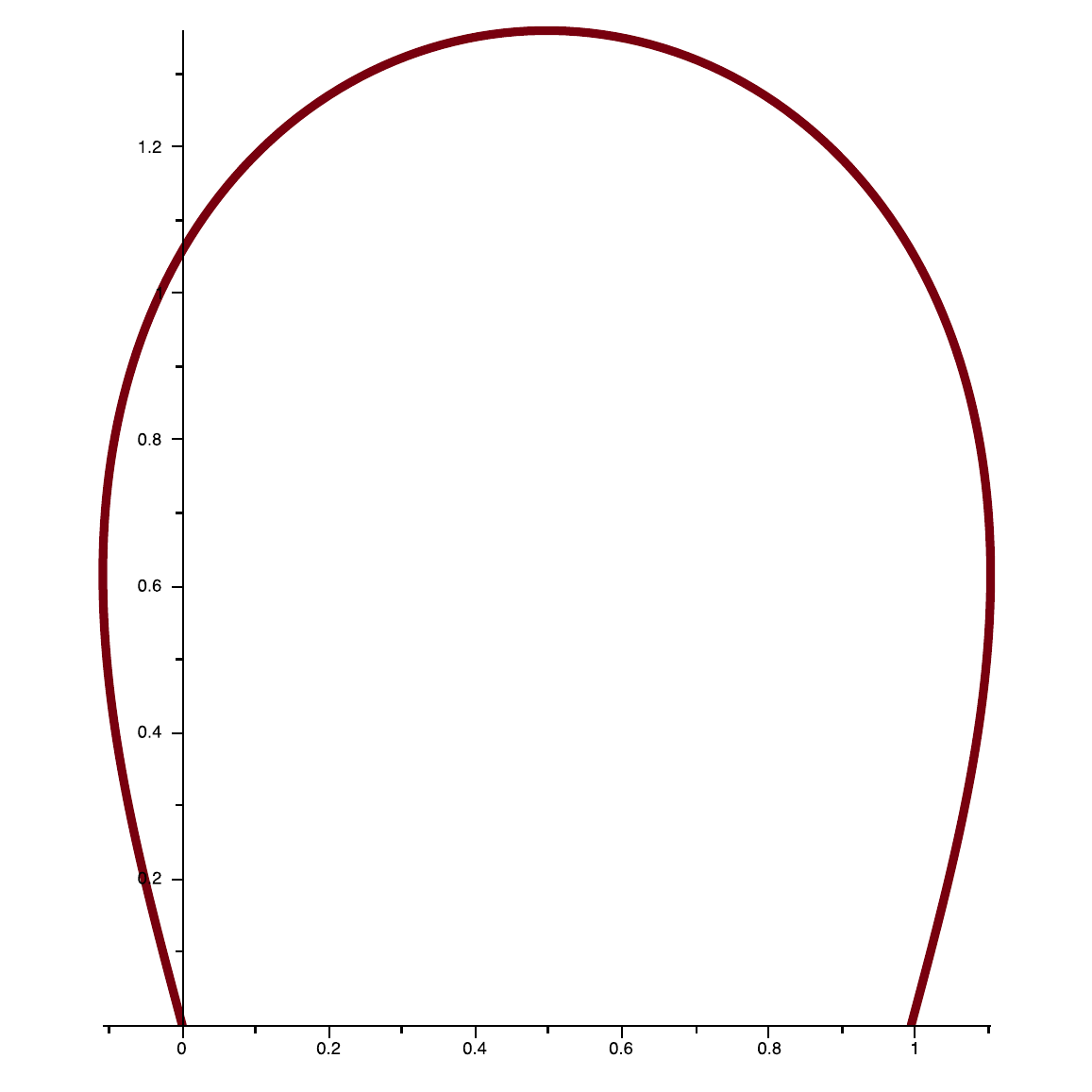}
  \caption{$(\lambda,1,1)$-Longer arcs ($\gamma_{\rm larc}^{\lambda,1,1}$) with $\lambda={1}/{20}$, $\lambda=\hat{\lambda}/2\simeq 0.35355$, $\lambda=\hat{\lambda}\simeq0.70710$ (from left to right), after rescaling.}
  \label{fig:larc_3patterns}
  \end{figure}
\end{center}

\vspace{-\baselineskip}

\begin{theorem}[Touching minimizers for large $\lambda$]\label{thm:touching_large_lambda}

There exists $\hat{\lambda}>0$ such that for any $\lambda>\hat{\lambda}$ and any obstacle $\psi:\mathbf{R}\to\mathbf{R}$ satisfying (A1)--(A3),
each minimizer $\gamma$ of $\mathcal{E}_\lambda$ in $A$ touches $\psi$, i.e., $I_\gamma \neq \emptyset$. 
\end{theorem}

Interestingly, the value $\hat{\lambda}\simeq 0.70107$ is known explicitly and can be chosen independently of $\psi$. This is not the case for $\bar{\lambda}$ in Theorem \ref{thm:hitting-escaping}. It is an interesting question for future research,  which conditions on the obstacle $\psi$  admit the choice $\bar{\lambda} = \hat{\lambda}$ in Theorem \ref{thm:hitting-escaping}. 

While the results seem likely due to the scaling behavior of $B$ and $L$, the absence of a maximum principle makes it nonstandard to classify the coincidence set. For the minimization in $A_{\rm graph}$, a maximum principle for the quantity $v = \frac{u''}{(1+u'^2)^\frac{5}{4}}$ can be used to derive nonemptyness of the coincidence set. We refer to \cite{DD18} for a detailed explanation. Our results are not at all relying on any maximum principle but more on the explicit knowledge of critical curves of $\mathcal{E}_\lambda$ in $A$. These are called \textit{penalized pinned elasticae} and have been studied in \cite{MYarXiv2409}, which forms the basis of our approach.  In particular the stability analysis of critical points in \cite{MYarXiv2409} becomes relevant in the proof of Theorem \ref{thm:touching_large_lambda}. We remark that each minimizer $\gamma$ of $\mathcal{E}_\lambda$ in $A$ must be a critical curve of $\mathcal{E}_\lambda$ on the noncoincidence set $[0,1] \setminus I_{\gamma}$. 
However, touching candidates with $I_{\gamma} \neq \emptyset$ are not necessarily globally critical points of $\mathcal{E}_\lambda$ and hence further analysis is required, in particular for ruling out touching candidates in the proof of Theorem~\ref{thm:hitting-escaping}.

Oftentimes, configurations of minimizers can be ruled out by means of \textit{energy arguments}, i.e., by showing that the investigated configuration requires an energy lying above the energy infimum. See e.g.\ \cite[Section 7]{DMN18} for an argument that rules out embedded minimizers just by looking at the energy of embedded competitors. However, for small $\lambda> 0$ the touching phenomenon can not be ruled out purely by looking at the energy of touching curves: recall that as mentioned earlier,  we may construct a family of curves $\{C_\lambda\}_{\lambda>0}$ that touch a symmetric cone obstacle for any $\lambda > 0$ but $\mathcal{E}_\lambda[C_\lambda]\to 0$ as $\lambda\to0$. 

To achieve the required compactness 
we will look at 
an auxiliary
minimization problem of $\mathcal{E}_\lambda$ subject to a \textit{rhomb obstacle constraint}, i.e. we minimize $\mathcal{E}_\lambda$ in
\[
A_{\rm sym}^\lozenge:= 
\Set{ \gamma \in W^{2,2}_{\rm imm}(0,1; \mathbf{R}^2) |
\begin{array}{l}
 \gamma(0)=(0,0), \ \gamma(1)= (1,0),
 \;  \gamma(1-x) = (1-\gamma^1(x),\gamma^2(x)), \\ 
 |\gamma^2(x)| \geq \psi(\gamma^1(x))) \text{ in } [0,1]
\end{array}
}.
\]
The obstacle is then to be understood as the region between the graphs of $\psi$ and $-\psi$  (which forms a rhomb if $\psi$ is affine linear). The compactness of this rhomb will play a decisive role in the discussion of the coincidence set of minimizers. 

The article is organized as follows. In Section~\ref{sec:penalizedPinnedElasticae} we collect some useful facts about critical points of $\mathcal{E}_\lambda$. In Section~\ref{sect:3} we discuss existence and regularity of minimizers for the problem we study. Section~\ref{eq:sec4} deals with the question of local minimality of symmetric cut-and-glued free elasticae in the case of $\lambda = 0$. Section~\ref{sect:Nontouching_small_lambda} deals with the noncoincidence of minimizers for small $\lambda > 0$ and Section \ref{sect:6} is dedicated to the coincidence phenomenon for large values of $\lambda$.    

\subsection*{Acknowledgments}
Much of this work was done during the second author’s visit at Freiburg University. 
The second author is very thankful to the first author for his warm hospitality and providing a fantastic research atmosphere.
The second author is supported by FMfI Excellent Poster Award 2022 and JSPS KAKENHI Grant Number 24K16951.

\section{Preliminaries}\label{sec:penalizedPinnedElasticae}

\subsection*{Notation}
Throughout this paper let $\langle \cdot, \cdot\rangle$ denote the Euclidean inner product in $\mathbf{R}^2$. In addition, let $e_1=(1,0)$ and $e_2=(0,1)$ be the canonical basis vectors of $\mathbf{R}^2$. Further, 
$R_\phi$ stands for the counterclockwise rotation matrix \textrm{with} 
angle $\phi\in \mathbf{R}$.
For an immersed curve $\gamma:[0,1]\to\mathbf{R}^2$, let $\partial_s$ denote the arclength derivative along $\gamma$, i.e., $\partial_s f=\frac{1}{|\gamma'|}f'$. 
Let $\boldsymbol{t}:= \partial_s \gamma$ denote the unit tangent, $\boldsymbol{n}:= R_{{\pi}/{2}}\partial_s \gamma$ denote the unit normal, $\boldsymbol{\kappa}:= \partial_{s}^2 \gamma$ denote the curvature vector, and $k:= \langle \partial_{s}^2 \gamma, \boldsymbol{n}\rangle$ denote the signed curvature.
For $0\leq a<b <2\pi$ we say $x\in [a,b]$ (mod $2\pi$) if $x \in \bigcup_{k \in \mathbf{Z}} [2k\pi+a,2k\pi+b]$.

We also introduce the following notation of a simple concatenation of curves: For $\gamma_j:[a_j, b_j]\to \mathbf{R}^2$ with $L_j:=b_j-a_j \geq0$, we define $\gamma_1\oplus \gamma_2:[0,L_1+L_2]\to \mathbf{R}^2$ by 
\begin{align}\label{eq:def-oplus}
\left( \gamma_1 \oplus \gamma_2 \right)(s):=
\begin{cases}
\gamma_1(s+a_1), \quad &s\in[0,L_1], \\
\gamma_2(s+a_2-L_1)+\gamma_1(b_1)-\gamma_2(a_2), \quad &s\in[L_1,L_1+L_2],
\end{cases}
\end{align}
and for $N\geq3$, inductively define $\gamma_1 \oplus \cdots \oplus \gamma_N := (\gamma_1 \oplus \cdots \oplus\gamma_{N-1})\oplus \gamma_N$.

For $\gamma \in A$, we define a \textit{non-coincidence set} and a \textit{coincidence set} by 
\[
\widetilde{N}_\gamma:= \big\{ s\in (0,L) \,|\, \tilde{\gamma}_2(s) > \psi(\tilde{\gamma}_1(s)) \big\}
\quad\text{and}\quad \widetilde{I}_\gamma:=(0,L)\setminus \widetilde{N}_\gamma, 
\]
respectively, where $L:=L[\gamma]$ and $\tilde{\gamma}$ denotes the arclength parametrization of $\gamma$.
Also define the \textit{coincidence set} $I_\gamma$ by the coincidence set with the original parameter, i.e., $I_\gamma:=\{x\in(0,1) \,|\, \gamma^2(x)=\psi(\gamma^1(x))\}$.

Next, we introduce some special elasticae, which play a role throughout the article.

\subsection{Penalized pinned elasticae}\label{sec:penalizedpinnedelasticae} 
Let $A_{\rm pin}$ be the set of immersed curves whose endpoints are fixed as follows: 
\[
A_{\rm pin}:= 
\Set{ \gamma \in W^{2,2}_{\rm imm}(0,1;\mathbf{R}^2) | \gamma(0)=(0,0), \ \ \gamma(1)=(1,0)  }. 
\]
\begin{definition}
Let $\lambda\geq0$.
We call 
$\gamma \in A_{\rm pin} \cap C^\infty$ a \textit{penalized pinned elastica} if the signed curvature $k:[0,L]\to\mathbf{R}$ of  the arclength parametrization of $\gamma$ satisfies $k(0)=k(L)=0$ and $2k''+k^3-\lambda k=0$ \ in \ $(0,L)$.
\end{definition}

The equation $2k''+k^3-\lambda k=0$ is the Euler--Lagrange equation for the length-penalized functional $\mathcal{E}_\lambda$. For $\lambda= 0$ it is equivalent to the elastica equation given in the introduction. This is readily checked using that $\boldsymbol{\kappa}$ is parallel to $\boldsymbol{n}$.
 Complete classification and all the explicit formulae of penalized pinned elasticae are already obtained (cf.\ \cite{MYarXiv2409}). We recall them here for the reader's convenience.
To this end, we 
introduce the functions $f:[\frac{1}{\sqrt{2}},1)\to\mathbf{R}$ and $g:[\frac{1}{\sqrt{2}},1)\to\mathbf{R}$ defined by 
\begin{align}
f(q)&:= (4q^4-5q^2+1)\mathrm{K}(q) + (-8q^4+8q^2-1)\mathrm{E}(q), \label{eq:def-f} 
\\
g(q)&:=8\big( 2\mathrm{E}(q)-\mathrm{K}(q) \big)^2 (2q^2-1), \label{eq:def-g} 
\end{align}
where $\mathrm{K}(q)$ and $\mathrm{E}(q)$ denote the complete elliptic integrals (cf.\ Appendix~\ref{sect:elliptic_functions}).
It is shown in \cite[Lemma 2.3]{MYarXiv2409} that $f$ has a unique root $\hat{q}\simeq 0.79257$.
Using this value $\hat{q}$ we define
\begin{align}\label{eq:def-hat_lambda}
\hat{\lambda}:=g(\hat{q})\simeq 0.70107. 
\end{align}
It is shown in \cite{MYarXiv2409}*{Lemma 2.4} that $g$ is strictly increasing in $(\frac{1}{\sqrt{2}},\hat{q}] \cup [q_*,1)$ and strictly decreasing in $(\hat{q},q_*)$, where $q_*\simeq 0.90891$ denotes the unique zero of $q\mapsto 2\mathrm{E}(q)-\mathrm{K}(q)$ (see also Lemma~\ref{lem:elliptic_2E-K}). 
Define 
\begin{itemize}
    \item For $c\in(0,\hat{\lambda}]$, let $q_1(c), q_2(c), q_3(c)$ be the solutions to $g(q)=c$ with 
    \begin{align}\label{eq:mod_PPE}
    q_1(c)\in(\tfrac{1}{\sqrt{2}},\hat{q}], \quad q_2(c) \in [\hat{q}, q_*), \quad q_3(c)\in(q_*,1).
    \end{align}
    We interpret $q_1(\hat{\lambda})=q_2(\hat{\lambda})=\hat{q}$.
    \item For $c>\hat{\lambda}$, let $q_3(c)\in(q_*,1)$ be a unique solution to $g(q)=c$.
\end{itemize}

These parameters
yield the following complete classification of penalized pinned elasticae.

\begin{proposition}[{\cite[Theorem 1.1]{MYarXiv2409}}]\label{prop:PPE-classification}
Let $\lambda>0$ and 
\begin{align}\label{eq:def-n_lam}
n_\lambda:=\bigg\lceil \sqrt{\tfrac{\lambda}{\hat{\lambda}}}\bigg\rceil.
\end{align}
Suppose that $\gamma \in A_{\rm pin}$ is a penalized pinned elastica.
Then, either $\gamma$ is a trivial line segment; or 
up to reflection and reparametrization, $\gamma$ is represented by one of the following curves. 
\begin{itemize}
    \item (Case I: $(\lambda,1,n)$-shorter arc) For $q=q_1(\frac{\lambda}{n^2})$ and $\alpha=2n(2\mathrm{E}(q)-\mathrm{K}(q))$, 
    \begin{align}\label{eq:sarc-formula}
        \gamma_{\rm sarc}^{\lambda,1,n}(s)&:= \frac{1}{\alpha}
        \begin{pmatrix}
        2\mathrm{E}(\am(\alpha s-\mathrm{K}(q),q),q) + 2\mathrm{E}(q) - \alpha s \\
        2q\cn(\alpha s-\mathrm{K}(q),q)
        \end{pmatrix} \qquad (s \in [0,2n \mathrm{K}(q)/\alpha])
    \end{align}
    for some $n \in \mathbf{N}$ such that $n\geq n_\lambda$. 
    \item (Case II: $(\lambda,1,n)$-longer arc) For $q=q_2(\frac{\lambda}{n^2})$ and $\alpha=2n(2\mathrm{E}(q)-\mathrm{K}(q))$, 
    \begin{align}\label{eq:gamma_larc-PPE}
        \gamma_{\rm larc}^{\lambda,1,n}(s)&:= \frac{1}{\alpha}
        \begin{pmatrix}
        2\mathrm{E}(\am(\alpha s-\mathrm{K}(q),q),q) + 2\mathrm{E}(q) - \alpha s \\
        2q\cn(\alpha s-\mathrm{K}(q),q)
        \end{pmatrix}  \qquad (s \in [0,2n \mathrm{K}(q)/\alpha])
    \end{align}
    for some $n \in \mathbf{N}$ such that $n\geq n_\lambda$. 
    \item (Case III: $(\lambda,1,n)$-loop) For $q=q_3(\frac{\lambda}{n^2})$ and $\alpha=2n(-2\mathrm{E}(q)+\mathrm{K}(q))$, 
    \begin{align}
        \gamma_{\rm loop}^{\lambda,1,n}(s)&:= \frac{1}{\alpha}
        \begin{pmatrix}
        -2\mathrm{E}(\am(\alpha s-\mathrm{K}(q),q),q) - 2\mathrm{E}(q) + \alpha s \\
        2q\cn(\alpha s-\mathrm{K}(q),q)
        \end{pmatrix} \qquad (s \in [0,2n \mathrm{K}(q)/\alpha])
    \end{align}
    for some $n \in \mathbf{N}$ such that $n \geq 1$. 
\end{itemize}
\end{proposition}

From these explicit formulae, we can deduce the maximal height of $\gamma^{\lambda,1,1}_{\rm larc}$. Indeed,
\begin{align}\label{eq:h_lambda}
\|(\gamma_{\rm larc}^{\lambda,1,1})^2\|_{L^\infty(0,L)}=(\gamma_{\rm larc}^{\lambda,1,1})^2(\tfrac{L}{2})=\frac{q_2(\lambda)}{2\mathrm{E}(q_2(\lambda)) - \mathrm{K}(q_2(\lambda))}=:h_\lambda,
\end{align}
where $(\gamma^{\lambda,1,1}_{\rm larc})^{2} := \langle \gamma^{\lambda,1,1}_{\rm larc}, e_2 \rangle$ and $L=L[\gamma_{\rm larc}^{\lambda,1,1}]$.
Since $q_2(\lambda) \to q_*$ as $\lambda\to0$, it 
follows from  \eqref{eq:h_lambda} 
that 
\begin{align}\label{eq:diverge_larc}
\lim_{\lambda \downarrow 0} h_\lambda= \lim_{\lambda \downarrow 0}(\gamma_{\rm larc}^{\lambda,1,1})^2(\tfrac{L}{2}) = \infty. 
\end{align}
The explicit formulae also  give us the energy formulae, e.g.\ 
\begin{align}\label{eq:longer-arc_energy}
    \mathcal{E}_\lambda[\gamma_{\rm larc}^{\lambda,1,n}]= 2\sqrt{2}n\sqrt{\lambda}\frac{1}{\sqrt{2q_2^2 -1}}\big( (4q_2^2 -3 ) \mathrm{K}(q_2)+2\mathrm{E}(q_2) \big),
\end{align}
where $q_2=q_2(\frac{\lambda}{n^2})$ as in \eqref{eq:mod_PPE}
(see \cite{MYarXiv2409}*{Equation 4.6}).

\begin{remark}\label{rem:energy_comparison_PPE}
Uniqueness of minimizers among nontrivial minimal penalized pinned elasticae has been also revealed in \cite{MYarXiv2409}. 
More precisely, it is shown in \cite{MYarXiv2409}*{Theorem 1.3} that for any nontrivial penalized pinned elastica $\gamma \in A_{\rm pin}$
\[
\mathcal{E}_\lambda[\gamma_{\rm larc}^{\lambda,1,1}] \leq \mathcal{E}_\lambda[\gamma] \quad \text{if} \ \ \lambda \leq \hat{\lambda}. 
\]
Moreover, equality is attained if and only if $\gamma$ coincides up to reflection and reparametrization with $\gamma_{\rm larc}^{\lambda,1,1}$. 
\end{remark}

\subsection{Figure-eight elastica}
For given $n\in\mathbf{N}$ and the unique zero $q_*$ of $q \mapsto2\mathrm{E}(q)-\mathrm{K}(q)$, define an \textit{$\frac{n}{2}$-fold figure-eight elastica} by 
\begin{align}\label{eq:figure-8}
    \gamma_{\rm leaf}^{\lambda,n}(s):=\frac{1}{\alpha_*}
        \begin{pmatrix}
        2\mathrm{E}(\am(\alpha_* s-\mathrm{K}(q_*),q_*),q_*) + 2\mathrm{E}(q_*) - \alpha_* s \\
        \sqrt{2}\cn(\alpha_* s-\mathrm{K}(q_*),q_*)
        \end{pmatrix}, 
        \ \ \text{where }\ \alpha_*:=\sqrt{\frac{\lambda}{2(2q_*^2-1)}}
\end{align}
with $s\in[0,2n\mathrm{K}(q_*)/\alpha_*]$. 
The length of $\gamma_{\rm leaf}^{\lambda,n}$ is given by $L[\gamma_{\rm leaf}^{\lambda,n}]=2n\mathrm{K}(q_*)/\alpha_*$. 
The signed curvature is then given by
\begin{align}\label{eq:figure-8_curvature}
    k(s)=2\alpha_* q_* \cn(\alpha_* s-\mathrm{K}(q_*),q_*) \quad \text{for }\ s\in [0,L]. 
\end{align}
As shown already in \cite[Proposition 4.8]{DNP2020}, a $\frac{1}{2}$-fold figure-eight elastica is the
unique global minimizer of $\mathcal{E}_\lambda$ in $A_{\rm drop}=\{ \gamma\in W^{2,2}_{\rm imm}(0,1;\mathbf{R}^2)\,|\, \gamma(0)=\gamma(1)\}$: 
\begin{proposition}\label{prop:fig8_minimizes}
Let $\lambda>0$ and suppose that $\gamma \in A_{\rm drop}$. 
Then, 
\begin{align}\label{eq:drop_minimal}
\mathcal{E}_\lambda[\gamma] \geq \mathcal{E}_\lambda[\gamma_{\rm leaf}^{\lambda,1}]=2\sqrt{2}\sqrt{\lambda}\frac{(4q_*^2-3)\mathrm{K}(q_*)+2\mathrm{E}(q_*)}{\sqrt{2q_*^2-1}}.
\end{align}
Equality in \eqref{eq:drop_minimal} is attained if and only if $\gamma$ is a $\frac{1}{2}$-fold figure-eight elastica (up to reflection, rotation, translation, and reparametrization).
\end{proposition}
One can give an alternative proof of the previous proposition along the lines of \cite{MYarXiv2409}.
For convenience 
we give the argument in Appendix~\ref{sect:appendixB}.

\subsection{Free elastica} 
Our classification in Proposition~\ref{prop:PPE-classification} is only valid provided that $\lambda > 0$. In the case of free elasticae, i.e. $\lambda=0$, we obtain a different classification of pinned elasticae. One readily checks that the limit $\lambda \rightarrow 0$ in the formulae of $\gamma^{\lambda,1,n}_{\rm larc}$ and $\gamma^{\lambda,1,n}_{\rm loop}$ is not well-defined, see e.g. \eqref{eq:diverge_larc}.
On the other hand, by letting $\lambda\to0$ 
in the explicit formula \eqref{eq:sarc-formula} of $\gamma_{\rm sarc}^{\lambda,1,1}$ 
we obtain a planar curve defined by 
\begin{align}\label{eq:rect-formula}
	\gamma_{\rm rect}(s)&:= \frac{1}{\alpha_0}
        \begin{pmatrix}
        2\mathrm{E}(\am(\alpha_0 s-\mathrm{K}(\tfrac{1}{\sqrt{2}}),\tfrac{1}{\sqrt{2}}),\tfrac{1}{\sqrt{2}}) + 2\mathrm{E}(\tfrac{1}{\sqrt{2}}) - \alpha_0 s \\
        \sqrt{2}\cn(\alpha_0 s-\mathrm{K}(\tfrac{1}{\sqrt{2}}),\tfrac{1}{\sqrt{2}})
        \end{pmatrix}, \qquad s \in [0, 2\mathrm{K}(\tfrac{1}{\sqrt{2}})/\alpha_0], 
\end{align}
where $\alpha_0:=4\mathrm{E}(\tfrac{1}{\sqrt{2}})-2\mathrm{K}(\tfrac{1}{\sqrt{2}})$.
The length and the signed curvature are given by $L_{\rm rect}:=2\mathrm{K}(\tfrac{1}{\sqrt{2}})/\alpha_0$ and 
\begin{align}\label{eq:kappa_rect}
k_{\rm rect}(s):= -\sqrt{2} \alpha_0 \cn \big(\alpha_0 s - \mathrm{K}(\tfrac{1}{\sqrt{2}}), \tfrac{1}{\sqrt{2}} \big),   
\end{align}
respectively.
In addition, we observe from \eqref{eq:rect-formula} that the tangential angle $\theta_{\rm rect}:[0, L_{\rm rect}]\to\mathbf{R}$ of $\gamma_{\rm rect}$ satisfies 
\begin{align}\label{eq:angle-rect}
    \theta_{\rm rect}(0)=\frac{\pi}{2}, \quad \theta_{\rm rect}(L_{\rm rect})=-\frac{\pi}{2}, \quad \text{and} \quad \theta_{\rm rect} \text{ is strictly decreasing}, 
\end{align}
and in particular $\gamma_{\rm rect}$ has a vertical slope at the endpoints.

It is well-known 
that if $\gamma \in A_{\rm pin}$ is a free elastica, then its signed curvature $k$ is given by $k\equiv0$ or $k= \pm k_{\rm rect}$, see e.g. \cite{Mandel}. 
Noting that 
\[
\alpha_0=4\mathrm{E}(\tfrac{1}{\sqrt{2}}) - 2\mathrm{K}(\tfrac{1}{\sqrt{2}}) = 2\int_0^1\frac{\sqrt{1-z^2}}{\sqrt{1-\frac{1}{2}z^2}}\; \mathrm{d}z = \frac{1}{\sqrt{2}}\int_0^1\frac{\xi^{-\frac{1}{4}}}{\sqrt{1-\xi}}\; \mathrm{d}\xi = \frac{1}{\sqrt{2}}\mathsf{B}\big( \tfrac{3}{4},\tfrac{1}{2}\big), 
\]
where $\mathsf{B}$ denotes the Beta function, we see that
\begin{align}\label{eq:rect-height}
    \gamma_{\rm rect}(\tfrac{L_{\rm rect}}{2}) = \left(\tfrac{1}{2}, 2\mathsf{B}(\tfrac{3}{4}, \tfrac{1}{2})^{-1} \right). 
\end{align}

This number coincides with the critical height threshold  $h_*$ defined by \eqref{def:h_*} for the existence of minimizers for the obstacle problem among graphs obtained in \cite{DD18, Mueller19, Miura21, Mueller21, Ysima}.

\section{Fundamental properties for the obstacle problem}\label{sect:3}

In this section, we collect some fundamental properties of (local) minimizers of $\mathcal{E}_\lambda$ in $A$ and study their optimal regularity. 
Hereafter, we always let $\psi$ satisfy assumptions (A1)--(A3) unless specified otherwise.

\subsection{Description of local minimizers}
First, as known in the previous results of Dall'Acqua--Deckelnick \cite{DD18}*{Proposition~3.2}, we check that any local minimizer $\gamma$ of $\mathcal{E}_\lambda$ in $A$ satisfies the Euler--Lagrange equation $2k''+k^3-\lambda k=0$ where $\gamma$ does not touch the obstacle.

\begin{lemma}[Euler--Lagrange equation above the obstacle]\label{lem:E-L_eq_on_noncoincidence}
Let $\lambda\geq0$ and $\gamma \in A$ be a local minimizer of $\mathcal{E}_\lambda$ in $A$.  
Then, the signed curvature $k$ of $\gamma$ is analytic on $\widetilde{N}_\gamma$ and satisfies $2k'' + k^3 -\lambda k=0$ in $\widetilde{N}_\gamma$. 
In addition, $k$ is also smooth in a neighborhood of $s=0$ and $s=L$ and satisfies $k(0)=k(L)=0$. 
\end{lemma}
\begin{proof}[Sketch of proof.]
Let $\gamma \in A$ be a local minimizer of $\mathcal{E}_\lambda$ in $A$ and $k$ denote its signed curvature. 
Fix $s_0 \in \widetilde{N}_\gamma$ arbitrarily. 
Then, there exists a neighborhood $U \subset \widetilde{N}_\gamma$ of $s_0$ such that,
for any $\varphi \in C^\infty_{\rm c}(U;\mathbf{R}^2)$ and for any $\varepsilon \in (-\varepsilon_0, \varepsilon_0)$ with sufficiently small $\varepsilon_0>0$, we have $\gamma_\varepsilon:=\gamma +\varepsilon\varphi(s(\cdot)) \in A$, where $s(x):=\int_0^x |\gamma'|$. 
This fact together with the known formulae of first derivatives of $B$ and $L$ (cf.\ \cite{MY_AMPA}*{Lemma A.1}) yields
\begin{align}\label{eq:EL-on_N_psi}
\frac{d}{d\varepsilon}\mathcal{E}_\lambda[\gamma_\varepsilon]\Big|_{\varepsilon=0}=
\int_U \Big( 2\langle \boldsymbol{\kappa}, \varphi'' \rangle - 3 |\boldsymbol{\kappa}|^2 \langle \boldsymbol{t}, \varphi' \rangle - \lambda \langle \boldsymbol{\kappa}, \varphi \rangle \Big) \; \mathrm{d}s=0. 
\end{align}
By a standard bootstrap argument we first deduce that $k$ is of class $C^\infty(U)$ and by integrating by parts we also deduce that $k$ satisfies 
$
2k''+k^3-\lambda k=0$ 
in $U$.
Since $k$ is now characterized as a solution of the polynomial differential equation, we obtain analyticity of $k$.

Explicit examination of solutions of $2 k'' + k^3- \lambda k = 0$, see e.g.\ \cite{LS_JDG}, actually yields that $k \in C^\infty(\overline{U})$. Combining (A2) with the fact that $\gamma(0)=(0,0), \gamma(1) =(1,0)$, this already implies that $k$ is smooth in a neighborhood of $s= 0$ and $s= L$.  Given that $ \boldsymbol{n}' = -k \boldsymbol{t}$ we find also that $\boldsymbol{n} \in C^1(\overline{U}; \mathbf{R}^2)$. Bootstrapping we actually obtain $\boldsymbol{n} \in C^\infty(\overline{U};\mathbf{R}^2)$.  We  turn to the boundary condition $k(0)=k(L)=0$.
Combining (A2) with the fact that $\gamma(0)=(0,0)$, we find a sufficiently small $\delta>0$ such that $\boldsymbol{n} \in C^\infty([0,\delta]; \mathbf{R}^2)$ and $\gamma_\varepsilon:=\gamma +\varepsilon\varphi(s(\cdot)) \boldsymbol{n}(s(\cdot)) \in A$ holds for any $\varphi \in W^{2,2}(0,\delta)$ with $\supp\,{\varphi} \subset [0,\delta)$ and $\varphi(0)=0$, and for any $\varepsilon \in (-\varepsilon_0, \varepsilon_0)$ with sufficiently small $\varepsilon_0>0$.
Therefore, the identity \eqref{eq:EL-on_N_psi} where $U$ is replaced with $[0,\delta)$ also holds for  perturbations of this kind, and by choosing $\varphi$ such that $\varphi'(0)=1$, and then integrating by parts, we obtain
\[
\left[2 k(s)\varphi'(s) \right]_{s=0}^{s=\delta} + \int_0^\delta \big( 2k'' + k^3 -\lambda k \big) \varphi \; \mathrm{d}s=0, 
\]
which implies that $k(0)=0$. Similarly we obtain $k(L)=0$.
\end{proof}

Thus, local minimizers are characterized by the Euler--Lagrange equation on the non-coincidence set.
To obtain a characterization that is also valid on the whole interval, we derive a so-called \textit{variational inequality} for local minimizers.
\begin{lemma}[Variational inequality]\label{lem:vari-ineq}
Let $\lambda\geq0$.
If $\gamma \in A$ is a local minimizer of $\mathcal{E}_\lambda$ in $A$, then the curvature vector $\boldsymbol{\kappa}$ satisfies
\begin{align}\label{eq:vari-ineq-explict}
    \int_0^L \Big(2\langle \boldsymbol{\kappa},e_2\rangle \varphi''-3|\boldsymbol{\kappa}|^2\langle\boldsymbol{t},e_2\rangle\varphi'-\lambda \langle\boldsymbol{\kappa},e_2\rangle\varphi \Big)\; \mathrm{d}s \geq0 \quad \text{for all} \ \ \varphi \in C^\infty_{\rm c}(0,L) \ \ \text{with} \ \ \varphi\geq0,  
\end{align}
where $L:=L[\gamma]$.
\end{lemma}
\begin{proof}
If $\gamma$ is a local minimizer of $\mathcal{E}_\lambda$ in $A$, then $\mathcal{E}_\lambda[\gamma]\leq \mathcal{E}_\lambda[\gamma_\varepsilon]$ holds for any perturbation $\{\gamma_\varepsilon\}_{\varepsilon\geq0}\subset A$ of $\gamma$ provided that all the perturbation functions $\gamma_\varepsilon$ are admissible. This can be achieved in particular by choosing perturbations in the positive vertical direction. Doing so, we find that $\gamma$ satisfies 
\begin{align}\label{eq:vari-ineq-soft}
    \frac{d}{d\varepsilon}\mathcal{E}_\lambda[\gamma+\varepsilon \eta e_2] \Big|_{\varepsilon=0}\geq0 \quad \text{for all} \ \ \eta \in C^\infty_{\rm c}(0,1) \ \ \text{with} \ \ \eta\geq0.
\end{align} 
By a straightforward calculation, we can prove that the above inequality is equivalent to 
\eqref{eq:vari-ineq-explict}. 
\end{proof}

Thus the variational inequality yields a necessary condition for a curve to be a local minimizer of $\mathcal{E}_\lambda$ in $A$.
We say that $\gamma \in A$ is a solution to the \textit{variational inequality} if $\gamma \in A$ satisfies \eqref{eq:vari-ineq-explict}.

\begin{remark}[Comparison with graph curves]\label{rem:vari-ineq}
Let us consider a curve represented by a graph of a function $\gamma_u(x):=(x,u(x))$ for some  $u\in W^{2,2}(0,1)\cap W^{1,2}_0(0,1)$ such that $u \geq \psi\vert_{[0,1]}$. One readily checks that $\gamma_u \in A$. Moreover, 
$\gamma_u + \varepsilon (v-u) e_2 \in A$ holds for any $\varepsilon \in [0,1]$ and  $v\in W^{2,2}(0,1)\cap W^{1,2}_0(0,1)$ with $v\geq \psi|_{[0,1]}$.  
Therefore, if $\gamma_u$ is a local minimizer of $B$ in $A$, then \eqref{eq:vari-ineq-soft} holds with the choice $\eta=v-u$. 
With this choice of $\eta$ we see that \eqref{eq:vari-ineq-soft} combined with $\lambda=0$ is reduced to
\[
\frac{d}{d\varepsilon}B[\gamma_u +\varepsilon \eta e_2]\Big|_{\varepsilon=0}=
\int_0^1\Big( 2\frac{u''}{(1+(u')^2)^{\frac{5}{2}}}(v-u)'' -5\frac{(u'')^2 u'}{(1+(u')^2)^{\frac{7}{2}}}(v-u)' \Big) \; \mathrm{d}x \geq0, 
\]
which is equivalent to the variational inequality in the graph case (see e.g.\ \cite{DD18}*{Equation 1.4}).
Thus our variational inequality \eqref{eq:vari-ineq-soft} can be regarded as an extension of the graph case $\gamma=\gamma_u$.

For   local minimizers among graph curves with $\lambda= 0$ it is already known that, (i) the 
regularity above the obstacle is 
$C^\infty$ (cf.\ \cite{DD18}*{Proposition~3.2}), (ii) the global regularity (i.e.\ regularity of the solution to the variational inequality in $(0,1)$) can be improved up to $W^{3,\infty}(0,1)$, which is optimal in the sense of the Sobolev class (cf.\ \cite{Miura21}*{Remark~3.7} or \cite{Ysima}*{Theorem~1.2}).
In our setting, the 
regularity above the obstacle 
is also $C^\infty$ as seen in Lemma~\ref{lem:E-L_eq_on_noncoincidence}.
On the other hand, improving the global regularity is more involved than in the graph case.
In fact, the leading term $\langle\boldsymbol{\kappa}, e_2\rangle\varphi''=k \langle\boldsymbol{n},e_2\rangle \varphi''$ in our variational inequality \eqref{eq:vari-ineq-explict} may degenerate if $\boldsymbol{n} =\pm e_1$ and such a degeneracy imposes an additional challenge to overcome in the regularity discussion.
\end{remark}
\subsection{Optimal regularity in the case of Lipschitz obstacles}\label{sec:3.2Regularity}

With a mild additional condition on the obstacle, the  difficulty displayed in Remark \ref{rem:vari-ineq} can be overcome and $W^{3,\infty}$-regularity (of the arclength parametrization of local minimizers of $\mathcal{E}_\lambda$ in $A$) can be deduced. To this end we first show a conditional regularity result for $\langle \boldsymbol{\kappa}, e_2 \rangle$ depending on the integrability of $|\boldsymbol{\kappa}|$.   

\begin{lemma}[Regularity via integrability]\label{lem:iteration}
    Let $\gamma \in A$ be such that the curvature vector $\boldsymbol{\kappa}$ of the arclength parametrization $\tilde{\gamma}$ satisfies the variational inequality \eqref{eq:vari-ineq-explict}. 
    Then one has for each $p \in [1,\infty]$
    \begin{equation}
        |\boldsymbol{\kappa}|^2 \in L_{loc}^p(0,L) \quad \Rightarrow \quad  \langle   \boldsymbol{\kappa}, e_2 \rangle \in W^{1,p}_{loc}(0,L). 
    \end{equation}
\end{lemma}
\begin{proof} Let $p \in [1,\infty]$ and $|\boldsymbol{\kappa}|^2 \in L^p_{loc}(0,L)$. 
Since $\boldsymbol{\kappa}$ satisfies \eqref{eq:vari-ineq-explict}, 
    by  \cite[Theorem 1.39]{EVGAR} there exists a Radon measure $\mu$ on $(0,L)$ such that 
\begin{equation}\label{eq:EvansGeriepyEquation}        \int_0^L \Big(2\langle \boldsymbol{\kappa},e_2\rangle \varphi''-3|\boldsymbol{\kappa}|^2\langle\boldsymbol{t},e_2\rangle\varphi'-\lambda \langle\boldsymbol{\kappa},e_2\rangle\varphi \Big)\; \mathrm{d}s  = \int_0^L \varphi \; \mathrm{d}\mu \qquad \textrm{for all } \varphi \in C_{\rm c}^\infty(0,L). 
    \end{equation}
    Now let $\delta \in (0, \frac{L}{2})$ be arbitrary. For each $\varphi \in C_{\rm c}^\infty(\delta, L-\delta)$ we have by Fubini's Theorem
    \begin{equation}
        \int_0^L \varphi(s) \; \mathrm{d}\mu(s) = \int_\delta^{L-\delta} \varphi(s) \; \mathrm{d}\mu(s) =  \int_\delta^{L-\delta} \int_\delta^s \varphi'(x) \; \mathrm{d}x \mathrm{d}\mu(s) = \int_\delta^{L-\delta} \mu((s,L-\delta)) \varphi'(s) \; \mathrm{d}s.
    \end{equation}
    Since $\mu$ is locally finite, we have that $f_\delta:(\delta,L-\delta) \rightarrow \mathbf{R};  y \mapsto \mu ((y,L-\delta))$ is a bounded function. As it is decreasing, it is also Borel measurable and hence lies in $L^\infty(\delta, L-\delta)$. We infer from this and \eqref{eq:EvansGeriepyEquation}  
    \begin{equation}
        \int_\delta^{L-\delta} \Big(2\langle \boldsymbol{\kappa},e_2\rangle \varphi''-3|\boldsymbol{\kappa}|^2\langle\boldsymbol{t},e_2\rangle\varphi'-\lambda \langle\boldsymbol{\kappa},e_2\rangle\varphi \Big)\; \mathrm{d}s  = \int_\delta^{L-\delta} f_\delta \varphi' \; \mathrm{d}s \qquad \textrm{for all } \varphi \in C_{\rm c}^\infty(\delta,L-\delta).
    \end{equation}
    Notice now that by assumption $g_\delta := 3|\boldsymbol{\kappa}|^2\langle\boldsymbol{t},e_2\rangle + f_\delta \in L^p(\delta, L-\delta)$ and 
   \begin{equation}\label{eq:DistributionalPoisson}
        \int_\delta^{L-\delta} \Big(2\langle \boldsymbol{\kappa},e_2\rangle \varphi''-g_\delta \varphi'-\lambda \langle\boldsymbol{\kappa},e_2\rangle\varphi \Big)\; \mathrm{d}s  = 0 \qquad \textrm{for all } \varphi \in C_{\rm c}^\infty(\delta,L-\delta).
    \end{equation}
    Choose now some fixed $\eta \in C_{\rm c}^\infty(\delta,L-\delta)$ such that $\int_\delta^{L-\delta} \eta = 1. $
    Now let $\zeta \in C_{\rm c}^\infty(\delta,L-\delta)$ be arbitrary. Define $c := \int_\delta^{L-\delta} \zeta(y) \; \mathrm{d}y$ and observe that also $\phi:= \int_\delta^\cdot \zeta- c \eta \in C_{\rm c}^\infty(\delta, L-\delta)$. Then \eqref{eq:DistributionalPoisson} yields 
    \begin{equation}
        \int_\delta^{L-\delta} \Big( 2 \langle \boldsymbol{\kappa},e_2\rangle (
        \zeta' - c \eta'') - g_\delta (
        \zeta - c \eta') - \lambda \langle \boldsymbol{\kappa},e_2\rangle \phi \Big) \; \mathrm{d}s = 0. 
    \end{equation}
    We rearrange to obtain 
    \begin{align}
        &\int_\delta^{L-\delta} 2 \langle \boldsymbol{\kappa}, e_2 \rangle \zeta' \; \mathrm{d}s  \\
        =& \int_\delta^{L-\delta} g_\delta \zeta \; \mathrm{d}s + c \int_\delta^{L-\delta} \Big( 2 \langle \boldsymbol{\kappa}, e_2 \rangle \eta'' - g_\delta \eta' - \lambda \langle \boldsymbol{\kappa}, e_2 \rangle \eta \Big) \; \mathrm{d}s + \int_{\delta}^{L-\delta} \lambda \langle \boldsymbol{\kappa}, e_2 \rangle \left( \int_\delta^{\cdot} \zeta \right) \; \mathrm{d}s. 
    \end{align}
    Since $|\langle \boldsymbol{\kappa}, e_2 \rangle| \leq |\boldsymbol{\kappa}|\leq \frac{1+ |\boldsymbol{\kappa}|^2}{2}$ we find that $\langle \boldsymbol{\kappa}, e_2 \rangle \in L^p(\delta,L-\delta)$. In particular $A := \int_\delta^{L-\delta} (2 \langle \boldsymbol{\kappa}, e_2 \rangle \eta'' - g_\delta \eta' - \lambda \langle \boldsymbol{\kappa}, e_2 \rangle \eta) \; \mathrm{d}s$ is a finite constant. Using this and also Fubini's Theorem in the last summand we find 
    \begin{equation}
        \int_\delta^{L-\delta} 2 \langle \boldsymbol{\kappa}, e_2 \rangle 
        \zeta'
        \; \mathrm{d}s = \int_\delta^{L-\delta} g_\delta 
         \zeta \; \mathrm{d}s + A c + \int_\delta^{L-\delta} \left( \lambda \int_{\cdot}^{L-\delta} \langle \boldsymbol{\kappa}, e_2 \rangle \right) \zeta \; \mathrm{d}s. 
    \end{equation}
    Recalling the definition of $c$ we have 
    \begin{equation}
         \int_\delta^{L-\delta} 2 \langle \boldsymbol{\kappa}, e_2 \rangle \zeta' \; \mathrm{d}s = \int_\delta^{L-\delta} \left( g_\delta + A + \lambda \int_{\cdot}^{L-\delta} \langle \boldsymbol{\kappa}, e_2 \rangle \right) \zeta \; \mathrm{d}s. 
    \end{equation}
    Noticing that $\big( g_\delta + A + \lambda \int_{\cdot}^{L-\delta} \langle \boldsymbol{\kappa}, e_2 \rangle \big)  \in L^p(\delta, L-\delta)$ and that  $\zeta \in C_{\rm c}^\infty(\delta, L-\delta)$ was arbitrary we obtain that $2 \langle \boldsymbol{\kappa}, e_2 \rangle \in W^{1,p}(\delta, L-\delta)$. The claim follows.  
\end{proof}

\begin{remark}\label{rem:3.5}
Notice that $|\boldsymbol{\kappa}|^2 \in L^1(0,L)$ is automatically satisfied for any $\gamma \in A$ as 
\begin{equation}
    \int_0^L |\boldsymbol{\kappa}|^2 \; \mathrm{d}s = B[\gamma] < \infty.
\end{equation}
Therefore Lemma \ref{lem:iteration} implies that  if
$\gamma$
solves the variational inequality, then $\langle\boldsymbol{\kappa}, e_2 \rangle \in W^{1,1}_{loc}(0,L)$. If $\gamma$ is additionally a local minimizer of $\mathcal{E}_\lambda$ in $A$, then we can actually deduce that $\langle \boldsymbol{\kappa}, e_2 \rangle \in W^{1,1}(0,L)$. Indeed, recall that $\langle \boldsymbol{\kappa}, e_2 \rangle = k \langle \boldsymbol{ n}, e_2 \rangle$ and in a neighborhood of $0$ and $L$ one has that  $k$ is analytic and $\boldsymbol{n} \in W^{1,2}(0,L)$, cf. Lemma \ref{lem:E-L_eq_on_noncoincidence}. Hence we may from now on use that $\langle \boldsymbol{\kappa}, e_2 \rangle  \in W^{1,1}(0,L)$. 
\end{remark}

 Next we intend to deduce Sobolev regularity for the  signed curvature $k$. To this end we  make use of the equation $k \langle \boldsymbol{n}, e_2\rangle = \langle \boldsymbol{\kappa}, e_2\rangle$. If $\langle \boldsymbol{n}, e_2 \rangle \neq 0$ one can solve for $k$ and obtain $k = \frac{\langle \boldsymbol{\kappa}, e_2\rangle}{\langle \boldsymbol{n}, e_2\rangle}$. Using this the Sobolev regularity of $\langle \boldsymbol{\kappa}, e_2\rangle$ inherits easily to $k$. Problems occur at points where $\langle \boldsymbol{n}, e_2 \rangle = 0$. If such points are on the noncoincidence set $\widetilde{N}_\gamma$, we do not have to bother because in this case regularity of $k$ is already obtained in Lemma \ref{lem:E-L_eq_on_noncoincidence}. So the only remaining problematic case is when $\langle \boldsymbol{n}(s_0), e_2 \rangle = 0$ for a point of coincidence $s_0$.  

 The next lemma rules out the existence of such points for Lipschitz continuous obstacles. 

 \begin{lemma}[Nonverticality on the coincidence set]\label{lem:nonveriticality}
 Suppose that $\psi$ is \textrm{Lipschitz continuous}, i.e. there exists some $K> 0$ such that $|\psi(x)- \psi(y)| \leq K |x-y|$ for all $x,y \in \mathbf{R}$. Moreover fix $\gamma \in A$ with arclength parametrization $\tilde{\gamma} : [0,L] \rightarrow \mathbf{R}^2$. Then  
 \begin{equation}
     \langle \boldsymbol{n}(s_0) ,e_2 \rangle \neq 0 \quad \textrm{for all} \ \  s_0 \in (0,L) \setminus \widetilde{N}_\gamma. 
 \end{equation}
 \end{lemma}
 \begin{proof}
     For a contradiction we assume the opposite, i.e., there exists $s_0 \in (0,L) \setminus \widetilde{N}_\gamma$ such that 
    \begin{equation}
        \langle\boldsymbol{n}(s_0),e_2\rangle=0.
    \end{equation}
    We infer that $(\tilde{\gamma}^1)'(s_0) =  0$ 
    and $|(\tilde{\gamma}^2)'(s_0)| =1$. Using that $\tilde{\gamma}^2(s) \geq \psi(\tilde{\gamma}^1(s))$ for each $s \in (0,L)$ and $\tilde{\gamma}^2(s_0)= \psi(\tilde{\gamma}^1(s_0))$ we find 
    \begin{equation}
        \tilde{\gamma}^2(s) - \tilde{\gamma}^2(s_0) \geq \psi(\tilde{\gamma}^1(s)) - \psi(\tilde{\gamma}^1(s_0)) \qquad \forall s \in (0,L). 
    \end{equation}
    We can now estimate the right-hand side by $-K |\tilde{\gamma}^1(s)- \tilde{\gamma}^1(s_0)|$ and divide by $|s-s_0|$ to obtain
     \begin{equation}
       \frac{ \tilde{\gamma}^2(s) - \tilde{\gamma}^2(s_0) }{|s-s_0|} \geq -K \left\vert \frac{\tilde{\gamma}^1(s)- \tilde{\gamma}^1(s_0)}{s-s_0} \right\vert  \qquad \forall s \in (0,L) \setminus \{s_0\}.
    \end{equation}
    Letting $s \rightarrow s_0+$ and using that $(\tilde{\gamma}^1)'(s_0)=0$ we deduce that $(\tilde{\gamma}^2)'(s_0) \geq 0$. Furthermore, letting $s \rightarrow s_0-$ and using that $(\tilde{\gamma}^1)'(s_0)=0$ we obtain $-(\tilde{\gamma}^2)'(s_0) \geq 0$. These observations imply $(\tilde{\gamma}^2)'(s_0) = 0$, 
    a contradiction. The claim follows.
 \end{proof}

 The previous results finally yield (optimal) $W^{1,\infty}$-regularity for $k$.  
 
 

\begin{lemma}[Curvature regularity of local minimizers in $A$] Suppose that $\psi$ is Lipschitz continuous. Further let $\gamma \in A$ be a local minimizer of $\mathcal{E}_\lambda$ in $A$. Then $k \in W^{1,\infty}(0,L)$. 
\end{lemma}
\begin{proof} Use again the arclength parametrization $\tilde{\gamma}$. Note from Remark \ref{rem:3.5} that $\langle \boldsymbol{\kappa}, e_2 \rangle$ lies in $W^{1,1}(0,L).$

\textbf{Step 1.}  We show $k \in W^{1,1}_{loc}(0,L)$.
Now fix some point $s_0 \in (0,L)$. If $s_0 \in \widetilde{N}_\gamma$ then we infer from Lemma \ref{lem:E-L_eq_on_noncoincidence} that $k$ is analytic (and hence also $W^{1,1}$) in a neighborhood of $s_0$. Now suppose that $s_0 \in (0,L) \setminus \widetilde{N}_\gamma$. By Lemma \ref{lem:nonveriticality} we have $\langle \boldsymbol{n}(s_0), e_2\rangle \neq 0$. 
We deduce from this (since $\boldsymbol{n} \in W^{1,2}(0,L) \subset C([0,L])$) that in an open neighborhood $U$ of $s_0$ we also have that  $\langle \boldsymbol{n}, e_2 \rangle \neq 0$. Remark \ref{rem:3.5} now yields that  
  $
        k \langle \boldsymbol{n}, e_2 \rangle = \langle \boldsymbol{\kappa}, e_2 \rangle \in W^{1,1}(U).
 $ 
    Hence we find  
    \begin{equation}
        k = \frac{\langle \boldsymbol{\kappa}, e_2 \rangle}{\langle \boldsymbol{n}, e_2 \rangle} \in W^{1,1}_{loc}(U). 
    \end{equation}
    We obtain that for each point $s_0\in (0,L)$ there exists an open neighborhood $U$ of $s_0$ such that $k \in W^{1,1}(U)$.  Step 1 is complete. 
    
     \textbf{Step 2.} $k \in W^{1,\infty}_{loc}(U)$. 
    Due to the one-dimensionality of $(0,L)$ we have $W^{1,1}_{loc}(0,L) \subset C(0,L) \subset L^\infty_{loc}(0,L)$ and hence $k \in L^\infty_{loc}(0,L)$. Now note that also 
    \begin{equation}
        |\boldsymbol{\kappa}|^2 = k^2  \in L^\infty_{loc}(0,L). 
    \end{equation}
    Thereupon, Lemma \ref{lem:iteration} yields that 
    \begin{equation}\label{eq:kappa2W1infty}
        \langle \boldsymbol{\kappa}, e_2 \rangle \in W^{1,\infty}_{loc}(0,L).
    \end{equation}
    We claim now that $k \in W^{1,\infty}_{loc}(0,L)$. Indeed, let $s_0 \in (0,L).$ If $s_0 \in \widetilde{N}_\gamma$, then clearly there exists an open neighborhood $U$ of $s_0$ such that $k$ is analytic on $U$ and hence also $k \in W^{1,\infty}_{loc}(U)$. If $s_0 \in (0,L) \setminus \widetilde{N}_\gamma$, then we infer again from Lemma \ref{lem:nonveriticality} that there exists an open neighborhood $U$ of $s_0$ such  $\langle \boldsymbol{n}, e_2 \rangle \neq 0$ on $U$. Notice further that $\langle \boldsymbol{n}, e_2 \rangle \in W^{1,\infty}_{loc}(U)$ as $\partial_s \langle \boldsymbol{n}, e_2 \rangle = \langle \boldsymbol{n}', e_2 \rangle = - \langle k \boldsymbol{t}, e_2 \rangle \in L^\infty_{loc}(U)$ due to the previously obtained  $L^\infty_{loc}$-regularity of $k$ and the fact that $|\boldsymbol{t}|=|e_2|=1$. Using \eqref{eq:kappa2W1infty} we find also that 
    \begin{equation}
        k \langle \boldsymbol{n}, e_2 \rangle = \langle \boldsymbol{\kappa}, e_2 \rangle \in W^{1,\infty}_{loc}(U)
    \end{equation}
    and thus $k = \frac{\langle \boldsymbol{\kappa}, e_2 \rangle}{\langle \boldsymbol{n}, e_2 \rangle} \in W^{1,\infty}_{loc}(U)$. 
    Using now the fact that $k$ is analytic in a neighborhood of $0$ and $L$ (cf.\ Lemma~\ref{lem:E-L_eq_on_noncoincidence}), we find that $k \in W^{1,\infty}(0,L)$.
 \end{proof}

As a corollary we deduce 
$W^{3,\infty}$-regularity for the arclength parametrization $\tilde{\gamma}$ of local minimizers. 

\begin{corollary}\label{cor:optimal_regularity} Suppose that $\psi$ is Lipschitz continuous. 
    Let $\gamma \in A$ be a local minimizer of $\mathcal{E}_\lambda$ in $A$.  Then $\tilde{\gamma} \in W^{3,\infty}(0,L)$. 
\end{corollary}
\begin{proof}
    The previous lemma shows that $k \in W^{1,\infty}(0,L)$. Notice that on $(0,L)$ we have
    \begin{equation}\label{eq:gamma''intermsofkappa}
        \tilde{\gamma}'' = \boldsymbol{\kappa} = k \boldsymbol{n}. 
    \end{equation}
    Now $k \in W^{1,\infty}(0,L)$ by the previous lemma. Furthermore $\boldsymbol{n} \in W^{1,\infty}(0,L)$ as $\boldsymbol{n}' = - k \boldsymbol{t} \in L^\infty(0,L)$. Therefore \eqref{eq:gamma''intermsofkappa} yields that $\tilde{\gamma}'' \in W^{1,\infty}(0,L)$. The claim follows. 
\end{proof}

\begin{remark}[Optimality]
    Notice that the $W^{3,\infty}$-regularity obtained in the previous corollary is optimal. Indeed, the minimizers of $\mathcal{E}_0$ in $A_{\rm graph}$ found in \cite[Theorem 3.1]{Miura21}
    are local minimizers for the minimization problem in $A$. 
    As discussed in Remark~\ref{rem:vari-ineq}, their regularity can not be improved further in the Sobolev classes, yielding the claimed optimality. We remark that \cite[Theorem 3.1]{Miura21} actually looks only at minimizers in $A_{\rm graph} \cap A_{\rm sym}$, but the symmetry assumption is not required for small obstacles, cf.\ \cite[Theorem 2.9]{Mueller21}.
\end{remark}

\subsection{Existence of global minimizers}

If $\lambda>0$, then existence of minimizers of $\mathcal{E}_\lambda$ in $A$ follows from the
standard direct method. 

\begin{lemma}\label{lem:globalminimizer_existience}
Let $\lambda>0$.
Then, there exists a solution to the following minimization problem 
\[
\min_{\gamma \in 
A } \mathcal{E}_\lambda[\gamma].    
\]
\end{lemma}
\begin{proof}
Let $\{\gamma_j\}_{j\in \mathbf{N}}\subset A$ be a minimizing sequence of $\mathcal{E}_\lambda$ in $A$, i.e., 
\begin{align}\label{eq:min_seq}
\lim_{j\to\infty}\mathcal{E}_\lambda[\gamma_j] = \inf_{\gamma \in
 A} \mathcal{E}_\lambda[\gamma].
\end{align}
In particular, this ensures that there exists $C>0$ such that $\mathcal{E}_\lambda[\gamma_j] \leq C$ for all $j \in \mathbf{N}$.
After reparametrization, we may suppose that $\gamma_j \in A$ is of constant speed so that $|\gamma_j'|\equiv L[\gamma_j]$. 
The assumption of 
constant speed
implies that 
\[
\|\gamma_j''\|_{L^2(0,1)}^2 = \int_0^{L[\gamma_j]}\big(L[\gamma_j]^2 |\tilde{\gamma}_j''(s)|^2 \big)L[\gamma_j]^{-1}\; \mathrm{d}s
=L[\gamma_j]^3 B[\gamma_j],
\]
where $\tilde{\gamma}_j$ denotes the arclength parametrization of $\gamma_j$.
This together with $\mathcal{E}_\lambda[\gamma_j] \leq C$ yields  
a uniform estimate of $\|\gamma_j''\|_{L^2(0,1)}$.
Using $|\gamma_j'|\equiv L[\gamma_j] \leq C$ and the boundary condition $\gamma_j(0)=(0,0)$, we also obtain the bounds on the $W^{1,2}$ norm.
Therefore, $\{\gamma_j\}_{j\in \mathbf{N}}$ is uniformly bounded in $W^{2,2}(0,1;\mathbf{R}^2)$ so that there are $\bar{\gamma} \in W^{2,2}(0,1;\mathbf{R}^2)$ and a subsequence (without relabeling) such that $\gamma_j$ converges to $\bar{\gamma}$ in the senses of $W^{2,2}$-weak and $C^1$ topology.
Thus the limit curve $\bar{\gamma}$ also satisfies $\bar{\gamma}(0)=(0,0)$, $\bar{\gamma}(1)=(1,0)$, and $\bar{\gamma}_2(t) \geq \psi(\bar{\gamma}_1(t))$ for any $t\in[0,1]$.
In addition, by the boundary condition $L[\gamma_j]\geq 
1
$ holds for all $j\in \mathbf{N}$, 
so that
\begin{equation}
    |\bar{\gamma}'| = \lim_{j \rightarrow \infty} |\gamma_j'| = \lim_{j \rightarrow \infty} L[\gamma_j] = L[\bar{\gamma}] \in (0,\infty),
\end{equation}
i.e., $\bar{\gamma}$ is immersed. 
Therefore, $\bar{\gamma} \in A$. It remains to show that $\bar{\gamma}$ minimizes the energy. 
 By the weak lower semicontinuity of the $L^2$ norm it follows that
\[
B[\bar{\gamma}] = L[\bar{\gamma}]^{-3} \|\bar{\gamma}''\|_{L^2(0,1)}^2 \leq \liminf_{j \to \infty} L[\gamma_j]^{-3} \|\gamma_j'' \|_{L^2(0,1)}^2 = \liminf_{j\to\infty}B[\gamma_j]. 
\]
Combining this with \eqref{eq:min_seq}, we see that
\[
\mathcal{E_\lambda [\bar{\gamma}]}
=B[\bar{\gamma}] + \lambda L[\bar{\gamma}]
\leq \liminf_{j \to \infty} B[\gamma_j] + \lambda \lim_{j \to \infty} L[\gamma_j]
\leq  \liminf_{j \to \infty} \mathcal{E}_\lambda[\gamma_j]
= \inf_{\gamma \in 
A
} \mathcal{E}_\lambda[\gamma].
\]
The proof is complete.
\end{proof}

On the other hand, if $\lambda=0$, then there is no global minimizer of $\mathcal{E}_0=B$ in $A$. 

\begin{lemma}\label{lem:3.10}
Let $\psi$ be an obstacle 
satisfying assumptions (A1)--(A3). 
Then there exists no global minimizer of $B$ in $A$ nor in $A_{\rm sym}$. 
\end{lemma}
\begin{proof}
Assume that there exists a minimizer $\gamma$ of $B$ in $A$. 
By assumption (A3) $\gamma$ is not a line segment, accordingly $B[\gamma]>0$. 
On the other hand, let $C_n$ be a circular arc with radius $1/n$ in the upper half-plane passing through points $(0,0)$ and $(1,0)$.
Then clearly one has $C_n^2>\psi\circ C_n^1$ and $B[C_n] \leq \frac{2\pi}{n} <B[\gamma]$ for sufficiently large values of $n\in \mathbf{N}$.
This contradicts the minimality of $\gamma$ in $A$. 
In addition, since $C_n$ also belongs to $A_{\rm sym}$, one can also check that there is no global minimizer in $A_{\rm sym}$. 
\end{proof}

\section{Analysis of the case of $\lambda = 0$}\label{eq:sec4}

Let us begin this section by introducing the following 
competitors whose (local) minimality properties we will investigate in this section:
\begin{definition}[Symmetric cut-and-glued free-elastica]\label{def:SCF}
We say that a planar curve $\gamma \in A$ is a \textit{symmetric cut-and-glued free-elastica} (for short SCF) if $\gamma$ satisfies the following conditions: 
\begin{align}\label{eq:def-SCF}\tag{SCF}
\begin{cases}
\text{(i) } \gamma \text{ is reflectionally symmetric}; \\
\text{(ii) } 2k'' + k^3=0 \ \text{ on } \ (0,\frac{L}{2}); \\
\text{(iii) } k(0)=0; \\
\text{(iv) } k< 0 \ \text{ on } \ (0,\frac{L}{2}); \\
\text{(v) } \tilde{\gamma}(\frac{L}{2})=(\frac{1}{2}, \psi(\frac{1}{2})), 
\end{cases} 
\end{align}
where $L=L[\gamma]$.
\end{definition}
Recall e.g.\ from \cite{Mandel} that for $s \in (0,\frac{L}{2})$ we have 
$\tilde{\gamma}(s) = \frac{1}{\alpha} R_\phi \gamma_{\rm rect}(\alpha s)$ for some $\alpha > 0$ and $\phi \in [0,2\pi)$. 
\begin{remark}
    Let $\psi$ be a symmetric cone obstacle and $\gamma \in A$ be an SCF. We claim that then $\widetilde{I}_\gamma=\{\frac{L}{2}\}$. Indeed, assume that there exists some $s_0 \in \widetilde{I}_\gamma \setminus \{\frac{L}{2}\}$. Due to reflection symmetry we may assume  $s_0 < \tfrac{L}{2}$. We have $\tilde{\gamma}(s_0)  = (a,\psi(a))$ for some $a \in \mathbf{R}$. 
    As mentioned earlier, there exist $\alpha > 0$ and $\phi \in [0,2\pi)$ such that $\tilde{\gamma}(s) = \frac{1}{\alpha} R_\phi \gamma_{\rm rect}(\alpha s), s \in (0, \frac{L}{2})$. Notice that \eqref{eq:def-SCF}-(iv) implies that $\tfrac{L}{2} \leq 2\mathrm{K}(\tfrac{1}{\sqrt{2}})/(\alpha\alpha_0)$, as $k_{\rm rect}$ changes sign at 
    $2\mathrm{K}(\tfrac{1}{\sqrt{2}})/\alpha_0$ (cf.\ \eqref{eq:kappa_rect}).
    Moreover, $a \neq \frac{1}{2}$ since $\tilde{\gamma} \vert_{[0,\frac{L}{2}]}$ must be embedded (due to the fact that $\gamma_{\rm rect}$ is embedded, as one readily checks observing from \eqref{eq:rect-formula} that $(\gamma_{\rm rect}^1)'>0$ on $(0, \frac{L_{\rm rect}}{2}$)). 
    Also $a>\frac{1}{2}$ is ruled out by the explicit shape of $\gamma_{\rm rect}$. 
    Thus we may assume that $a<\frac{1}{2}$, and this implies that $\psi'(a)>0$.
    Notice that $\psi$ and $\tilde{\gamma}$ are twice differentiable at $a$ and $\psi'$ is constant in a neighborhood of $a$. This and the fact that $\tilde{\gamma}^2 - \psi \circ \tilde{\gamma}^1$ attains a local minimum at $s_0$ imply
    \begin{align}\label{eq:0404-1}
         \tilde{\gamma}^2(s_0) - \psi(\tilde{\gamma}^1(s_0))  = 0, \quad
         (\tilde{\gamma}^2)'(s_0) - \psi'(a) (\tilde{\gamma}^1)'(s_0)  = 0, \quad
         (\tilde{\gamma}^2)''(s_0) - \psi'(a) (\tilde{\gamma}^1)''(s_0)  \geq 0.
    \end{align}
Using this we have 
\begin{equation}
    k(s_0) = \langle \boldsymbol{\kappa},\boldsymbol{n} \rangle(s_0) = (\tilde{\gamma}^2)''(s_0) (\tilde{\gamma}^1)'(s_0) - (\tilde{\gamma}^1)''(s_0) (\tilde{\gamma}^2)'(s_0)  = (\tilde{\gamma}^1)'(s_0) [ (\tilde{\gamma}^2)''(s_0) - \psi'(a) (\tilde{\gamma}^1)''(s_0)].
\end{equation}
In particular, since $k(s_0)<0$ by \eqref{eq:def-SCF}-(iv) and $ (\tilde{\gamma}^2)''(s_0) - \psi'(a) (\tilde{\gamma}^1)''(s_0) \geq0$ by \eqref{eq:0404-1}, it follows that $(\tilde{\gamma}^1)'(s_0)<0$. 
This and \eqref{eq:0404-1} also imply that $(\tilde{\gamma}^2)'(s_0)=\psi'(a) (\tilde{\gamma}^1)'(s_0)<0$. 
Therefore, $\tilde{\gamma}'(s_0)$ lies in the third quadrant. 
In addition, since $\gamma_{\rm rect}'(s)$ lies in the first and fourth quadrant for any $s\in [0, 2\mathrm{K}(\frac{1}{\sqrt{2}})/\alpha_0]$, we see that $\phi > \frac{\pi}{2}$. Moreover, $(\tfrac{1}{2},\psi(\tfrac{1}{2})) = \tfrac{1}{\alpha} R_\phi (\gamma_{\rm rect}(\tfrac{\alpha L}{2}))$ and the fact that $\gamma_{\rm rect}(\tfrac{\alpha L}{2})$ lies in the first quadrant yields that $\phi \in (0,\tfrac{\pi}{2}) \cup(\frac{3\pi}{2},2\pi)$. Since $(0,\frac{\pi}{2})$ is already ruled out we only have to consider the case $(\frac{3\pi}{2},2\pi)$. Curves that satisfy \eqref{eq:def-SCF} with the restriction of $\phi \in (\frac{3\pi}{2},2\pi)$ have been studied in \cite[Equation 3.13]{Miura21}. One obtains a unique graphical shape that touches the obstacle only once at the tip, as discussed in \cite[Equation 3.10]{Miura21}. A contradiction, and hence we find that $\widetilde{I}_\gamma=\{\frac{L}{2}\}$ if $\gamma$ is an SCF. 
\end{remark}

It is already shown that if $\psi$ is a symmetric cone and satisfies $\psi(\frac{1}{2})<h_*$, where $h_*$ is the constant given by \eqref{def:h_*}, 
then there exists a unique minimizer of $\mathcal{E}_0=B$ in  $A_{\rm graph} \cap A_{\rm sym}$ (see e.g.\ \cite{Miura21}), and this 
unique minimizer is characterized as a symmetric cut-and-glued free-elastica. 
On the other hand, if $\psi(\frac{1}{2})\geq h_*$, then there  exists neither a symmetric cut-and-glued free-elastica nor a global minimizer of $B$ in  $A_{\rm graph}$. 

In the following we show that there exists a unique cut-and-glued free-elastica in the class $A$, not $A_{\rm graph}$. 
To this end, let us introduce the polar tangential angle of $\gamma_{\rm rect}$, defined by \eqref{eq:rect-formula}. 
Here, for an arclength parametrized smooth planar curve $\gamma:[0,L]\to\mathbf{R}^2$ with $\gamma(s)\neq0$ for all $s\in(0,L)$, the \textit{polar tangential angle} function $\omega:(0,L) \to \mathbf{R}$ is defined as a smooth function such that 
\begin{align}
R_{\omega(s)}\Big(\frac{\gamma(s)}{|\gamma(s)|}\Big)=\boldsymbol{t}(s) \ \ \text{for} \ s\in(0,L),   
\end{align}
where $L=L[\gamma]$. 
Such a polar tangential angle function exists, is unique up to shifts in $2\pi\mathbf{Z}$  and is scale invariant in the following sense: For $\alpha>0$ and $\phi\in[0,2\pi)$, let $\omega_{\alpha, \phi}$ be the polar tangential angle of $s\mapsto \frac{1}{\alpha}R_\phi \gamma(\alpha s)$. 
Then, for each $s\in[0,L]$, 
\begin{align}\label{eq:omega-scale-invariant}
    \omega_{\alpha,\phi}(s)= \omega(\alpha s) \quad (\text{mod } 2\pi\mathbf{Z}). 
\end{align}

\begin{lemma}\label{lem:polar-rectangular}
Let $\omega_{\rm rect}:(0,L) \to \mathbf{R}$ be the polar tangential angle of $\gamma_{\rm rect}$, where $L=L[\gamma_{\rm rect}]$. 
Then $\omega_{\rm rect}$ is strictly decreasing in $(0,L)$. 
In addition, 
\begin{align}\label{eq:values-polar-rect}
\lim_{s \downarrow 0}\tan\big( -\omega_{\rm rect}(s) \big) =0, \quad 
\tan\big( -\omega_{\rm rect}(\tfrac{L}{2})\big) = 2h_*, \quad 
\lim_{s \uparrow L}\tan\big( -\omega_{\rm rect}(s) \big) = \infty.
\end{align}
\end{lemma}
\begin{proof}
First we show monotonicity of $\omega_{\rm rect}$. 
Since monotonicity of $\omega_{\rm rect}|_{(0,\frac{L}{2}]}$ is already shown in \cite{Miura21}, it suffices to prove monotonicity of $\omega_{\rm rect}|_{[\frac{L}{2},L)}$.
 Recall from \cite{Miura21}*{Corollary 2.5} that the sign of $k_{\rm rect}(s)\omega_{\rm rect}'(s)$ coincides with that of $\langle \gamma_{\rm rect}(s), \epsilon(s) \rangle$, where $\epsilon = \gamma + k^{-1}\boldsymbol{n}$ denotes the locus of the centers of osculating circles as in \cite{Miura21}*{Section 2.1}. Since $k_{\rm rect}<0$ on $[\frac{L}{2},L)$ (cf.\ \eqref{eq:kappa_rect}), to prove that $\omega_{\rm rect}$ is strictly decreasing, it suffices to show that 
\begin{align}\label{eq:polar-rect-goal}
\langle \gamma_{\rm rect}(s), \epsilon(s) \rangle =
|\gamma_{\rm rect}(s)|^2 + \frac{\langle\gamma_{\rm rect}(s), \boldsymbol{n}(s)\rangle}{k_{\rm rect}(s)} >0 \quad \text{for all} \quad s\in( \tfrac{L}{2},L). 
\end{align}
Set
\[
\xi(s):= \frac{1}{2}|\gamma_{\rm rect}(s)|^2 + \frac{\langle\gamma_{\rm rect}(s), \boldsymbol{n}(s)\rangle}{k_{\rm rect}(s)} \quad \text{for} \quad s\in( \tfrac{L}{2},L).
\]
Then we have $\langle \gamma_{\rm rect}(s), \epsilon(s) \rangle > \xi(s)$ for any $s\in( \tfrac{L}{2},L)$, and by using the Frenet--Serret formula $\boldsymbol{n}'=-k \boldsymbol{t}$ and the fact 
$\langle \boldsymbol{t}, \boldsymbol{n}\rangle=0$
we can compute 
\[
\xi'(s)=-\frac{\langle\gamma_{\rm rect}(s), \boldsymbol{n}(s)\rangle k_{\rm rect}'(s)}{k_{\rm rect}(s)^2}.
\]
By formula \eqref{eq:rect-formula} we have $\gamma_{\rm rect}^i(s)>0$ and $\boldsymbol{n}^i(s)>0$ for any $s\in (\frac{L}{2},L)$ and $i=1,2$, 
which leads to 
$\langle\gamma_{\rm rect}(s), \boldsymbol{n}(s)\rangle>0$
for any $s\in (\frac{L}{2},L)$.
In addition, it follows from \eqref{eq:kappa_rect} that $k'_{\rm rect}>0$ in $(\frac{L}{2},L]$. 
Thus $\xi$ is strictly decreasing in $(\frac{L}{2},L)$, and hence, for any $s\in(\frac{L}{2},L)$, 
\[
\xi(s) \geq \lim_{s\uparrow L} \xi(s) = \frac{1}{2} + \lim_{s\uparrow L} \frac{\langle \gamma_{\rm rect}(s), -k_{\rm rect}(s)\boldsymbol{t}(s)\rangle}{k_{\rm rect}'(s)}=\frac{1}{2},
\]
where we have used l'Hospital's rule in the penultimate step.
Thus we have $\langle \gamma_{\rm rect}(s), \epsilon(s)\rangle>\frac{1}{2}$ for any $s\in(\frac{L}{2},L)$, and hence the desired monotonicity of $\omega$ follows.

Next we show \eqref{eq:values-polar-rect}. 
One can deduce from \eqref{eq:rect-formula} that $\gamma_{\rm rect}(s)=:(X(s),Y(s))$ satisfies $\lim_{s\downarrow0}\frac{X(s)}{Y(s)}=\lim_{s\downarrow0}\frac{X'(s)}{Y'(s)}=0$, which in combination with the fact that $Y>0$ on $(\frac{L}{2},L)$ gives
\[
\frac{\gamma_{\rm rect}(s)}{|\gamma_{\rm rect}(s)|}=\Big(\frac{X(s)}{\sqrt{X(s)^2+Y(s)^2}}, \frac{Y(s)}{\sqrt{X(s)^2+Y(s)^2}} \Big) \to 1 \quad \text{as} \ s\to0.
\]
On the other hand, $\boldsymbol{t}(0)=e_2$ also follows from \eqref{eq:rect-formula}, and hence we obtain $\lim_{s\downarrow0}\cos(-\omega_{\rm rect}(s)) =1$, so that $\lim_{s\downarrow0}\tan(-\omega_{\rm rect}(s)) =0$.
The second equality of \eqref{eq:values-polar-rect} immediately follows from \eqref{eq:rect-height} combined with the fact that $\boldsymbol{t}(\frac{L}{2})=e_1$. 
Finally, by \eqref{eq:rect-height} we have $\gamma_{\rm rect}(L)=(1,0)$ and $\boldsymbol{t}(L)=(0,-1)$, which leads to $\lim_{s\uparrow L}\tan(-\omega_{\rm rect}(s)) =\infty$.
\end{proof}

We apply Lemma~\ref{lem:polar-rectangular} to show existence and uniqueness of symmetric cut-and-glued free-elasticae. 

\begin{lemma}[Unique existence of symmetric cut-and-glued free-elasticae]\label{lem:exsitence-SCF}
Let $\psi$ be a symmetric cone obstacle 
satisfying (A1)--(A3). 
Then there exists a unique symmetric cut-and-glued free-elastica 
$\gamma \in A$. 
In addition, $\gamma$ satisfies the following properties: 
\begin{itemize}
    \item[(i)] if $\psi(\frac{1}{2})<h_*$, then $\gamma \in A_{\rm graph}$ and $\lim_{s \uparrow \frac{L}{2}} k'(s) < 0$; 
    \item[(ii)] if $\psi(\frac{1}{2})=h_*$, then $\gamma = \gamma_{\rm rect}$ (up to reparametrization). In particular,  $k'(\tfrac{L}{2})=0$ and $\gamma \not\in A_{\rm graph}$; 
    \item[(iii)] if $\psi(\frac{1}{2})>h_*$, then $\gamma \not\in A_{\rm graph}$ and the signed curvature $k$ satisfies $\lim_{s\uparrow \frac{L}{2}}k'(s) >0$. 
\end{itemize}
Here $k$ denotes the signed curvature of $\gamma$.
\end{lemma} 

\begin{proof}
\textbf{Step 1} (\textsl{We prove existence of symmetric cut-and-glued free-elasticae})\textbf{.}
By Lemma~\ref{lem:polar-rectangular}, there exists $s_*\in(0,L[\gamma_{\rm rect}])$ such that $\tan{(-\omega_{\rm rect}}(s_*))=2\psi(\frac{1}{2})$. 
We also choose $\phi\in(-\frac{\pi}{2}, \frac{\pi}{2})$ such that $R_\phi\gamma_{\rm rect}'(s_*)=e_1$. 
In addition, let $\alpha>0$ be a scaling factor to satisfy
\begin{align}\label{eq:construct-SCF}
    \frac{1}{\alpha}R_\phi\gamma_{\rm rect}(\alpha s_*)=(\tfrac{1}{2},\psi(\tfrac{1}{2})).
\end{align} 
Finally define $\gamma_*:[0,s_*]\to\mathbf{R}^2$ by $\gamma_*(s):=\frac{1}{\alpha}R_\phi\gamma_{\rm rect}(\alpha s)$, and  
consider the reflectionally symmetric constant-speed parametrized curve $\gamma:[0,1]\to\mathbf{R}^2$ uniquely given by $\gamma|_{[0,1/2]}=\gamma_*$ (up to reparametrization).
Then, property \eqref{eq:def-SCF}-(i) in Definition \ref{def:SCF} immediately follows and properties \eqref{eq:def-SCF}-(ii)--(iv) are inherited from
$\gamma_{\rm rect}$. 
In addition, \eqref{eq:construct-SCF} means that $\gamma$ satisfies property  \eqref{eq:def-SCF}-(v) of Definition \ref{def:SCF}, and hence $\gamma$ is a symmetric cut-and-glued free-elasticae. 

\textbf{Step 2} (\textsl{We prove uniqueness})\textbf{.}
If $\gamma$ is a symmetric cut-and-glued free-elastica, then combining \eqref{eq:def-SCF}-(ii) with \cite[Proposition 3.3]{Lin96}, we see that the signed curvature $k$ of $\gamma$ is given by 
\begin{align}\label{eq:241227-1}
    k(s)=-\sqrt{2}\alpha\cn(\alpha s +\beta,\tfrac{1}{\sqrt{2}}) \quad s\in[0,\tfrac{L}{2}]
\end{align}
for some $\alpha\geq0$ and $\beta\in\mathbf{R}$.
In addition, by \eqref{eq:def-SCF}-(iii) and (iv), we may choose $\beta=\mathrm{K}(\frac{1}{\sqrt{2}})$.
Note that $\alpha=0$ is ruled out by symmetry of $\gamma$ and \eqref{eq:def-SCF}-(iv). 
Thus we have 
\begin{align}\label{eq:kappa_formula_half}
    k(s)=\alpha k_{\rm rect}(\alpha s), \quad s\in [0,\tfrac{L}{2}], 
\end{align} 
which implies that, for some $\phi\in [0,2\pi)$, 
\begin{align}\label{eq:SCF-unique-formula}
\tilde{\gamma}(s)=\frac{1}{\alpha}R_\phi\gamma_{\rm rect}(\alpha s), \quad s\in[0,\tfrac{L}{2}]. 
\end{align}
Let $\omega:(0,L)\to \mathbf{R}$ and $\omega_{\rm rect}:(0,L_{\rm rect})\to \mathbf{R}$ denote the polar tangential angle of $\tilde{\gamma}$ and $\gamma_{\rm rect}$, respectively, where $L=L[\gamma]$ and $L_{\rm rect}:=L[\gamma_{\rm rect}]$.
Note that $\tan{(-\omega(\frac{L}{2}))}=\tan{(-\omega_{\rm rect}(\frac{\alpha L}{2}))}$ follows from \eqref{eq:SCF-unique-formula} and the scale
invariance of $\omega_{\rm rect}$ (cf.\ \eqref{eq:omega-scale-invariant}). 
By \eqref{eq:def-SCF}-(v), we have $\tan{(-\omega(\tfrac{L}{2}) )}=2\psi(\tfrac{1}{2})$.
On the other hand, by Lemma~\ref{lem:polar-rectangular} we can find an $\ell\in(0,L_{\rm rect})$ such that 
\begin{align}\label{eq:ell-psi-1to1}
    \tan{\big(-\omega_{\rm rect}(\ell) \big)}=2\psi(\tfrac{1}{2}). 
\end{align}
Note that this $\ell\in(0,L_{\rm rect})$ is uniquely determined by $\psi(\frac{1}{2})$.
This together with $2\psi(\tfrac{1}{2})=\tan{(-\omega(\frac{L}{2}))}=\tan{(-\omega_{\rm rect}(\frac{\alpha L}{2}))}$ yields 
\begin{align}\label{eq:SCF-unique-formula-alpha-L}
\frac{\alpha L}{2}=\ell.
\end{align}
Combining this with \eqref{eq:def-SCF}-(i) and \eqref{eq:SCF-unique-formula}, we deduce that 
\begin{align}\label{eq:SCF-unique-formula-phi}
e_1 = \tilde{\gamma}'(\tfrac{L}{2}) = R_\phi \gamma_{\rm rect}'(\tfrac{\alpha L}{2}) = R_\phi \gamma_{\rm rect}'(\ell). 
\end{align}
Moreover, \eqref{eq:def-SCF}-(v) and \eqref{eq:SCF-unique-formula} yield
\begin{align}\label{eq:SCF-unique-formula-alpha}
(\tfrac{1}{2}, \psi(\tfrac{1}{2}))=\tilde{\gamma}(\tfrac{L}{2})=\frac{1}{\alpha} R_{\phi}\gamma_{\rm rect}(\ell).
\end{align}
Since $\ell\in(0,L_{\rm rect})$ is unique, we observe from \eqref{eq:SCF-unique-formula-alpha-L}--\eqref{eq:SCF-unique-formula-alpha} that $\alpha>0$, $\phi\in[0,2\pi)$, and $L>0$ in \eqref{eq:SCF-unique-formula} are uniquely determined by $\psi(\frac{1}{2})$. 
Therefore if a symmetric cut-and-glued free-elastica exists, then it is unique. 

\textbf{Step 3} (\textsl{We deduce the claimed properties of symmetric cut-and-glued free-elasticae})\textbf{.}
Let $\gamma$ be a symmetric cut-and-glued free-elastica. 
Recall from Step~2 that the arclength parametrization $\tilde{\gamma}$ is given by \eqref{eq:SCF-unique-formula} with  $\alpha>0$ and $\phi\in [0,2\pi)$ satisfying \eqref{eq:SCF-unique-formula-alpha-L}--\eqref{eq:SCF-unique-formula-alpha}.
In the following we divide the proof into three cases.

First, suppose that $\psi(\frac{1}{2})<h_*$. 
Let $\theta$ be a tangential angle of $\gamma$, i.e., $\theta \in C^0([0,L])$ is such that $\tilde{\gamma}'= (\cos \theta, \sin \theta)$. 
By \eqref{eq:ell-psi-1to1} and Lemma~\ref{lem:polar-rectangular}, we have $0 < \ell < \frac{L_{\rm rect}}{2}$, and this and \eqref{eq:kappa_rect} yield $\langle\gamma_{\rm rect}'(\ell), e_i\rangle>0$ for $i=1,2$. 
Then, it follows from \eqref{eq:SCF-unique-formula-phi} that $\phi\in (-\frac{\pi}{2},0)$, so that by \eqref{eq:SCF-unique-formula} $\theta(0)-2\pi m \in (0,\frac{\pi}{2})$ for some $m\in \mathbf{Z}$; without loss of generality we may assume $\theta(0) \in (0,\frac{\pi}{2})$. 
Here note from \eqref{eq:def-SCF}-(i) that $\theta(\frac{L}{2})=0$  and by \eqref{eq:def-SCF}-(iv) we have $k< 0$ on $(0,L/2)$.
Since by \eqref{eq:SCF-unique-formula}  one has $\tilde{\gamma} \in C^\infty([0,L/2])$, one can easily check that $\theta \in C^\infty([0,L/2])$ and $\theta'= k < 0$. This implies that $\theta$ is strictly decreasing in $(0,\frac{L}{2})$, and this together with $\theta(0)\in (0,\frac{\pi}{2})$ and $\theta(\frac{L}{2})=0$ gives $0\leq\theta(s) <\pi/2$ for all $s\in[0,L/2]$, and hence $\gamma \in A_{\rm graph}$.

Next we consider the case $\psi(\frac{1}{2})=h_*$. 
Then, by \eqref{eq:values-polar-rect} we have $\ell=\frac{L_{\rm rect}}{2}$. 
Since $\gamma_{\rm rect}'(\frac{L_{\rm rect}}{2})=e_1$, we deduce from \eqref{eq:SCF-unique-formula-phi} that $\phi=0$. 
Furthermore, in view of \eqref{eq:rect-height} and \eqref{eq:SCF-unique-formula-alpha}, we see that $\alpha=1$. 
Thus, by \eqref{eq:SCF-unique-formula}, we have $\tilde{\gamma}=\gamma_{\rm rect}$.  
Thus $k'(\frac{L}{2})=k_{\rm rect}'(\frac{L_{\rm rect}}{2})=0$.

Finally we turn to the case $\psi(\frac{1}{2})>h_*$. 
By \eqref{eq:values-polar-rect} we have $\ell>\frac{L_{\rm rect}}{2}$, which means $k_{\rm rect}'(\ell)<0$.
On the other hand, in view of \eqref{eq:kappa_formula_half} and \eqref{eq:SCF-unique-formula-alpha-L}, we see that 
\begin{align}
    \lim_{s \uparrow \frac{L}{2}} k'(s) 
    = \alpha k_{\rm rect}'(\ell). 
\end{align}
The proof is complete.
\end{proof}

Next, we investigate relationship between the variational inequality and symmetric cut-and-glued free-elasticae. 
\begin{lemma}\label{lem:vari-ineq-and-SCF}
Let $\gamma \in A$ be a symmetric cut-and-glued free-elastica and $L:=L[\gamma]$. 
Then, $\gamma$ satisfies \eqref{eq:vari-ineq-explict} with $\lambda=0$ if and only if $\lim_{s \uparrow \frac{L}{2}} k'(s) \leq0$. 
\end{lemma}
\begin{proof}
Fix an arbitrary $\varphi\in C^\infty_{\rm c}(0,L)$ with $\varphi\geq0$ in $[0,L]$. 
Noting symmetry of $\gamma$ and the fact that $k$ is smooth on $(0,L/2)$ (cf.\ \eqref{eq:def-SCF}-(ii)), we integrate by parts to compute
\begin{align}
    &\int_0^{\frac{L}{2}}\Big(2\langle\boldsymbol{\kappa},e_2\rangle\varphi''-3|\boldsymbol{\kappa}|^2\langle\boldsymbol{t},e_2\rangle\varphi'\Big)\; \mathrm{d}s \\
    = & \int_0^{\frac{L}{2}}\big(2{k}''+k^3 \big)\langle\boldsymbol{n},e_2\rangle\varphi \; \mathrm{d}s 
    + \Big[ 2k\langle\boldsymbol{n},e_2\rangle\varphi' -2k'\langle\boldsymbol{n},e_2\rangle\varphi -k^2\langle\boldsymbol{t},e_2\rangle\varphi \Big]_{s=0}^{s=\frac{L}{2}} 
    \\  =& -2\Big(\lim_{s\uparrow\frac{L}{2}}k'(s)\Big)\varphi(\tfrac{L}{2}) + 2 k(\tfrac{L}{2})\varphi'(\tfrac{L}{2})  , 
\end{align}
where in the last equality we used \eqref{eq:def-SCF}-(ii), (iii), $\boldsymbol{t}(\frac{L}{2})=e_1$ and $\boldsymbol{n}(\frac{L}{2})=e_2$. 
Since $\gamma$ is symmetric, along the same argument we have 
\begin{align}
    &\int_{\frac{L}{2}}^L\Big(2\langle\boldsymbol{\kappa},e_2\rangle\varphi''-3|\boldsymbol{\kappa}|^2\langle\boldsymbol{t},e_2\rangle\varphi'\Big)\; \mathrm{d}s = 2\Big(\lim_{s\downarrow\frac{L}{2}}k'(s)\Big)\varphi(\tfrac{L}{2}) - 2 k(\tfrac{L}{2})\varphi'(\tfrac{L}{2})  
\end{align}
Since $\lim_{s\downarrow\frac{L}{2}}k'(s) = -\lim_{s\uparrow\frac{L}{2}}k'(s)$ follows from symmetry, we have
\[
\int_0^L \Big(2\langle\boldsymbol{\kappa},e_2\rangle\varphi''-3|\boldsymbol{\kappa}|^2\langle\boldsymbol{t},e_2\rangle\varphi' 
\Big)\; \mathrm{d}s
= -4\Big(\lim_{s\uparrow\frac{L}{2}}k'(s) \Big)\varphi(\tfrac{L}{2}). 
\]
The proof is complete.
\end{proof}

\subsection{Stability analysis for symmetric cut-and-glued free-elasticae}

We are in a position to prove Theorem~\ref{thm:destabilization-SCF}.
For the critical case $\psi(\frac{1}{2})=h_*$, we employ a similar trick to \cite{MY_Crelle}.

\begin{proof}[Proof of Theorem~\ref{thm:destabilization-SCF}]
Throughout this proof let $\gamma$ denote a symmetric cut-and-glued free-elastica.
First we consider the case $\psi(\frac{1}{2})<h_*$. 
By Lemma~\ref{lem:exsitence-SCF}-(i), $\gamma$ is given by a  graph of a function; in particular, $\gamma:=(X,Y)$ satisfies $X'(x)>0$ for any $x\in[0,1]$.
Take an arbitrary sequence $\{\gamma_n\}_{n\in \mathbf{N}} \subset A_{\rm sym}$ such that $\gamma_n:=(X_n,Y_n) \to \gamma$ in $W^{2,2}(0,1;\mathbf{R}^2)$, and hence also in $C^1([0,1];\mathbf{R}^2)$. 
Then, $X_n'>0$ on $[0,1]$ for all large $n\in \mathbf{N}$, which implies that $\gamma_n \in A_{\rm graph}$ for all large $n$.
Therefore, this together with the fact that $\gamma$ is a global minimizer of $B$ in $A_{\rm graph} \cap A_{\rm sym}$ (cf.\ \cite{Ysima}*{Theorem~1.2}), $B[\gamma] \leq B[\gamma_n]$ holds for all large $n$. 
Hence $\gamma$ is a local minimizer in $A_{\rm sym}$. 
On the other hand, $\gamma$ is not a global minimizer since there  exists no global minimizer in $A_{\rm sym}$ (cf.\ Lemma~\ref{lem:3.10}).

Next we consider the case $\psi(\frac{1}{2})=h_*$.
Let us recall from Lemma~\ref{lem:exsitence-SCF} that in this case  $\tilde{\gamma}=\gamma_{\rm rect}$.
For a contradiction assume that $\gamma$ is a local minimizer of $B$ in $A$. 
Let $L:=L[\gamma]$.
For each $n \in \mathbf{N}$, we define the arclength parametrized curve $\gamma_n:[0,L+2/n]\to\mathbf{R}^2$ by
\[
\gamma_n:=\gamma^{n,+}_{\rm seg}\oplus  \gamma_{\rm rect}\oplus \gamma^{n,-}_{\rm seg}, 
\]
where 
$\gamma^{n,\pm}_{\rm seg}(s)= (0,\pm s)$ for $s\in[0,\tfrac{1}{n}]$. 
Then, by Lemma~\ref{lem:exsitence-SCF}-(ii) combined with the fact that $\gamma_{\rm rect}'(0)=e_2$, we see that $\gamma_n$ is of class $C^1$. 
It also follows by construction that $\gamma_n^2> \psi\circ\gamma_n^1$ in $[0,L+2/n]$, and hence, up to reparametrization, $\gamma_n \in A$; in particular the coincidence set for $\gamma_n$ is empty, i.e.,  $I_{\gamma_n}=\emptyset$. 
Moreover, since $B[\gamma_n]=B[\gamma]$ for any $n\in\mathbf{N}$ and since $\gamma$ is a local minimizer of $B$ in $A$, we see that for suitably large $n \in \mathbf{N}$ $\gamma_n$ is also a local minimizer in $A$.
Therefore, combining Lemma~\ref{lem:E-L_eq_on_noncoincidence} with $I_{\gamma_n}=\emptyset$, we see that the signed curvature $k_n$ of $\gamma_n$ is analytic on $(0,L+2/n)$.
On the other hand, $k_n$ is not continuously differentiable in any neighborhood of 
$s={1}/{n}$: In fact, by \eqref{eq:kappa_rect},
\[
\lim_{s\uparrow \frac{1}{n}} k_n'(s) =0, \quad \lim_{s\downarrow \frac{1}{n}} k_n
'(s) = k_{\rm rect}'(0) = -\sqrt{2}\alpha_0^2 <0,
\] 
where $\alpha_0=4\mathrm{E}(\tfrac{1}{\sqrt{2}})-2\mathrm{K}(\tfrac{1}{\sqrt{2}})$.
This contradicts the analyticity of $k_n$. 
Thus $\gamma$ can not be a local minimizer of $B$ in  $A$. 
This argument also implies that $\gamma$ is not a local minimizer in $A_{\rm sym}$. 

We now address the remaining case $\psi(\frac{1}{2})>h_*$. If a symmetric cut-and-glued free-elastica were a local minimizer, then it would solve 
the variational inequality (cf.\ Lemma~\ref{lem:vari-ineq}). 
However, this is a contradiction in view of Lemma~\ref{lem:exsitence-SCF}-(iii) and Lemma~\ref{lem:vari-ineq-and-SCF}.
\end{proof}

\begin{remark}
   Consider again a symmetric cone obstacle $\psi$ that satisfies (A1)--(A3).
   \begin{itemize}
    \item[(i)] If $\psi$ has the critical height $\psi(\frac{1}{2})=h_*$, the unique symmetric cut-and-glued free-elastica $\gamma \in A$ is then not a local minimizer by Theorem~\ref{thm:destabilization-SCF}. However, it surprisingly solves the variational inequality! Indeed, this is readily checked with Lemma~\ref{lem:vari-ineq-and-SCF} combined with Lemma~\ref{lem:exsitence-SCF}-(ii).
   The existence of such a nonminimizing solution of the variational inequality reveals once more the nonconvex nature of the bending energy $B$. 
    \item[(ii)] Notice that below the critical height $h_*$, it has been shown that all solutions of the variational inequality in $A_{\rm graph} \cap A_{\rm sym}$ are minimizers. Even more holds true: There exists only one solution of the variational inequality in $A_{\rm graph} \cap A_{\rm sym}$, cf. \cite[Lemma 6.6]{Mueller21}.  
    \end{itemize} 
\end{remark}

\section{Nontouching minimizers for small $\lambda > 0$}\label{sect:Nontouching_small_lambda}

In this section we prove Theorem~\ref{thm:hitting-escaping}, i.e., for all sufficiently small $\lambda>0$ minimizers of $\mathcal{E}_\lambda$ in $A_{\rm sym}$ do not touch a fixed symmetric cone obstacle $\psi$. 

We remark that it is not possible to estimate the energy of curves touching the obstacle from below for any small $\lambda> 0$. Therefore, further analysis which relies on properties of minimizers is required.

\begin{remark}\label{rem:escaping-circular_arcs}
Let $\{\lambda_n\}_{n\in \mathbf{N}} \subset (0,\infty)$ be an arbitrary sequence such that $\lambda_n \to 0$ ($n\to\infty$). 
For each $n\in \mathbf{N}$ we construct a family $\{C_n\}_{n\in \mathbf{N}}$ as follows: 
For sufficiently large $n\in\mathbf{N}$, we consider the circle with the radius $\lambda_n^{-\frac{1}{2}}$ and the center $(X_n,Y_n)$ given by
\[
X_n=-\frac{a_\psi}{a_\psi^2+1}\Big(b_\psi+\frac{1}{\sqrt{\lambda_n}}\sqrt{a_\psi^2+1}\Big) - \frac{1}{a_\psi^2+1}\bigg(-b_\psi^2-2b_\psi\frac{1}{\sqrt{\lambda_n}}\sqrt{a_\psi^2+1}\bigg)^{\frac{1}{2}}, 
\]
\[
Y_n=\frac{1}{a_\psi^2+1}\Big(b_\psi+\frac{1}{\sqrt{\lambda_n}}\sqrt{a_\psi^2+1}\Big) - \frac{a_\psi}{a_\psi^2+1}\bigg(-b_\psi^2-2b_\psi\frac{1}{\sqrt{\lambda_n}}\sqrt{a_\psi^2+1}\bigg)^{\frac{1}{2}}, 
\]
where $a_\psi:=\psi'(0)>0$ and $b_\psi:=\psi(0)<0$.
This circle passes through $(0,0)$ and is tangent to $\psi|_{(-\infty,\frac{1}{2}]}$, which can be readily checked using the affine linearity of $\psi$. 
Among this circle, let $C_{n,1}:[0,1]\to\mathbf{R}^2$ denote the circular arc that satisfies $C_{n,1}(0)=(0,0)$ and $C_{n,1}'(1)$ is parallel to $e_2$.
We also define $C_{n,2}: [0,1] \to \mathbf{R}^2$ by a circular arc such that $C_{n,2}(0)=C_{n,1}(1)$, $C_{n,2}'(0)$ is parallel to $e_2$, 
$C_{n,2}^1(1)=\frac{1}{2}$, and $C_{n,2}'(1)$ is parallel to $e_1$. For sufficiently large $n$ there holds $C_{n,2}^2(1) > \psi(\frac{1}{2})$ and the radius $r_n$ of $C_{n,2}$ satisfies 
\[\lambda_n^{-\frac{1}{2}}\leq r_n\leq 2 \lambda_n^{-\frac{1}{2}}+ \tfrac{1}{2},
\]
where the last inequality follows from the fact that 
$Y_n >0$ for all large $n\in \mathbf{N}$ (see also Figure  \ref{fig:escaping}).
We define $C_n:[0,1]\to\mathbf{R}^2$ to be a reflectionally symmetric curve such that $C_n|_{[0,\frac{1}{2}]}$ coincides with $C_{n,1}\oplus C_{n,2}$ up to reparametrization (see also Figure~\ref{fig:escaping}). 
By definition we see that 
\begin{align}
\frac{1}{2} \mathcal{E}_{\lambda_n}[C_n] &= B[C_{n,1}] + B[C_{n,2}] + \lambda_n \big( L[C_{n,1}] + L[C_{n,2}] \big) \\
&\leq 2\pi \sqrt{\lambda_n} + 2\pi\frac{1}{r_n} + \lambda_n \Big(2\pi\frac{1}{\sqrt{\lambda_n}}+ 2\pi r_n\Big) 
= O(\sqrt{\lambda_n}) \quad \lambda_n\to0.
\end{align}
These arguments imply that for sufficiently small values of $\lambda$ there exist curves $C_\lambda$ such that $I_{C_\lambda}\neq\emptyset$ but the energy of $C_\lambda$ vanishes if $\lambda\to0$. Accordingly, hitting the obstacle does not yield any lower energy bound independent of $\lambda$.
\begin{center}
    \begin{figure}[htbp]
      \includegraphics[scale=0.05]{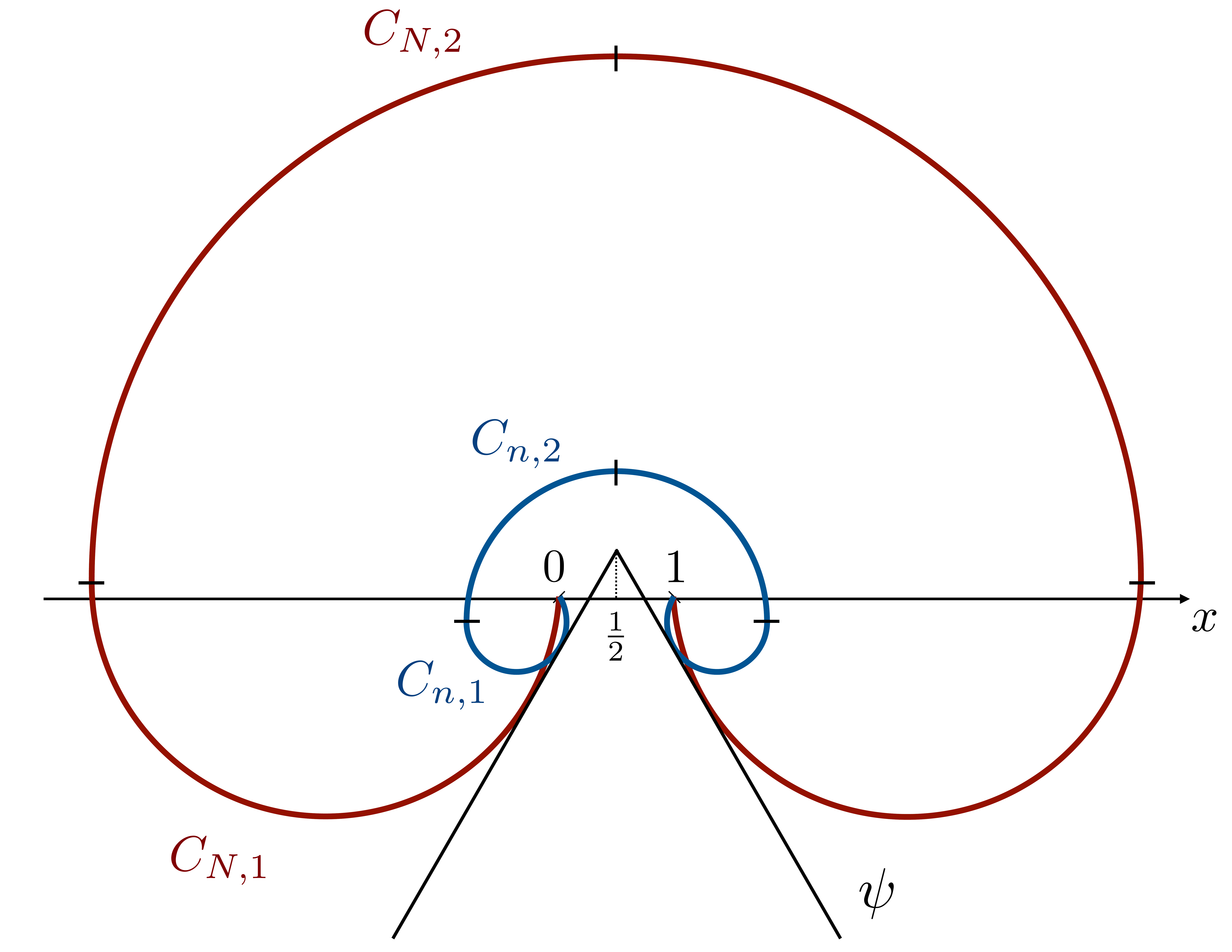}
  \caption{Examples of constructions of the curve $C_n$ as in Remark~\ref{rem:escaping-circular_arcs}. The energy $\mathcal{E}_{\lambda_n}[C_n]$ can be close to $0$ as $n\to\infty$ while $C_n$ always touch the obstacle.
  }
  \label{fig:escaping}
  \end{figure}
\end{center}
\vspace{-1.5\baselineskip}
\end{remark}
In order to overcome the difficulty mentioned above,  
we introduce an  auxiliary variational problem, called the \textit{rhomb problem}: consider the minimizing problem of $\mathcal{E}_\lambda$ in 
\[
A_{\rm sym}^\lozenge:= 
\Set{ \gamma \in W^{2,2}_{\rm imm}(0,1; \mathbf{R}^2) |
\begin{array}{l}
 \gamma(0)=(0,0), \ \gamma(1)=\gamma(1,0), \;  \gamma(1-x) = (1-\gamma^1(x),\gamma^2(x)), \\ 
 |\gamma^2(x)| \geq \psi(\gamma^1(x))) \text{ in } [0,1]
\end{array}
}.
\]
We now discuss existence and some properties of minimizers in $A_{\rm sym}^\lozenge$. The existence of minimizers can be derived almost along the lines of Lemma \ref{lem:globalminimizer_existience}. For that reason we postpone a detailed proof of the following lemma to Appendix \ref{sect:appendixB}.

\begin{lemma}\label{lem:existence-rho}
For each $\lambda>0$ there exists a minimizer of $\mathcal{E}_\lambda$ in $A_{\rm sym}^\lozenge$. 
\end{lemma}

\begin{remark}\label{rem:smallenergyInfima}
    Notice that $\inf\limits_{A_{\rm sym}^\lozenge }\mathcal{E}_\lambda \rightarrow 0$ as $\lambda \rightarrow 0$. The reason for this is as follows. For $\lambda < \hat{\lambda}$ (given as in \eqref{eq:def-hat_lambda}) and also such that $\psi(\frac{1}{2}) < h_\lambda$ (cf.\ \eqref{eq:h_lambda}) one readily checks that $\gamma^{\lambda,1,1}_{\rm larc}$ is admissible in $A_{\rm sym}^\lozenge$. Hence we may estimate 
   $$
        \inf_{A_{\rm sym}^\lozenge} \mathcal{E}_\lambda  \leq \mathcal{E}_\lambda[\gamma^{\lambda,1,1}_{\rm larc}] \longrightarrow 0 \quad (\lambda \rightarrow 0). 
 $$ More precisely one has $\mathcal{E}_\lambda[\gamma_{\rm larc}^{\lambda,1,1}] = O(\sqrt{\lambda})$ as can be seen with \eqref{eq:longer-arc_energy}. 
\end{remark}

\begin{remark}\label{rem:minimizer-touching_rhomb}
Lemma~\ref{lem:existence-rho} enables us to discuss properties of minimizers for the rhomb problem.
In fact, a minimizer of $\mathcal{E}_\lambda$ in $A_{\rm sym}^\lozenge$ has similar properties to minimizers of $\mathcal{E}_\lambda$ in $A_{\rm sym}$. 
For example, by the same argument as in Lemma~\ref{lem:E-L_eq_on_noncoincidence}, if $\gamma$ is a minimizer of $\mathcal{E}_\lambda$ in $A_{\rm sym}^\lozenge$, then its signed curvature $k$ satisfies $k(0)=0$ and $2k''+k^3-\lambda k=0$ on the non-coincidence set.
\end{remark}

We intend to prove that minimizers in $A_{\rm sym}$ do not touch the symmetric cone obstacle for sufficiently small $\lambda> 0$. To this end it turns out helpful to show first the nontouching property for minimizers in $A_{\rm sym}^\lozenge$.

\begin{definition} 
For $\gamma \in A_{\rm sym}^\lozenge$, we define the \textit{coincidence set} $I_\gamma^\lozenge$ by 
    \[
    I_\gamma^\lozenge := \big\{ x\in [0,1] \,\big|\, |\gamma^2(x)|=\psi(\gamma^1(x)) \big\}.
    \]
Let $\mathcal{M}_\lambda^\lozenge$ denote 
    the set of all minimizers of $\mathcal{E}_\lambda$ in $A_{\rm sym}^\lozenge$ that touch the obstacle, i.e. 
    \begin{equation}
        \mathcal{M}_\lambda^\lozenge = \Big\{ \gamma \in A_{\rm sym}^\lozenge \ \Big|\ \mathcal{E}_\lambda[\gamma] = \inf_{\sigma \in A_{\rm sym}^\lozenge} \mathcal{E}_\lambda[\sigma], \ I_\gamma^\lozenge \ne \emptyset \Big\}.
    \end{equation}
\end{definition}

Our goal is to show that $\mathcal{M}_\lambda^\lozenge = \emptyset$ for sufficiently small $\lambda$. In order to do so we deduce several properties of elements of $\mathcal{M}_\lambda^\lozenge$. One first property we need is embeddedness of these elements.
To this end we recall

\begin{lemma}[\cite{MYarXiv2409}*{Lemma~4.3}]\label{lem:property-h}
The function $h:(\frac{1}{\sqrt{2}},1)\to \mathbf{R}$ defined by 
\begin{align}\label{eq:def-h}
    h(q):=\frac{1}{\sqrt{2q^2-1}}\big((4q^2-3)\mathrm{K}(q)+2\mathrm{E}(q) \big), \quad q\in\big(\tfrac{1}{\sqrt{2}},1)
\end{align}
satisfies
\begin{align}\label{eq:diff-h}
h'(q)=-\frac{f(q)}{(2q^2-1)^{\frac{3}{2}}q(1-q^2)}, \quad q\in (\tfrac{1}{\sqrt{2}},1),
\end{align}
where $f$ is the function defined by \eqref{eq:def-f}.
In particular, $h$ is strictly increasing on $(\hat{q},1)$. 
\end{lemma}

Now we investigate embeddedness of a minimizer in the rhomb problem.

\begin{lemma}\label{lem:embed-touching-rhomb}
Let $\hat{\lambda}>0$ be a constant given by \eqref{eq:def-hat_lambda}.
Suppose that $\lambda \in (0,\hat{\lambda})$ and a symmetric cone obstacle $\psi:\mathbf{R}\to\mathbf{R}$  satisfies \eqref{eq:height-nontouching}. 
Then each 
element of $\mathcal{M}_\lambda^\lozenge$ does
 not have a self-intersection. 
\end{lemma}
\begin{proof}
We argue by contradiction. 
If the statement were false, then there would exist 
$\Gamma \in \mathcal{M}_\lambda^\lozenge$ 
such that $\Gamma(a)=\Gamma(b)$ for some $0\leq a < b \leq1$.  
Then, noting that up to reparametrization $\Gamma|_{[a,b]} \in A_{\rm drop}=\{ \gamma\in W^{2,2}_{\rm imm}(0,1;\mathbf{R}^2)\,|\, \gamma(0)=\gamma(1)\}$, we deduce from Proposition~\ref{prop:fig8_minimizes} that
\begin{align}\label{energy-for-embedded1}
   \mathcal{E}_\lambda[\Gamma] \geq \mathcal{E}_\lambda[\Gamma|_{[a,b]}] &\geq \inf_{\gamma \in A_{\rm drop}}\mathcal{E}_\lambda[\gamma] 
    = 2\sqrt{2}\sqrt{\lambda}\frac{(4q_*^2-3)\mathrm{K}(q_*)+2\mathrm{E}(q_*)}{\sqrt{2q_*^2-1}}.
\end{align}
On the other hand, by the assumption $\lambda \in (0,\hat{\lambda})$ we
may consider the curve $\gamma^{\lambda,1,1}_{\rm larc}$ as given in Proposition \ref{prop:PPE-classification}   
and by \eqref{eq:height-nontouching} we deduce that $\gamma^{\lambda,1,1}_{\rm larc} \in A_{\rm sym}^\lozenge$. This yields
\begin{align}\label{energy-for-embedded2}
    \mathcal{E}_\lambda[\Gamma] = \inf_{\gamma \in A_{\rm sym}^\lozenge}\mathcal{E}_\lambda[\gamma] \leq \mathcal{E}_\lambda[\gamma^{\lambda,1,1}_{\rm larc}] = 2\sqrt{2}\sqrt{\lambda}\frac{(4q_2^2 -3 ) \mathrm{K}(q_2)+2\mathrm{E}(q_2)}{\sqrt{2q_2^2 -1}}, 
\end{align}
where in the last equality we  used the energy formula for $\gamma^{\lambda,1,1}_{\rm larc}$ (cf.\ \eqref{eq:longer-arc_energy}) and set $q_2:=q_2(\lambda)$ as in \eqref{eq:mod_PPE}.  
Thus, using the function $h$ defined by \eqref{eq:def-h} we obtain $2\sqrt{2}\sqrt{\lambda}h(q_*)\leq\mathcal{E}_\lambda[\Gamma]\leq 2\sqrt{2}\sqrt{\lambda}h(q_2)$, which contradicts Lemma~\ref{lem:property-h} since $\hat{q}<q_2<q_*$. 
\end{proof}

Next we prove a result about the tangential angle of elements in $\mathcal{M}_\lambda^\lozenge$. 
This result will once again motivate the introduction of the rhomb problem. 
For $\gamma\in A$ let $l$ be the \textit{first touching point} in terms of the arclength parametrization, i.e., 
\[l:= \min\big\{s\in [0,L[\gamma]] \mid |\tilde{\gamma}^2(s)| = \psi( \tilde{\gamma}^1(s)) \big\}.\]
In the original problem, the first touching point $l$ of $\gamma$ may diverge to $-\infty$, and then the tangential angles on $[0,l]$ may uniformly tend to $-\theta_\psi:=-\arctan{\psi'(0)}$.
However, in the rhomb problem, we can always find a point with a different angle from that of the first touching point.

\begin{lemma}\label{lem:variation-angle-rho}
Let $\lambda \in (0,\hat{\lambda})$.
Let $\gamma \in \mathcal{M}_\lambda^\lozenge$ (with arclength reparametrization $\tilde{\gamma} : [0,L] \rightarrow \mathbf{R}^2$) and $l$ be the first touching point of $\gamma$. Further, let $\theta$ be a tangential angle function for $\tilde{\gamma}$, i.e. $\theta \in C^0([0,L])$ and $\tilde{\gamma}'=(\cos(\theta),\sin(\theta))$.
Then, there exists $l^* \in (0,l)$ such that 
\begin{equation} \label{eq:variation-theta-uniform}
\theta(l) - \theta(l^*)  \in [c_\psi, 2\pi- c_\psi] \quad (\mathrm{mod} \ 2\pi).
\end{equation}
Here $c_\psi\in (0,\frac{\pi}{2})$ is a constant only depending on $\psi$ (independent of $\lambda$), defined by 
\[
c_\psi:= \min\big\{ \theta_\psi-\arctan{(2\psi(\tfrac{1}{2}))}, \ \arctan{(2\psi(\tfrac{1}{2}))} \big\}, \quad \text{where} \ \ \theta_\psi:=\arctan\psi'(0).
\]
\end{lemma}
\begin{proof}
Since $\gamma \in \mathcal{M}_\lambda^\lozenge$ is a minimizer of $\mathcal{E}_\lambda$ in $A_{\rm sym}^\lozenge$, then the reflection of $\gamma$ is also a minimizer in $A_{\rm sym}^\lozenge$. Thus, up to reflection, 
we may assume $\tilde{\gamma}^2(l)\geq0$ (i.e., if  $\tilde{\gamma}^2(l)<0$, then we may consider the reflection of $\gamma$ instead of the original curve $\gamma$).
Hereafter let $a_0 \in (0,\frac{1}{2})$ denote by the unique point such that $\psi'(a_0)>0$ and $\psi(a_0)=0$.
Throughout the proof let $a\in [a_0, \frac{1}{2}]$ be such that $\tilde{\gamma}(l)=(a, \psi(a))$. 
Then, by the Cauchy's mean value theorem, there exists $l^*\in(0,l)$ such that 
\[
\frac{(\tilde{\gamma}^2)'(l^*)}{(\tilde{\gamma}^1)'(l^*)} = \frac{\tilde{\gamma}^2(l)- \tilde{\gamma}^2(0)}{\tilde{\gamma}^1(l)- \tilde{\gamma}^1(0)} = \frac{\psi(a)}{a}, 
\]
from which it follows that 
\begin{equation}\label{eq:angle-l*}
\theta(l^*)=  \arctan{\big(\tfrac{\psi(a)}{a}\big)} \quad \textrm{or} \quad   \theta(l^*)=  \pi + \arctan{\big(\tfrac{\psi(a)}{a}\big)}\quad (\text{mod } 2\pi).
\end{equation}

\textbf{Case I.} First we address the case $a = \frac{1}{2}$. 
By $\lambda\in(0,\hat{\lambda})$ and the embeddedness of $\gamma$ (cf.\ Lemma~\ref{lem:embed-touching-rhomb}) we see that $l=\frac{L[\gamma]}{2}$, and then by the symmetry it also follows that $\theta(l)=2m\pi$ for some $m\in \mathbf{Z}$. 
This together with \eqref{eq:angle-l*} yields 
\[
\theta(l)-\theta(l^*)  \in [     \pi - \arctan(2\psi(\tfrac{1}{2})), 2\pi - \arctan(2\psi(\tfrac{1}{2})) ] \quad  (\text{mod } 2\pi).
\]
Using that $\pi - \arctan(2\psi(\tfrac{1}{2})) \geq \frac{\pi}{2} 
> c_\psi$ and  $\arctan(2\psi(\tfrac{1}{2})) \geq c_\psi$ we obtain \eqref{eq:variation-theta-uniform} in this case. 

\textbf{Case II.} Next we consider the case $a \in (a_0,\frac{1}{2})$.
Note that in this case $\psi\circ\tilde{\gamma}^1$ is differentiable at $s=l$. 
Since $\tilde{\gamma}^2(l)=\psi(a)>0$, not only $|\tilde{\gamma}^2|-\psi\circ\tilde{\gamma}^1$ but also $\tilde{\gamma}^2-\psi\circ\tilde{\gamma}^1$ attains a local minimum with value $0$ at $s=l$.
Therefore we have $(\tilde{\gamma}^2)'(l)-\psi'(a)(\tilde{\gamma}^1)'(l)=0$, which implies that 
$\theta(l)=\theta_\psi$ or $\theta(l) = \pi + \theta_\psi$ (mod $2\pi$).
Combining this with \eqref{eq:angle-l*}, we see that 
\[
\theta(l)-\theta(l^*) = \theta_\psi-\arctan{(\tfrac{\psi(a)}{a})} \quad \text{or} \quad \theta(l)-\theta(l^*) = \theta_\psi-\arctan{(\tfrac{\psi(a)}{a})}+\pi \quad (\text{mod }2\pi).
\]
Since $[a_0,1/2]\ni a\mapsto \arctan{(\tfrac{\psi(a)}{a})}$ is strictly increasing, we see that 
$\theta(l)-\theta(l^*)\in [\theta_\psi-\arctan{2\psi(\tfrac{1}{2})}, 3\pi/2]$ (mod $2\pi$).
This clearly implies that \eqref{eq:variation-theta-uniform} holds.

\textbf{Case III.} Now we turn to the remaining case $a=a_0$.
After another reflection we may assume that $(\tilde{\gamma}^2)'(l)\geq0$, and particularly $\theta(l)\in [0,\pi]$ (mod $2\pi$). 
We show that in this case 
\begin{equation}\label{eq:theta-estimate}
\theta(l)\in [\theta_\psi, \pi-\theta_\psi] \quad  (\text{mod } 2\pi). 
\end{equation}  
We argue by contradiction, so suppose that $\theta(l)\in [0, \theta_\psi)$ (mod $2\pi$) or that $\theta(l)\in (\pi-\theta_\psi,\pi]$ (mod $2\pi$). 
In the former case, assuming additionally $\theta(l) \neq 0$, the assumption implies that 
\[
\frac{(\tilde{\gamma}^2)'(l)}{(\tilde{\gamma}^1)'(l)}=\tan\theta(l) < \psi'(0)=\psi'(a_0) =\psi'(\tilde{\gamma}^1(l)).
\]
This inequality and the observation that $(\tilde{\gamma}^1)'(l)>0$ (due to the fact that $\theta(l) \in (0, \theta_\psi) \subset (0,\frac{\pi}{2})$) imply $(\tilde{\gamma}^2-\psi\circ\tilde{\gamma}^1)'(l)<0$. The  last  inequality clearly holds also for the case $\theta(l)=0$.
Hence 
\[
\tilde{\gamma}^2(l+\delta) < \psi(\tilde{\gamma}^1(l+\delta)) \quad \text{for sufficiently small } \delta>0.
\]
A contradiction. 
If $\theta(l)\in (\pi-\theta_\psi,\pi]$ (mod $2\pi$), then similar to the previous argument we can deduce that 
\[
-\tilde{\gamma}^2(l-\delta) < \psi(\tilde{\gamma}^1(l-\delta)) \quad \text{for sufficiently small } \delta>0, 
\]
which is a contradiction to the fact that 
$\gamma\in A_{\rm sym}^\lozenge$.
Thus we have shown \eqref{eq:theta-estimate}.
In view of the fact that $a = a_0$, we deduce from \eqref{eq:angle-l*} that $\theta(l^*) = 0$ or $\theta(l^*) = \pi$ (mod $2\pi$). 
Then this together with \eqref{eq:theta-estimate} yields $\theta(l)-\theta(l^*)\in[\theta_\psi, \pi-\theta_\psi]$ (mod $2\pi$). 
The proof is now complete.
\end{proof}

We are now ready to prove noncoincidence of minimizers of $\mathcal{E}_\lambda$ in $A_{\rm sym}^\lozenge$ for small $\lambda > 0$. 

\begin{lemma}\label{lemma:rhombproblemnontouhing}
    There exists some $\bar{\lambda}>0$ such that for all $\lambda \leq \bar{\lambda}$ each minimizer of $\mathcal{E}_\lambda$ in $A_{\rm sym}^\lozenge$ does not touch the obstacle, i.e. $\mathcal{M}_\lambda^\lozenge = \emptyset$ for all $\lambda \leq \bar{\lambda}$. Furthermore, all minimizers are given by $\gamma_{\rm larc}^{\lambda,1,1}$ and the reflection of $\gamma_{\rm larc}^{\lambda,1,1}$ at the $x$-axis. 
\end{lemma}
\begin{proof}
For a contradiction we assume the opposite, i.e. there exist $\{\lambda_n\}_{n\in \mathbf{N}}$ with $\lambda_n \to 0$ ($n\to\infty$) such that $\mathcal{M}_{\lambda_n}^\lozenge \neq \emptyset$. Choose $\gamma_n \in \mathcal{M}_{\lambda_n}^\lozenge$ for all $n \in \mathbf{N}$. From Remark \ref{rem:smallenergyInfima} we infer that 
\begin{equation}\label{eq:EnergyTendsTOzero}
    \mathcal{E}_{\lambda_n}[\gamma_n] = \inf_{A_{\rm sym}^\lozenge} \mathcal{E}_{\lambda_n} \longrightarrow 0 \qquad (n \rightarrow \infty). 
\end{equation}
Let $x_n$ be the first touching point of $\gamma_n$ in terms of the original parameter, i.e., 
$x_n := \min \{ x \in [0,1] \,|\, \gamma_n^2(x) = \psi(\gamma_n^1(x)) \}$, 
and define $l_n:=\int_0^{x_n}|\gamma_n'|$. 
Then, since the signed curvature $k_n$ of $\gamma_n$ satisfies the Euler--Lagrange equation of $\mathcal{E}_{\lambda_n}$ on $(0,l_n)$ and also satisfies $k_n(0)=0$ (cf.\ Remark~\ref{rem:minimizer-touching_rhomb}), $\gamma_n|_{[0,x_n]}$ can be characterized as a line segment or a wavelike elastica. 
The former is impossible in view of the $C^1$-regularity of $\gamma_n$ and the embeddedness (cf.\ Lemma~\ref{lem:embed-touching-rhomb}), and hence we find that 
$k_n|_{[0,l_n]}$ 
is (up to reflection) given by 
\begin{equation}\label{eq:240410-1}
k_n(s) = 2 \alpha_n q_n \cn(\alpha_n s -\mathrm{K}(q_n), q_n), \quad s\in [0, l_n],
\end{equation}
for some $\alpha_n>0$ and $q_n \in [\frac{1}{\sqrt{2}},1)$ with $\lambda_n=2\alpha_n^2(2q_n^2-1)$. 
Note that by the change of variables $\xi=\am(\alpha_n s- \mathrm{K}(q_n),q_n)$ 
\begin{equation}\label{eq:calc-B-xi}
     \int_{\gamma_n \mid_{[0,l_n]}} k_n^2 \; \mathrm{d}s  = 4\alpha_n^2 q_n^2 \int_0^{l_n} \cn(\alpha_n s -\mathrm{K}(q_n), q_n)^2 \; \mathrm{d}s  = \alpha_n \int_{-\frac{\pi}{2}}^{\xi_n} \frac{4q_n^2 \cos^2\xi}{\sqrt{1-q_n^2 \sin^2\xi}} \; \mathrm{d}\xi, 
\end{equation}
where $\xi_n:=\am(\alpha_n l_n-\mathrm{K}(q_n),q_n)$.
This yields the lower bound for $\mathcal{E}_{\lambda_n}[\gamma_n]$ as follows: 
\begin{equation}\label{eq:B-gamma-n}
    \mathcal{E}_{\lambda_n}[\gamma_n] \geq B[\gamma_n |_{[0,l_n]}] =\int_{\gamma_n \mid_{[0,l_n]}} k_n^2 \; \mathrm{d}s=\alpha_n \int_{-\frac{\pi}{2}}^{\xi_n} \frac{4q_n^2 \cos^2\xi}{\sqrt{1-q_n^2 \sin^2\xi}} \; \mathrm{d}\xi.
\end{equation}
We now split the rest of the proof into several steps.

\textbf{Step~1} (\textsl{Asymptotic behavior of $\xi_n$})\textbf{.} 
Formula \eqref{eq:240410-1} implies that the arclength parametrization $\tilde{\gamma}_n|_{[0,l_n]}$ is, up to translation, rotation, and reflection,  given by 
\[
\tilde{\gamma}_n(s)= \frac{1}{\alpha_n}
\begin{pmatrix}
2\mathrm{E}(\am(\alpha_n s -\mathrm{K}(q_n),q_n), q_n))-\alpha_n s \\
2q_n \cn (\alpha_n s -\mathrm{K}(q_n), q_n) 
\end{pmatrix},
\quad s\in[0,l_n].
\]
By reparametrization $\xi=\am(\alpha_n s-\mathrm{K}(q_n), q_n)$ we have $\tilde{\gamma}_n(s)=\hat{\gamma}_n(\xi(s))$ using 
\begin{equation} \label{eq:gamma-hat}
\hat{\gamma}_n(\xi)= \frac{1}{\alpha_n}
\begin{pmatrix}
2\mathrm{E}(\xi, q_n)- \mathrm{F}(\xi,q_n) \\
2q_n \cos \xi 
\end{pmatrix}, 
\quad \xi \in [-\tfrac{\pi}{2}, \xi_n].
\end{equation}
Let $a_n$ be a point in [0,1] 
such that $\tilde{\gamma}_n(l_n) = (a_n, \psi(a_n))$. 
By definition of  
$A_{\rm sym}^\lozenge$
we have $a_0 \leq a_n \leq 1-a_0$. 
In view of \eqref{eq:gamma-hat} we see that 
\begin{align}\label{eq:250204-1}
\sqrt{a_n^2 +\psi(a_n)^2} &= \big| \tilde{\gamma}_n(l_n) - \tilde{\gamma}_n(0) \big| = \big| \hat{\gamma}_n(\xi_n) - \hat{\gamma}_n(-\tfrac{\pi}{2}) \big| 
\geq \frac{2q_n}{\alpha_n} |\cos{\xi_n}|. 
\end{align}
Notice that in this calculation we have ignored the  aforementioned translations, rotations and reflections as they are isometries of $\mathbf{R}^2$. 
Combining  \eqref{eq:250204-1} with \eqref{eq:B-gamma-n}, we obtain
\[
  \mathcal{E}_{\lambda_n}[\gamma_n] \geq\frac{2q_n |\cos{\xi_n}|}{\sqrt{a_n^2 +\psi(a_n)^2} } \int_{-\frac{\pi}{2}}^{\xi_n} \frac{4q_n^2 \cos^2\xi}{\sqrt{1-q_n^2 \sin^2\xi}} \; \mathrm{d}\xi
   \geq \frac{2\sqrt{2}|\cos{\xi_n}|}{\sqrt{a_n^2 +\psi(a_n)^2} } \int_{-\frac{\pi}{2}}^{\xi_n}  \cos^2\xi \; \mathrm{d}\xi, 
\]
where we also used $q_n\geq \frac{1}{\sqrt{2}}$ in the last equality.
In view of the fact that $|a_n| \leq 1$  
we may define $C_\psi := \max_{x\in[0,1]} \sqrt{x^2 +\psi(x)^2} < \infty$ and estimate 
\[
 \mathcal{E}_{\lambda_n}[\gamma_n] \geq \frac{2\sqrt{2}|\cos{\xi_n}|}{C_\psi} \int_{-\frac{\pi}{2}}^{\xi_n}  \cos^2\xi \; \mathrm{d}\xi.
\]
This together with \eqref{eq:EnergyTendsTOzero} implies that, up to a subsequence (but still denoted by $\{\xi_n\}_{n\in \mathbf{N}}$), 
\begin{align} \label{eq:asymptotic-xi_n}
\xi_n \to \infty \quad \text{or} \quad \xi_n \to \frac{\pi}{2} +(m-1)\pi \ \ \text{for some}\ \  m\in \mathbf{N}\cup\{0\}. 
\end{align}

Next we show that $\xi_n$ does not converge to $-\frac{\pi}{2}$.
To this end suppose on the contrary that $\xi_n \to -\frac{\pi}{2}$ ($n\to\infty$). 
Let $\theta_n$ be the tangential angle of $\tilde{\gamma}_n$.
By Lemma~\ref{lem:variation-angle-rho}, for each $n\in \mathbf{N}$ there is $l^*_n\in (0,l_n)$ such that 
\begin{equation} \label{theta_n-uniform-estimate}
\cos\big(\theta(l_n)-\theta(l^*_n)\big) \leq \cos c_\psi<1,
\end{equation}
where $c_\psi>0$  is defined as in Lemma~\ref{lem:variation-angle-rho}. 
On the other hand, note that by the change of variables
\[
\cos\big(\theta(l_n)-\theta(l^*_n)\big) 
=\langle \tilde{\gamma}_n'(l_n), \tilde{\gamma}_n'(l_n^*) \rangle
=\frac{\langle \hat{\gamma}_n'(\xi_n), \hat{\gamma}_n'(\xi_n^*) \rangle}{|\hat{\gamma}_n'(\xi_n)| |\hat{\gamma}_n'(\xi_n^*)|}, 
\]
where $\xi^*_n:=\am(\alpha_n l_n^* - \mathrm{K}(q_n),q_n)$. 
To calculate the right-hand side we deduce from \eqref{eq:gamma-hat} that 
\begin{equation}\label{eq:tangentofgamma}
\hat{\gamma}_n'(\xi) = \frac{1}{\alpha_n}\frac{d}{d\xi}
	\begin{pmatrix}
            2 \mathrm{E}(\xi,q_n)- \mathrm{F}(\xi,q_n) \\ 2 q_n \cos{\xi} 
        \end{pmatrix} 
=\frac{1}{\alpha_n}
	\begin{pmatrix}
           \frac{1-2q_n^2\sin^2\xi}{\sqrt{1-q_n^2\sin^2\xi}} \\ -2 q_n \sin\xi
        \end{pmatrix}.
\end{equation}
This yields 
\begin{align}
\begin{split}\label{eq:gamma-hat-angle_n}
     & \frac{\langle \hat{\gamma}_n'(\xi_n), \hat{\gamma}_n'(\xi_n^*) \rangle}{|\hat{\gamma}_n'(\xi_n)| |\hat{\gamma}_n'(\xi_n^*)|}
     = \sqrt{1-q_n^2 \sin^2\xi_n} \sqrt{1-q_n^2 \sin^2 \xi_n^* } 
     \left\langle \begin{pmatrix}
         \frac{1-2q_n^2\sin^2\xi_n}{\sqrt{1-q_n^2\sin^2\xi_n}} \\ -2 q_n \sin\xi_n
      \end{pmatrix} , \begin{pmatrix}
         \frac{1-2q_n^2\sin^2\xi_n^*}{\sqrt{1-q_n^2\sin^2\xi_n^*}} \\ -2 q_n \sin\xi_n^*
      \end{pmatrix}  \right\rangle\\ 
      & = \sqrt{1-q_n^2 \sin^2 \xi_n} \sqrt{1-q_n^2 \sin^2\xi_n^*} \left[ \frac{(1-2q_n^2 \sin^2 \xi_n)(1-2q_n^2\sin^2 \xi_n^* )}{\sqrt{1-q_n^2 \sin^2 \xi_n} \sqrt{1-q_n^2 \sin^2 \xi_n^* }}+  4q_n^2 \sin \xi_n\sin \xi_n^* \right] \\ 
      & = (1-2q_n^2 \sin^2\xi_n)(1-2q_n^2\sin^2 \xi_n^* ) + 4q_n^2 \sin \xi_n \sin \xi_n^* \sqrt{1-q_n^2 \sin^2 \xi_n} \sqrt{1-q_n^2 \sin^2 \xi_n^* }. 
\end{split}
\end{align}
Since $\frac{1}{\sqrt{2}} \leq q_n < 1$, 
there exist $q_\infty \in [\frac{1}{\sqrt{2}},1]$ and a subsequence  $\{n_j\}_{j\in \mathbf{N}}$ such that $q_{n_j} \to q_\infty$ ($j\to \infty$). 
Then, combining \eqref{eq:gamma-hat-angle_n} with the assumption that $-\frac{\pi}{2} \leq \xi^*_n \leq \xi_n \to-\frac{\pi}{2}$ ($n\to \infty$), we see that 
\[
\lim_{j\to\infty} \cos\big(\theta(l_{n_j})-\theta(l^*_{n_j})\big) 
= (1-2q_\infty^2)(1-2q_\infty^2) + 4q_\infty^2  \sqrt{1-q_\infty^2} \sqrt{1-q_\infty^2 }
=1, 
\]
which contradicts \eqref{theta_n-uniform-estimate}.
Thus $\xi_n$ does not converge to $-\frac{\pi}{2}$ and this together with \eqref{eq:asymptotic-xi_n} implies that 
\begin{align} \label{eq:asymptotic-xi_n-improved}
\xi_n \to \infty \quad \text{or} \quad \xi_n \to \frac{\pi}{2} +m\pi \ \ \text{for some}\ \  m\in \mathbf{N}\cup\{0\}. 
\end{align}

\textbf{Step~2} (\textsl{Estimate parameters $\alpha_n$ and $q_n$})\textbf{.}
Thanks to the previous step, we see that $\xi_n\geq \frac{\pi}{4}$ for sufficiently large $n$. 
Combining \eqref{eq:B-gamma-n} with the fact that  $q_n^2 \geq \frac{1}{2}$, we obtain
\[
\mathcal{E}_{\lambda_n}[\gamma_n] \geq \alpha_n \int_{-\frac{\pi}{2}}^{\xi_n} \frac{4q_n^2 \cos^2\xi}{\sqrt{1-q_n^2 \sin^2\xi}} \; \mathrm{d}\xi 
\geq 2\alpha_n \int_{-\frac{\pi}{2}}^{\xi_n}  \cos^2\xi  \; \mathrm{d}\xi
\geq 2\alpha_n \int_{-\frac{\pi}{2}}^{0}  \cos^2\xi  \; \mathrm{d}\xi \geq \alpha_n\frac{\pi}{2}
\]
for large enough $n\in\mathbf{N}$. 
This together with \eqref{eq:EnergyTendsTOzero} implies that $\alpha_n \to 0$ ($n\to\infty$). 
In addition, up to a subsequence, $q_n$ converges to some $q_\infty \in [\frac{1}{\sqrt{2}}, 1]$. 
We here claim the following dichotomy: 
\begin{align}\label{eq:revision-2}
    \text{$q_\infty$ \ is either \ $q_*$ or \ $1$}, 
\end{align}
where $q_*$ is a unique zero of $q\mapsto |2\mathrm{E}(q)-\mathrm{K}(q)|$. 
To this end, let us introduce $X(\xi,q):= 2 \mathrm{E}(\xi,q) - \mathrm{F}(\xi,q)$. 
In view of \eqref{eq:gamma-hat} and \eqref{eq:250204-1} we have
\begin{equation}\label{eq:revision-1}
    \alpha_n \sqrt{a_n^2 + \psi(a_n)^2} = \alpha_n |\gamma_n(x_n)|  = \alpha_n |\hat{\gamma}_n(\xi_n)- \hat{\gamma}_n(-\tfrac{\pi}{2})| \geq |X(\xi_n,q_n) -X(-\tfrac{\pi}{2},q_n)|.
\end{equation}
Suppose that $q_\infty<1$. 
In the case of $\xi_n\to\frac{\pi}{2}+m\pi$, noting that $|X(-\tfrac{\pi}{2}+m\pi,q_n) -X(-\frac{\pi}{2},q_n)|=2(m+1)|2 \mathrm{E}(q_n) - \mathrm{K}(q_n)|$, we see that the right-hand side of \eqref{eq:revision-1} converges to $2(m+1)|2 \mathrm{E}(q_\infty) - \mathrm{K}(q_\infty)|$. 
On the other hand, combining $\alpha_n\to0$ with the fact that $\sqrt{a_n^2+\psi(a_n)^2}$ is bounded, the left-hand side of \eqref{eq:revision-1} converges to $0$. 
Thus we have $q_\infty=q_*$. 
In the case of $\xi_n\to\infty$, the right-hand side of \eqref{eq:revision-1} diverges if $q_\infty \in (0,1)\setminus\{q_*\}$. 
Consequently, under the assumption $q_\infty<1$, we must have $q_\infty=q_*$, which yields 
\eqref{eq:revision-2}

\textbf{Step~3} (\textsl{Energy estimate})\textbf{.}
Now we deduce an improved estimate of $\mathcal{E}_{\lambda_n}[\gamma_n]$. 
By symmetry we have to count the energy of $\tilde{\gamma}_n|_{[0, l_n]}$ twice, so that we obtain  
\begin{equation}
    \mathcal{E}_{\lambda_n}[\gamma_n] = B[\gamma_n]+ \lambda_n L[\gamma_n] 
    \geq 2B[\tilde{\gamma}|_{[0, l_n]}] + 2\lambda_n L[\tilde{\gamma}|_{[0, l_n]}].
\end{equation}
Similar to \eqref{eq:B-gamma-n} we obtain 
\begin{align}
\begin{split}
	 2B[\tilde{\gamma}|_{[0, l_n]}] =  2\alpha_n \int_{-\frac{\pi}{2}}^{\xi_n} \frac{4q_n^2 \cos^2\xi}{\sqrt{1-q_n^2 \sin^2\xi}} \; \mathrm{d}\xi \geq 2\alpha_n \int_{-\frac{\pi}{2}}^{0} \frac{4q_n^2 \cos^2\xi}{\sqrt{1-q_n^2 \sin^2\xi}} \; \mathrm{d}\xi  = \alpha_n \int_{-\frac{\pi}{2}}^{\frac{\pi}{2}} \frac{4q_n^2 \cos^2\xi}{\sqrt{1-q_n^2 \sin^2\xi}} \; \mathrm{d}\xi. 
\end{split}
\end{align}
On the other hand, by the fact that $\lambda_n=2\alpha_n^2(2q_n^2-1)$ and by the change of variables $\xi=\am(\alpha_n s-\mathrm{K}(q_n),q_n)$, we have 
\begin{equation}\label{eq:comute-L-gamma_n}
2\lambda_n L[\tilde{\gamma}|_{[0, l_n]}] = 2\lambda_n \int_0^{l_n}\; \mathrm{d}s = 4 \alpha_n (2q_n^2-1) \int_{-\frac{\pi}{2}}^{\xi_n} \frac{1}{\sqrt{1-q_n^2 \sin^2\xi}} \; \mathrm{d}\xi.
\end{equation}
Since $\xi_n\geq \frac{\pi}{4}$ holds for sufficiently large $n\in\mathbf{N}$, it follows that 
\begin{align}
	2\lambda_n L[\tilde{\gamma}|_{[0, l_n]}] & \geq 4 \alpha_n (2q_n^2-1) \int_{-\frac{\pi}{2}}^{\frac{\pi}{4}} \frac{1}{\sqrt{1-q_n^2 \sin^2\xi}}\; \mathrm{d}\xi \\
	& = 4 \alpha_n (2q_n^2-1) \mathrm{K}(q_n) +   4 \alpha_n (2q_n^2 - 1) \int_0^{\frac{\pi}{4}} \frac{1}{\sqrt{1-q_n^2 \sin^2\xi}} \; \mathrm{d}\xi  \\ & \geq  4 \alpha_n (2q_n^2-1) \mathrm{K}(q_n) +   \pi \alpha_n (2q_n^2 - 1).
\end{align}
The previous computations therefore yield 
\begin{equation}\label{eq:lowerbound-touching}
    \mathcal{E}_{\lambda_n}[\gamma_n] \geq \alpha_n \int_{-\frac{\pi}{2}}^{\frac{\pi}{2}} \frac{4q_n^2 \cos^2\xi}{\sqrt{1-q_n^2 \sin^2\xi}} \; \mathrm{d}\xi + 4 \alpha_n (2q_n^2-1) \mathrm{K}(q_n) +   \pi \alpha_n (2q_n^2 - 1).
\end{equation}

\textbf{Step~4}  (\textsl{Construction of a comparison curve})\textbf{.}
Let $\hat{\gamma}_*:[-\pi/2, \pi/2] \to \mathbf{R}^2$ be a curve defined by 
\[
\hat{\gamma}_*(\xi):=
\begin{pmatrix}
    2\mathrm{E}(\xi, q_*)-\mathrm{F}(\xi, q_*) \\
    2q_*\cos{\xi}
\end{pmatrix}.
\]
Note that $\hat{\gamma}_*$ parametrizes a 
 half-fold figure-eight elastica, see \eqref{eq:figure-8} with $n=1$.
Using this $\hat{\gamma}_*$, for each $n\in \mathbf{N}$ we define $c_n:[-\pi/2, \alpha_n+\pi/2]$ by 
\begin{align}\label{eq:def-c_n}
c_n (\xi):= 
\begin{cases}
\frac{1}{\alpha_n}\hat{\gamma}_*(\xi) & \xi \in[-\frac{\pi}{2},0], \\
\frac{1}{\alpha_n}(\xi, 2q_*)  & \xi \in [0, \alpha_n], \\
\frac{1}{\alpha_n}\hat{\gamma}_*(\xi-\alpha_n) +(1,0) \ \ & \xi \in [\alpha_n, \alpha_n+\frac{\pi}{2}]. 
\end{cases}
\end{align}
Now we show that $c_n \in A_{\rm sym}^\lozenge$ for all sufficiently large $n\in \mathbf{N}$, after suitable reparametrization. 
First, the reflectional symmetry immediately follows by construction.
Noting that 
\[
\lim_{\xi\uparrow0}c_n'(\xi)=\lim_{\xi\downarrow0}c_n'(\xi)=\lim_{\xi\uparrow\alpha_n}c_n'(\xi)=\lim_{\xi\downarrow\alpha_n}c_n'(\xi)= (\tfrac{1}{\alpha_n},0), 
\]
we see that $c_n\in C^1([-\pi/2, \alpha_n+\pi/2];\mathbf{R}^2)$. 
In addition, $c_n$ is
also of class $H^2$ since each component is of class $H^2$ and the values and tangents coincide at the gluing points.
Furthermore, by the fact that $\frac{2q_n}{\alpha_n} \geq \frac{\sqrt{2}}{\alpha_n} > \frac{1}{\alpha_n} \geq \psi(\frac{1}{2})$ for large enough $n$, we see that for all (large) $n\in \mathbf{N}$
\[
c_n^2(\xi)=\frac{2q_n}{\alpha_n}  > \psi(\tfrac{1}{2}) = \max_{z \in \mathbf{R}} \psi(z)  \geq  
    \psi(c_n^1(\xi)) \quad \text{for all}\ \  \xi \in [0, \alpha_n].
\]
Moreover, one readily checks that for $\xi \in [-{\pi}/{2},0]$ there holds $c_n^1(\xi) \leq 0$, which is why $c_n^2(\xi)\geq 0 \geq \psi(c_n^1(\xi))$ for all $\xi \in [-{\pi}/{2},0]$.
From the fact that for $\xi \in [\alpha_n, {\pi}/{2}+\alpha_n]$ one has $c_n^1(\xi)\geq1$, similarly we obtain $c_n^2(\xi)\geq 0 \geq \psi(c_n^1(\xi))$ for all $\xi \in [\alpha_n, {\pi}/{2}+\alpha_n]$. 
Thus we see that for all (large) $n\in \mathbf{N}$ $c_n \in A_{\rm sym} \subset A_{\rm sym}^\lozenge$ (up to reparametrization). 
Now we compute the energy of $c_n$. 
Similar to \eqref{eq:calc-B-xi} we can compute the bending energy of $c_n|_{[-\frac{\pi}{2},0]}=\hat{\gamma}_*|_{[-\frac{\pi}{2},0]}$ so that
\begin{equation}
    B[c_n] = \alpha_n\int_{-\frac{\pi}{2}}^\frac{\pi}{2} \frac{4q_*^2\cos^2\xi}{\sqrt{1-q_*^2\sin^2\xi}} \; \mathrm{d}\xi. 
\end{equation}
By definition we have 
\begin{align}
    \lambda_n L[c_n] &= \lambda_n \big( L[c_n|_{[-\frac{\pi}{2},0]}] + L[c_n|_{[0,\alpha_n]}] + L[c_n|_{[\alpha_n, \alpha_n+\frac{\pi}{2}]}] \big) \\
    &= \lambda_n \Big( \frac{1}{\alpha_n} \mathrm{K}(q_*) + 1 + \frac{1}{\alpha_n} \mathrm{K}(q_*) \Big) =  4\alpha_n(2q_n^2-1)\mathrm{K}(q_*)+2\alpha_n^2(2q_n^2-1), 
\end{align}
where in the last equality we used $\lambda_n=2\alpha_n^2(2q_n^2-1)$.  
Thus it follows that 
\begin{align}\label{eq:energy-c_n}
   \mathcal{E}_{\lambda_n}[c_n]  & = \alpha_n\int_{-\frac{\pi}{2}}^\frac{\pi}{2} \frac{4q_*^2\cos^2\xi}{\sqrt{1-q_*^2\sin^2\xi}} \; \mathrm{d}\xi + 4\alpha_n(2q_n^2-1)\mathrm{K}(q_*) +2\alpha_n^2(2q_n^2-1). 
\end{align}

\textbf{Step~5} (\textsl{Conclusion})\textbf{.}
Recall from Step 1 
and Step 3
that $\gamma_n$ 
is a minimizer in $A_{\rm sym}^\lozenge$ and $c_n\in A_{\rm sym}^\lozenge$. 
Therefore
\[ \mathcal{E}_{\lambda_n}[c_n]\geq \inf_{\gamma\in A_{\rm sym}^\lozenge}\mathcal{E}_{\lambda_n}[\gamma] =\mathcal{E}_{\lambda_n}[\gamma_n]. 
\]
These together with \eqref{eq:lowerbound-touching} and \eqref{eq:energy-c_n} imply that
\begin{align}
0 \geq \mathcal{E}_{\lambda_n}[\gamma_n] - \mathcal{E}_{\lambda_n}[c_n] 
&\geq \alpha_n \int_{-\frac{\pi}{2}}^\frac{\pi}{2} \bigg(\frac{4q_n^2\cos^2\xi}{\sqrt{1-q_n^2\sin^2\xi}} -\frac{4q_*^2\cos^2\xi}{\sqrt{1-q_*^2\sin^2\xi}} \bigg)\; \mathrm{d}\xi \\
&\quad+\alpha_n(\pi-2\alpha_n)(2q_n^2-1) + 4\alpha_n(2q_n^2-1)\big(\mathrm{K}(q_n)-\mathrm{K}(q_*) \big).
\end{align}
Dividing this inequality by $\alpha_n>0$, we obtain
\begin{align}\label{eq:revision-3}
0\geq  \int_{-\frac{\pi}{2}}^\frac{\pi}{2} \bigg(\frac{4q_n^2\cos^2\xi}{\sqrt{1-q_n^2\sin^2\xi}} -\frac{4q_*^2\cos^2\xi}{\sqrt{1-q_*^2\sin^2\xi}} \bigg)\; \mathrm{d}\xi +(\pi-2\alpha_n)(2q_n^2-1) + 4(2q_n^2-1)\big(\mathrm{K}(q_n)-\mathrm{K}(q_*) \big). 
\end{align}
Here recall from \eqref{eq:revision-2} that $q_n$ tends to $q_*$ or $1$. 
If $q_n\to q_*$, then we find that \eqref{eq:revision-3} of the right-hand side tends to a strictly positive value $\pi(2q_*^2-1)$. 
If $q_n\to 1$, then \eqref{eq:revision-3} of the right-hand side diverges as $\mathrm{K}(1-)=\infty$. 
A contradiction. 
We have now shown that minimizers do not touch the obstacle. By Remark~\ref{rem:minimizer-touching_rhomb} they must be given by a penalized pinned elastica, all of which are characterized by the prototypes in Proposition~\ref{prop:PPE-classification}.
Since a trivial line segment cannot lie in $A_{\rm sym}^\lozenge$ by (A3) but $\gamma_{\rm larc}^{\lambda,1,1}$ is admissible by \eqref{eq:height-nontouching}, by using minimality of $\gamma_{\rm larc}^{\lambda,1,1}$ (cf.\ Remark~\ref{rem:energy_comparison_PPE}) we conclude that (up to reflection) $\gamma_{\rm larc}^{\lambda,1,1}$ is the only minimizer in $A_{\rm sym}^\lozenge$. 
\end{proof}

Now we apply the non-touching result for $A_{\rm sym}^\lozenge$ to the original problem.

\begin{proof}[Proof of Theorem~\ref{thm:hitting-escaping}]
Let $\bar{\lambda}$ be as in Lemma \ref{lemma:rhombproblemnontouhing}. Fix $\lambda \leq \bar{\lambda}$ and assume for a contradiction that some minimizer $\gamma_\lambda$ of $\mathcal{E}_\lambda$ in $A_{\rm sym}$ is touching, i.e.\ $I_{\gamma_\lambda} \neq \emptyset$. Notice that $\gamma^{\lambda,1,1}_{\rm larc}$ also lies in $A_{\rm sym}$, so we infer 
\begin{equation}\label{eq:contradictonmsymandmsymrhomb}
    \mathcal{E}_\lambda[\gamma_\lambda] \leq \mathcal{E}_\lambda[\gamma_{\rm larc}^{\lambda,1,1}].
\end{equation}
However, under assumption \eqref{eq:height-nontouching}, $\gamma_\lambda$ and $\gamma_{\rm larc}^{\lambda,1,1}$ both also lie in $A_{\rm sym}^\lozenge$ and according to Lemma \ref{lemma:rhombproblemnontouhing}
 $\gamma_{\rm larc}^{\lambda,1,1}$ is (up to reflection) the unique minimizer of $\mathcal{E}_\lambda$ in $A_{\rm sym}^\lozenge$. In particular one infers 
$$ \mathcal{E}_\lambda[\gamma_{\rm larc}^{\lambda,1,1}] \leq \mathcal{E}_\lambda[\gamma_\lambda].$$
 This and \eqref{eq:contradictonmsymandmsymrhomb} yield that $\mathcal{E}_\lambda[\gamma_{\rm larc}^{\lambda,1,1}] = \mathcal{E}_\lambda[\gamma_\lambda]$. In particular, $\gamma_\lambda$ is another minimizer of $\mathcal{E}_\lambda$ in $A_{\rm sym}^\lozenge$. According to Lemma \ref{lemma:rhombproblemnontouhing}, $\gamma_\lambda$ must  coincide with $\gamma_{\rm larc}^{\lambda,1,1}$ or its reflection at the $x$-axis. However, $\gamma_{\rm larc}^{\lambda,1,1}$ does not touch the obstacle, so $\gamma_\lambda \neq \gamma_{\rm larc}^{\lambda,1,1}$. The minimizer $\gamma_\lambda$ can also not coincide with the reflection of $\gamma_{\rm larc}^{\lambda,1,1}$ as this reflection does not lie in $A_{\rm sym}$.  A contradiction.
\end{proof}

\section{Touching the obstacle for large $\lambda$}\label{sect:6}

In the previous section we have learned that for small values of $\lambda> 0$  one can in general have nontouching minimizers. The range of values $\lambda$ for which the nontouching result can be obtained depends in general on $\psi$, cf. Theorem~\ref{thm:hitting-escaping}. The goal of this section is to show that for suitably large values of $\lambda$  the obstacle must be touched by each minimizer. A crucial difference is that in this section we 
obtain a range of values of $\lambda$ that does not depend on $\psi$ (as long as $\psi$ satisfies assumptions (A1)--(A3)). The key ingredient is the stability analysis for critical points of $\mathcal{E}_\lambda$ in $A_{\rm pin}$ performed in  \cite{MYarXiv2409}.  Recall the definition of the value $\hat{\lambda}$ in \eqref{eq:def-hat_lambda}. This will exactly coincide with the threshold value given in Theorem  \ref{thm:touching_large_lambda}, whose proof we present now. 
\begin{proof}[Proof of Theorem \ref{thm:touching_large_lambda}.]
   Let $\lambda > \hat{\lambda}$ and let $\psi: \mathbf{R} \rightarrow \mathbf{R}$ be an arbitrary obstacle  satisfying (A1)--(A3). Suppose that $\gamma \in A$ is a minimizer of $\mathcal{E}_\lambda$ in $A$. We assume for a contradiction that $I_\gamma= \emptyset$. This means in particular that $\gamma^2 > \psi \circ \gamma^1$. Define
    \begin{equation}\label{eq:deltadefinitio}
        \delta := \min_{x\in [0,1]} \big(\gamma^2(x) - \psi (\gamma^1(x)) \big) > 0.
    \end{equation}    
    According to Lemma \ref{lem:E-L_eq_on_noncoincidence}, $\gamma$ must be a penalized pinned elastica with parameter $\lambda$.  In particular, up to reparametrization and reflection, $\gamma$ lies in one of the following three classes (cf. Proposition \ref{prop:PPE-classification})
    \begin{itemize}
        \item[(C1)] $\gamma = \gamma_{\rm sarc}^{\lambda, 1,n}$ for some $n \geq n_\lambda,$ 
       \item[(C2)] $\gamma = \gamma_{\rm larc}^{\lambda, 1,n}$ for some $n \geq n_\lambda,$
       \item[(C3)] $\gamma = \gamma_{\rm loop}^{\lambda, 1,n}$ for some $n \geq 1.$
    \end{itemize} 
   First we assume that $\gamma$ lies in class either (C1) or (C2).  In this case one has 
   \begin{equation}
       n \geq n_\lambda = \left\lceil \tfrac{ \lambda}{\hat{\lambda}}\right\rceil \geq 2.
   \end{equation}
   Due to this fact, all the curves listed in (C1) and (C2) are not local minimizers in $A_{\rm pin}$, see \cite{MYarXiv2409}*{Theorem 3.7 and Theorem 3.8}. 
   This implies that there exists a sequence $\{\gamma_j\}_{j \in \mathbf{N}} \subset A_{\rm pin}$ converging to $\gamma$ in $H^2(0,1)$ such that $\mathcal{E}_\lambda[\gamma_j] < \mathcal{E}_\lambda[\gamma]$ for all $j \in \mathbf{N}$. Since $H^2(0,1)$ embeds into $C^0([0,1])$ this convergence is also uniform. In particular (due to uniform continuity of $\psi$) we deduce from \eqref{eq:deltadefinitio} that for $j$ sufficiently large there holds for any $x\in[0,1]$
\begin{equation}
    \gamma_j^2(x) - \psi \circ \gamma_j^1(x) > \frac{\delta}{2} >0.
\end{equation}
 Hence we conclude that $\gamma_j \in A$ for $j$ sufficiently large. This contradicts the minimality of $\gamma$ in $A$. Next assume that  $\gamma$ lies in class (C3). By \cite{MYarXiv2409}*{Theorem 3.6}, $\gamma$ is unstable in $A_{\rm pin}$, i.e.\ not a local minimizer in $A_{\rm pin}$. A similar argument as above yields that $\gamma$ can not be energy-minimal in $A$. The claim follows. 
\end{proof}

\appendix

\section{Elliptic integrals and functions}\label{sect:elliptic_functions}
In this article we have used the \textit{elliptic integrals}
 \begin{equation}
     \mathrm{F}(x,q) := \int_0^x \frac{1}{\sqrt{1-q^2\sin^2(\theta) }} \; \mathrm{d}\theta \qquad \textrm{and} \qquad \mathrm{E}(x,q) := \int_0^x \sqrt{1-q^2\sin^2(\theta)} \; \mathrm{d}\theta 
 \end{equation}
 for $x \in \mathbf{R}$ and $q \in (0,1)$. 
 Further we define $\mathrm{K}(q) := \mathrm{F}(\frac{\pi}{2},q)$ and $\mathrm{E}(q) := \mathrm{E}(\frac{\pi}{2},q)$.
It is known that $[0,1) \ni q\mapsto \mathrm{K}(q)$ and $[0,1]\ni q \mapsto \mathrm{E}(q)$ are strictly increasing and strictly decreasing, respectively.
More precisely, one has
\begin{align} \label{eq:diff-elliptic-int}
\mathrm{K}'(q)=\frac{\mathrm{E}(q)}{q(1-q^2)} - \frac{\mathrm{K}(q)}{q}>0, \quad \mathrm{E}'(q)= \frac{\mathrm{E}(q)}{q} - \frac{\mathrm{K}(q)}{q}<0.
\end{align}
This monotonicity leads to the following

\begin{lemma}\label{lem:elliptic_2E-K}
The function $Q:[0,1)\ni q\mapsto 2\mathrm{E}(q)-\mathrm{K}(q)$ is strictly decreasing, and satisfies $Q(0)=1$, $\lim_{q\to 1}Q(q)=-\infty$.
In particular, there exists a unique $q_*\in(0,1)$ such that $2\mathrm{E}(q_*)-\mathrm{K}(q_*)=0$.
\end{lemma}

The constant $q_*$ stands for the modulus of the so-called figure-eight elastica; more precisely, the signed curvature of the figure-eight elastica is given by $k(s)=2q_*\cn(s,q_*)$, up to similarity and reparametrization (cf.\ \cite[Definition 5.1]{MR23}). 

Next we recall the \textit{Jacobian elliptic functions}  $\cn,\sn$. 

\begin{definition}[Elliptic functions] \label{def:Jacobi-function}
We define the \textit{Jacobian amplitude function} $\am(x,q)$ by the inverse function of $F(x,q)$, so that
\[
x= \int_0^{\am(x,q)} \frac{1}{ \sqrt{1-q^2\sin^2\theta} }\; \mathrm{d}\theta \quad \text{for}\ x\in\mathbf{R}.
\]
For $q\in[0,1)$, the \textit{Jacobian elliptic functions} are given by
\begin{align}
\cn(x, q):= \cos \am(x,q), \qquad \sn(x, q):= \sin \am(x,q), \qquad x\in \mathbf{R}.
\end{align}
\end{definition}

The Jacobian elliptic functions have the following fundamental properties. 
\begin{proposition} \label{prop:cn}
Let $\cn(\cdot,q)$ and $\sn(\cdot,q)$ be the elliptic functions with modulus $q \in [0,1)$. 
Then, $\cn(\cdot,q)$ is an even $2\mathrm{K}(q)$-antiperiodic function on $\mathbf{R}$ and, in $[0,2\mathrm{K}(q)]$, strictly decreasing from $1$ to $-1$. 
Further, $\sn(\cdot,q)$ is an odd $2\mathrm{K}(q)$-antiperiodic function and in $[-\mathrm{K}(q),\mathrm{K}(q)]$ strictly increasing from $-1$ to $1$.
\end{proposition}

\section{Technical proofs}\label{sect:appendixB}

\begin{proof}[Proof of Proposition~\ref{prop:fig8_minimizes}]
Existence of a global minimizer $\gamma$ of $\mathcal{E}_\lambda$ in $A_{\rm drop}$ follows from the standard direct method. 
In what follows along the argument as in \cite{MYarXiv2409}*{Proof of Theorem 1.1} we show that $\gamma$ is an $\frac{n}{2}$-fold figure-eight elastica.

By minimality of $\gamma$ we find that $\gamma$ is also a critical point of $\mathcal{E}_\lambda = B+\lambda L$ in $A_{\rm drop}$.
Then, from the same argument as in \cite{MYarXiv2409}*{Lemma 2.2}, the signed curvature $k:[0,L]\to\mathbf{R}$ of $\gamma$ satisfies $2k''+k^3-\lambda k=0$ in $(0,L)$ and $k(0)=k(L)=0$. 
In addition, in view of \cite{Lin96}*{Proposition 3.3}, there exist $\sigma \in \{+,-\}$, $\alpha>0$, $q\in[0,1)$, and $s_0\in \mathbf{R}$ such that
\[
k(s)=2\sigma \alpha q \cn(\alpha s+s_0, q), \quad s\in[0,L], 
\]
with the relation that $2\alpha^2(2q^2-1)=\lambda$. 
By the antiperiodicity of $\cn$ (cf.\ \cite{MYarXiv2409}*{Proposition A.3}), we may choose $s_0=-\mathrm{K}(q)$.  
Moreover, noting that $k(L)=0$, we find that $L={2n}\mathrm{K}(q)/{\alpha}$ for some $n\in \mathbf{N}$.
Note that 
\[
\tilde{\gamma}^1(s)=\int_0^s \cos{\theta}(t)\;\mathrm{d}t, \quad \theta(s):=\int_0^s k(t)\;\mathrm{d}t, 
\]
where $\tilde{\gamma}$ denotes the arclength parametrization of $\gamma$. 
By calculations same as in \cite{MYarXiv2409}*{Equation 2.16}, we establish that 
\[
|\gamma^1(1)-\gamma^1(0)| = |\tilde{\gamma}^1(L) - \tilde{\gamma}^1(0)| = \frac{2n}{\alpha}|2\mathrm{E}(q) - \mathrm{K}(q)|.
\]
Since $\gamma$ must satisfy $\gamma^1(0)=\gamma^1(1)$ by the boundary condition, it follows from the above equation that $2\mathrm{E}(q) - \mathrm{K}(q)=0$, i.e., $q=q_*$. 
Furthermore, this together with $2\alpha^2(2q^2-1)=\lambda$ yields $\alpha=\alpha_*$, given by \eqref{eq:figure-8}.
Thus the signed curvature $k$ is given by \eqref{eq:figure-8_curvature} with $L={2n}\mathrm{K}(q_*)/{\alpha_*}$. 
Using the fact that $\tilde{\gamma}(s)=\int_0^s (\cos{\theta(t)},\sin{\theta(t)} )\;\mathrm{d}t$, and by considering the integral we obtain $\tilde{\gamma}$ is given by \eqref{eq:figure-8} with some $n\in\mathbf{N}$, up to reflection, rotation, and translation.

Finally, since a straightforward calculation (cf.\ \cite{MYarXiv2409}*{Section 4}) yields 
\begin{align}
    \mathcal{E}_\lambda[\gamma_{\rm loop}^{\lambda,n}]&= \int_0^L k(s)^2\;\mathrm{d}s + \lambda \frac{2n}{\alpha_*}\mathrm{K}(q_*) 
    = 2\sqrt{2}n\sqrt{\lambda}\frac{(4q_*^2-3)\mathrm{K}(q_*)+2\mathrm{E}(q_*)}{\sqrt{2q_*^2-1}}, 
\end{align}
we see that a $\frac{1}{2}$-fold figure-eight elastica is a global minimizer and it is unique. 
This fact with the choice of $n=1$ in the above equation yields \eqref{eq:drop_minimal}. 
\end{proof}

\begin{proof}[Proof of Lemma \ref{lem:existence-rho}]
Let $\{\gamma_n\}_{n\in \mathbf{N}}\subset A_{\rm sym}^\lozenge$ be a minimizing sequence of $\mathcal{E}_\lambda$ in $A_{\rm sym}^\lozenge$, i.e.,
\begin{align} \label{eq:min_seq-touching}
 \lim_{n\to\infty}\mathcal{E}_\lambda[\gamma_n]=\inf_{\gamma\in A_{\rm sym}^\lozenge}\mathcal{E}_\lambda[\gamma]. 
\end{align}
Up to reparametrization, we may suppose that $\gamma_n$ is of constant speed so that $|\gamma'_n| \equiv L[\gamma_n]$.
It follows from \eqref{eq:min_seq-touching} that there exists $C>0$ such that $\mathcal{E}_\lambda[\gamma_n]\leq C$ for any $n\in \mathbf{N}$, and hence the assumption of constant speed implies that
\begin{align}\label{eq:normalized-E_lambda}
    B[\gamma_n] = \frac{1}{L[\gamma_n]^3} \| \gamma_n'' \|_{L^2}^2.
\end{align}
This yields a uniform estimate of $\| {\gamma}_n'' \|_{L^2}$ due to the fact that $L[\gamma_n]$ is bounded by definition of $\mathcal{E}_\lambda$.
Using  $|\gamma_n'|\equiv L[\gamma_n]$ 
and the boundary condition, we also obtain the bounds on the $W^{1,\infty}$-norm. 
Therefore, $\{\gamma_n\}_{n\in\mathbf{N}}$ is uniformly bounded in $W^{2,2}$, and this implies that there are $\gamma_\infty \in W^{2,2}(0,1;\mathbf{R}^2)$ and a subsequence $\{n_j\}_{j\in\mathbf{N}}$ such that $\gamma_{n_j}$ converges to $\gamma_\infty$ in the senses of $W^{2,2}$-weak  and $C^1$-topology.
Thus the limit curve satisfies $\gamma_\infty(0)=(0,0)$, $\gamma_\infty(1)=(1,0)$, and $|\gamma_\infty^2|\geq \psi \circ \gamma_\infty^1$ in $[0,1]$. In particular $\gamma_\infty \in A_{\rm sym}^\lozenge$.
Finally, similar to \eqref{eq:normalized-E_lambda} we have $\|\gamma_\infty''\|_{L^2}^2 = L[\gamma_\infty]^3B[\gamma_\infty]$, and this with the weak lower semicontinuity for $\|\cdot\|_{L^2}$ yields 
\[
\mathcal{E}_\lambda[\gamma_\infty] \leq \liminf_{n\to\infty}B[\gamma_n] + \lambda \lim_{n\to\infty}L[\gamma_n] \leq \liminf_{n\to\infty}(B[\gamma_n] + \lambda L[\gamma_n]).
\]
Combining this with \eqref{eq:min_seq-touching}, we see that $\gamma_\infty$ is a minimizer of $\mathcal{E}_\lambda$ in $A_{\rm sym}^\lozenge$. 
\end{proof}

\bibliographystyle{abbrv}
\bibliography{ref_Mueller-Yoshizawa}

\end{document}